\documentclass[11pt]{article}
\usepackage{hyperref}
\usepackage[left=3cm,right=3cm,top=3cm,bottom=3cm]{geometry} 
\usepackage{amsmath,amssymb} 
\usepackage{amsmath, amsthm,amssymb,amsfonts,latexsym}
\usepackage{amsthm}
\usepackage{comment}
\usepackage{tikz}
\usepackage{lipsum}
\usepackage{color}
\usepackage[capitalise]{cleveref}
\usetikzlibrary{patterns}
\usetikzlibrary{intersections}
\usetikzlibrary{arrows,positioning}
\usepackage{graphicx}
\usepackage{pgfplots}
\pgfplotsset{compat=1.14}
\usetikzlibrary{plotmarks}
  \usepgfplotslibrary{patchplots}
  \usepackage{grffile}
\usepgfplotslibrary{fillbetween}
\usepackage{float}
\usepackage{caption}
\usepackage{subcaption}
\pgfmathdeclarefunction{gauss}{3}{%
  \pgfmathparse{1/(#3*sqrt(2*pi))*exp(-((#1-#2)^2)/(2*#3^2))}%
}
\pgfmathdeclarefunction{cdf}{3}{%
  \pgfmathparse{1/(1+exp(-0.07056*((#1-#2)/#3)^3 - 1.5976*(#1-#2)/#3))}%
}
\pgfmathdeclarefunction{fq}{3}{%
  \pgfmathparse{1/(sqrt(2*pi*#1))*exp(-(sqrt(#1)-#2/#3)^2/2)}%
}
\pgfmathdeclarefunction{fq0}{1}{%
  \pgfmathparse{1/(sqrt(2*pi*#1))*exp(-#1/2))}%
}

\colorlet{mydarkblue}{blue!30!black}
\def\Nn{50}

\newcommand*{\rom}[1]{\expandafter\@slowromancap\romannumeral #1@}

\newcommand{\gettikzxy}[3]{%
  \tikz@scan@one@point\pgfutil@firstofone#1\relax
  \edef#2{\the\pgf@x}%
  \edef#3{\the\pgf@y}%
}

\newcommand{\sA}{\mathcal A}

\newcommand{\sD}{\mathcal D}
\newcommand{\sE}{\mathcal E}

\newcommand{\sH}{\mathcal H}
\newcommand{\sI}{\mathcal I}
\newcommand{\sK}{\mathcal K}
\newcommand{\sL}{\mathcal L}
\newcommand{\sM}{\mathcal M}

\newcommand{\sP}{\mathcal P}

\newcommand{\sR}{\mathcal R}
\newcommand{\sS}{\mathcal S}

\newcommand{\sX}{\mathcal X}
\newcommand{\sY}{\mathcal Y}

\newcommand{\R}{\mathbb R}
\newcommand{\E}{\mathbb E}
\newcommand{\N}{\mathbb N}

\newcommand{\bV}{\mathbb V}

\newcommand{\bU}{\mathbb U}

\newcommand{\Leb}{\mbox{Leb}}
\newcommand{\supp}{\mbox{supp}}
\newcommand{\olx}{\overline{x}}
\newcommand{\oly}{\overline{y}}

\newtheorem{thm}{Theorem}[section]
\newtheorem{prop}[thm]{Proposition}
\newtheorem{eg}[thm]{Example}
\newtheorem{lem}[thm]{Lemma}
\newtheorem{cor}[thm]{Corollary}
\newtheorem{rem}[thm]{Remark}
\newtheorem{sass}[thm]{Standing Assumption}

\newtheorem{defn}[thm]{Definition}

\newcommand{\cblue}{\color{blue}}

\begin{document}
\title{An injective martingale coupling}
\author{
David Hobson\thanks{Department of Statistics, University of Warwick \textit{d.hobson@warwick.ac.uk}} \hspace{10mm}
Dominykas Norgilas\thanks{Department of Mathematics, North Carolina State University \textit{dnorgil@nscu.edu}}
}
\date{\today}
\maketitle

\begin{abstract}
We give an injective martingale coupling; in particular, given measures $\mu$ and $\nu$ in convex order on $\R$ such that $\nu$ is continuous, we construct a martingale transport such that for each $y$ in the support of the target law $\nu$ there is a {\em unique} $x$ in {a support of} the initial law $\mu$ such that (some of) the mass at $x$ is transported to $y$.
Then $\pi$ has disintegration $\pi(dx,dy) = \nu(dy) \delta_{\theta(y)}(dx)$ for some function $\theta$.

More precisely we construct a martingale coupling $\pi$ of the measures $\mu$ and $\nu$ such that there is a set $\Gamma_\mu$ such that $\mu(\Gamma_\mu)=1$ and a disintegration $(\pi_x)_{x \in \Gamma_\mu}$ of $\pi$ of the form $\pi(dx,dy) = \pi_x(dy) \mu(dx)$ such that, with $\Gamma_{\pi_x}$ a support of $\pi_x$, we have $\# \{ x \in \Gamma_\mu  : y \in \Gamma_{\pi_x}  \} \in \{ 0,1 \}$ for all $y$ and $\{ y : \# \{ x \in \Gamma_\mu  : y \in \Gamma_{\pi_x} \} = 1 \} = \supp(\nu)$. 
Moreover, if $\mu$ is continuous we may take $\Gamma_{\pi_x} = \supp(\pi_x)$ for each $x$. However, we cannot also insist that $\Gamma_\mu = \supp (\mu)$.


\end{abstract}

\section{Introduction}

Suppose $\mu$ and $\nu$ are measures on $\R$ in convex order, denoted by $\mu\leq_{cx}\nu$. Consider the set of martingale couplings of $\mu$ and $\nu$, i.e. transports of $\mu$ to $\nu$ which respect the martingale property. By Strassen's Theorem~\cite{Strassen:65}
this set is non-empty.


Ruodu Wang~\cite{Wang:22} asks: suppose $\nu$ is continuous; is there a martingale coupling $\pi$ such that 
for each $y$ in the support of the target law $\nu$ there is a {\em unique} $x$ in the support of the initial law $\mu$ such that mass at $x$ is transported to $y$ under $\pi$. 
We call such a $\pi$ an injective martingale coupling. {(Our terminology is inspired by the following: let $f:\R \to \R$ be a function; then $f$ is injective if for every $y \in \R$, $\# \{x :  f(x)=y \} \leq 1$; if we replace $f$ with the support of $\pi_x$, where $\pi_x$ is the disintegration of $\pi$, then we are asking that for each $y$, $\# \{x : \mbox{$y$ is in the support of $\pi_x$} \} \leq 1$.)}

In fact, there are several versions of this question depending on whether one requires the result to hold for all $y$ or just $\nu$-almost surely all $y$, and depending on what definition of the support of a measure we use for $\mu$, $\nu$ and $\pi_x$.
Our goal in this paper is first to clarify these issues and second to show that the answer (to a fairly strong form of Ruodu Wang's question) is `Yes' by constructing (in an explicit fashion) 
an injective $\pi$.

For a measure $\eta$ on $(\R, {\mathcal B}(\R))$ we say $\Gamma\in {\mathcal B}(\R)$ is {\em a support of $\eta$} if $\eta(\Gamma)=\eta(\R)$. Denote by $\supp(\eta)$ the smallest closed set $\Gamma$ such that $\eta(\Gamma)=\eta(\R)$ (and say $\supp(\eta)$ is {\em the closed support of $\eta$}). We also define {\em the interval support of $\eta$}, denoted by $I_\eta$,  as follows: let $I$ be the smallest interval of the form $I=[\alpha_\eta,\beta_\eta]$ with $-\infty \leq \alpha_\eta \leq \beta_\eta \leq \infty$ such that $I$ is a support of $\eta$; then let $I_\eta$ be the union of the open interval $(\alpha_\eta,\beta_\eta)$ together with any endpoints $\gamma \in \{ \alpha_\eta,\beta_\eta \}$ which are charged by $\eta$. 
Finally, let $\supp_I(\eta) = \supp(\eta) \cap I_\eta$.

Let $\sP$ denote the set of integrable measures on $(\R, {\mathcal B}(\R))$ and let $\sP_p\subset\sP$ denote the set of integrable measures with total mass $p$. 
For $x\in\R$, let $L^1(x)$ denote the set of $\eta \in \sP_1$ such that $\int_{\R}y \eta(dy)  = x$.


Suppose $\mu \leq_{cx} \nu$, i.e., $\int_\R fd\mu\leq\int_\R fd\nu$ for all convex $f:\R\to\R$. Let $\sM(\mu,\nu)$ denote the set of martingale couplings (or transports) of $\mu$ to $\nu$, i.e., the set of probability measures $\pi$ on $\R \times \R$ such that $\pi$ has first marginal $\mu$, second marginal $\nu$, and satisfies the martingale property $\int_{x \in A} (y-x) \pi(dx,dy) = 0$ for all Borel sets $A\subseteq \R$. Equivalently, $\pi$ is the joint law of a pair of random variables $(X,Y)$ such that $\sL(X)=\mu$, $\sL(Y)=\nu$ and $\E^\pi[Y|X]=X$ (where $\sL(Z)$ denotes the law of a random variable $Z$). For $\pi \in \sM(\mu,\nu)$ we can write $\pi$ in terms of its disintegration $\pi(dx,dy) = \mu(dx) \pi_x(dy)$, and then $\pi_x \in L^1(x)$ for $\mu$-a.e. $x\in\R$. {{Conversely, if $\Gamma$ is the set where $\pi_x$ is defined, then we say $(\Gamma, \{ \pi_x \}_{x \in \Gamma})$ defines a martingale coupling of $\mu$ and $\nu$, if $\Gamma$ is a support of $\mu$, if $\pi_x \in L^1(x)$ for each $x \in \Gamma$ and if for all Borel sets $A$, $\int_{x \in \R} \mu(dx) \pi_x(A) = \nu(A)$. }}

We now define weak and strong notions of injectivity.

\begin{defn}
$(\Gamma_\mu,(\pi_x)_{x \in \Gamma_\mu})$ defines
a weakly injective martingale coupling of $\mu$ and $\nu$ if 
$(\Gamma_\mu,(\pi_x)_{x \in \Gamma_\mu})$ defines a martingale coupling of $\mu$ and $\nu$, and there exists a family of supports {$\{\Gamma_{\pi_x}\}_{x \in \Gamma_\mu}$ of $(\pi_x)_{x \in \Gamma_\mu}$ (i.e., for each $x \in \Gamma_\mu$, $\pi_x(\Gamma_{\pi_x})=\pi_x(\R)$)}, and a support $\Gamma_\nu$ of $\nu$ such that
\begin{equation}
\label{eq:introweak}
\begin{cases}
\# \{ x \in \Gamma_\mu : y \in \Gamma_{\pi_x}\} \in \{0,1 \}, \hspace{10mm} \forall y\in\R, \\
\{y : \# \{ x \in \Gamma_\mu : y \in \Gamma_{\pi_x}  \} = 1 \} = \Gamma_\nu .
\end{cases} 
\end{equation}
\end{defn}

\begin{defn}
{{$(\Gamma_\mu, \{ \pi_x \}_{x \in \Gamma_\mu})$ defines
strongly injective martingale coupling of $\mu$ and $\nu$ if 
$(\Gamma_\mu, \{ \pi_x \}_{x \in \Gamma_\mu})$ defines a martingale coupling of $\mu$ and $\nu$ 
and in addition}}
\begin{equation}
\label{eq:introstrong}
\begin{cases}
\# \{ x \in \Gamma_\mu : y \in \supp_I( \pi_x ) 
\in \{0, 1\}, \hspace{10mm} \forall y\in\R, \\
\{y : \# \{ x \in \Gamma_\mu : y \in \supp_I( \pi_x ) 
\} = 1 \} = \supp(\nu) .
\end{cases}
\end{equation}
\end{defn}

Comparing \eqref{eq:introweak} with \eqref{eq:introstrong}, in 
\eqref{eq:introstrong} we insist on particular choices for the support of $\nu$ and $\Gamma_{\pi_x}$; in particular, we assume that $\Gamma_\nu  = \supp(\nu)$, and that $\Gamma_{\pi_x} = \supp_I(\pi_x)$. Note that a strongly injective martingale coupling is automatically a weakly injective martingale coupling.

We call a measure $\eta\in\sP$ \textit{continuous}, if the distribution function $y\to\eta((-\infty,y])$ is continuous. Then our main result is:

\begin{thm}
\label{thm:mainintro}
Suppose $\mu \leq_{cx} \nu$ where $\mu\in\sP_1$ is an arbitrary measure and $\nu\in\sP_1$ is continuous. Then there exists a strongly injective martingale coupling of $\mu$ and $\nu$.
\end{thm}

\begin{rem}
    \label{rem:mainintrocts}
In the case where $\mu$ is also continuous the strongly injective martingale coupling we construct has the property that for all $x\in\Gamma_\mu$, the support of $\pi_x$ is finite and hence $\supp_I(\pi_x)  = \supp(\pi_x)$. It follows that in this setting the definition of a strongly injective martingale coupling may be modified so that the conditions in \eqref{eq:introstrong} read
$\# \{ x \in \Gamma_\mu : y \in \supp( \pi_x ) \} \in \{0, 1\}$ for all $y\in\R$ and 
$\{y : \# \{ x \in \Gamma_\mu : y \in \supp( \pi_x ) \} = 1 \} = \supp(\nu)$, and the conclusion that there exists a strongly injective martingale coupling still holds.
\end{rem}

Whilst finishing the first version of this paper, we became aware of the paper Nutz et al. \cite{Nutz:22}, where (among other things) the authors establish the existence of a weakly injective martingale coupling. (In the language of \cite{Nutz:22}, they find a \textit{backward Monge} martingale coupling, but we think that the term injective is more suggestive.) 
The main insight of \cite{Nutz:22} is that by restricting $\mu$ to a set $A_{{\rm NWZ}}=\{d\mu/(d\mu+d\nu)\geq d\nu/(d\mu+d\nu)\}$, and then by embedding this restriction to $\nu$ via the \textit{shadow measure} $S^\nu(\mu\lvert_{A_{{\rm NWZ}}})$ (see Beiglb\"ock and Juillet \cite{BJ:21}), one can find $\pi^{A_{{\rm NWZ}}} \in \sM(\mu\lvert_{A_{{\rm NWZ}}},S^\nu(\mu\lvert_{A_{{\rm NWZ}}}))$ which is a weakly injective coupling. In particular, the left-curtain coupling of $\mu\lvert_{A_{{\rm NWZ}}}$ and $S^\nu(\mu\lvert_{A_{{\rm NWZ}}})$ is injective. Then one (inductively) repeats this process for the remaining masses $(\mu-\mu\lvert_{A_{{\rm NWZ}}})$ and $(\nu-S^\nu(\mu\lvert_{A_{{\rm NWZ}}}))$. The resulting coupling is baptized the \textit{barcode} backward martingale coupling.

The focus in Nutz et al.~\cite{Nutz:22} is on existence and general properties of (weakly) injective martingale couplings {(and Nutz et al. prove many interesting results which we do not discuss, for example, on (non)-uniqueness, and on the fact that injective couplings are dense in the set of martingale couplings).}
What distinguishes this paper from \cite{Nutz:22}, (apart from the completely different construction of an injective coupling), is our emphasis on describing the injective coupling as fully as possible, and our focus on strongly injective couplings\footnote{The main result of Nutz et al.~\cite[Theorem 2.1]{Nutz:22} is that in the setting of Theorem~\ref{thm:mainintro}, there exists a coupling $\pi \in \sM(\mu,\nu)$ and a Borel function $h:\R \to \R$ such that $\pi(\{(h(y),y):y \in \R\})=1$. Then if $\Gamma_\nu$ is a support of $\nu$, if $\Gamma_\mu \supseteq \{ x : x = h(y) : y \in \Gamma_\nu \}$ and if $(\pi_x)_{x \in \Gamma_\mu}$ is a disintegration of $\pi$ then $(\Gamma_\mu, (\pi_x)_{x \in \Gamma_\mu})$ defines (in our notation) a weakly injective martingale coupling of $\mu$ and $\nu$.}

{Both Nutz et al.~\cite{Nutz:22} and this paper make extensive use of the left-curtain martingale coupling $\pi^{lc}$, introduced by Beiglb\"ock and Juillet \cite{BeiglbockJuillet:16}. However, $\pi^{lc}$ is not injective - typically (when $\nu$ is continuous) for the disintegration $\pi^{lc}_x$ of $\pi^{lc}$ we have $\# \{ x : y \in \supp_I(\pi^{lc}_x) \} \leq 2$ but there is a large set of $y$ such that $y \in \{ x : y \in \supp_I(\pi^{lc}_x) \}$ along with another point.}

{Our proof of Theorem \ref{thm:mainintro} relies on the properties of the left-curtain martingale coupling $\pi^{lc}$ in two ways. First it uses the left-curtain coupling to divide the problem into a countable family of sub-problems, such that for each sub-problem the corresponding pair $(\mu,\nu)$ is still in convex order, but has additional structure. Second, for each sub-problem we use a modified version of the left-curtain coupling to construct an injective coupling. The key idea is to choose a special starting point $x_0 \in \R$ and a (maximal) interval $[x_0,x_1)$ to the right of $x_0$ and to couple $\mu \lvert_{[x_0,x_1)}$ with $S^\nu(\mu \lvert_{[x_0,x_1)})$ using the left-curtain coupling. For a well-chosen interval $[x_0,x_1)$ this coupling is injective, and can be described in semi-explicit terms (given the results of Beiglb\"{o}ck et al.~\cite{BeiglbockHobsonNorgilas:18}) using potentials and tangents. Then in the next step we consider an interval $(x_2,x_0)$ to the left of $x_0$ and embed $\mu \lvert_{(x_2,x_0)}$ in $\nu - S^\nu(\mu \lvert_{[x_0,x_1)})$ via the right-curtain coupling. Again, for an appropriate choice of $x_2$ the construction remains injective, and the resulting disintegration $\pi_x(dy)$ can be calculated semi-explicitly.}

Inductively we will obtain a sequence of points $...<x_{2k}<...<x_2<x_0<x_1<...<x_{2k+1}<...$ such that $\mu\lvert_{[x_{2k},x_{2(k-1)}]}$ and $\mu\lvert_{[x_{2k+1},x_{2k+3}]}$ are embedded using alternate right and left-curtain couplings. In this way we explicitly construct two locally strictly monotonic functions such that the coupling concentrates on the graph of these two functions. Ultimately these functions define our (strongly) injective martingale coupling.

The outline of the paper is as follows. In the next section we 
give several examples, with the aim of motivating our notion of a strongly injective martingale coupling. 
In Section~\ref{sec:reduction} we show how the general problem can be reduced to the `irreducible case' in which $\mu$ is also continuous. In Section~\ref{sec:exampleconstruction} we describe our main construction in a simple setting. The rest of the paper shows how this construction can be defined in the case of general continuous measures.
In Section~\ref{sec:notation}
we introduce some notation and results for convex hulls, 
which is used in Section~\ref{sec:dispersion} to construct a pair of functions which ultimately will form the main part of the construction of an injective martingale coupling. In Section~\ref{sec:construction} we show that an injective martingale coupling exists in a case with certain regularity properties by constructing a martingale coupling. Then in Section~\ref{sec:lc} we show how the left curtain coupling of Beiglb\"{o}ck and Juillet~\cite{BeiglbockJuillet:16} {can be used to divide the problem for general (continuous) measures $\mu$ and $\nu$ into a countable family of sub-problems, each of which satisfies the regularity conditions required for the analysis of Section~\ref{sec:construction}. Finally,
putting it all together we deduce the existence of a strongly injective martingale coupling, for general $\mu$ and continuous $\nu$ in convex order.}

\section{Examples}
\label{sec:examples}
The first example shows that there is no possibility of an injective martingale coupling (except in a few special cases\footnote{One such special case is when $\mu$ is a point mass.}) if $\nu$ has atoms, and justifies the fact that Wang~\cite{Wang:22} only asks for injective couplings in the case where $\nu$ is continuous.

Let $\delta_x$ denote the Dirac point mass at $x\in\R$. 

\begin{eg}
\label{eg:discretenu}
Consider the case where $\mu \sim U[-1,1]$ and $\nu = \frac{1}{2} \delta_{-1} + \frac{1}{2} \delta_{1}$. Then $\sM(\mu,\nu)$ is a singleton. Indeed, there must exist $\Gamma \subseteq [-1,1]$ with {$\mu(\Gamma)=1$} such that for $x \in \Gamma$ we have
$\pi_x(dy) = \frac{(x+1)}{2} \delta_1(dy) + \frac{(1-x)}{2} \delta_{-1}(dy)$. Then, for $y = 1$, $\{ x : y \in \supp (\pi_x) \} \supseteq \Gamma \cap (-1,1]$, and for $y = -1$, $\{ x : y \in \supp (\pi_x) \} \supseteq \Gamma \cap [-1,1)$, and in either case (i.e., for all $y \in \supp(\nu)$),  $\{ x : y \in \supp (\pi_x) \} \supseteq \Gamma \cap (-1,1)$ and has $\mu$-measure equal to 1.
\end{eg}

In the light of Example~\ref{eg:discretenu}, hereafter we assume that $\nu$ is continuous. 

Let $f,h:\R\to\R$ be (Borel measurable and) such that $f(x)\leq x\leq h(x)$ for all $x\in\R$. Then, in the case $f(x)<h(x)$, define the martingale mixture distribution $\pi^{f,h}_x \in L^1(x)$ by
\begin{equation}
\label{eq:pifh}
\pi^{f,h}_x(dy) = \frac{h(x)-x}{h(x)-f(x)} \delta_{f(x)}(dy) + \frac{x-f(x)}{h(x)-f(x)} \delta_{h(x)}(dy),\quad x\in\R.
\end{equation}
If $f(x)<x<h(x)$ then this distribution places mass on two points.

The next example, which as far as we are aware is the earliest example of an injective coupling in the literature, shows that
if there exists an injective coupling then there is no expectation of uniqueness. 

\begin{eg}
\label{eg:HNeuberger}
This example is taken from Hobson and Neuberger~\cite[Section 6.3]{HobsonNeuberger:12}.
Suppose $\mu \sim U[-1,1]$ and $\nu \sim U[-2,2]$. For each $a \in [-1,1]$ there exists a pair of monotonic, strictly-increasing, surjective functions $f^a_{HN}:[-1,1] \mapsto [-2,a]$ and $h^a_{HN}:[-1,1] \mapsto[a,2]$
such that\footnote{The pair of conditions in \eqref{eq:nonuniqueness} are exactly the conditions required to ensure that the initial law mass in $[-1,x)$ maps onto the target law mass in $[-2,f^a_{HN}(x)) \cup [a,h^{a}_{HN}(x))$ in a way which preserves both mass and mean.
Solving \eqref{eq:nonuniqueness} explicitly, for each $a\in[-1,1]$ we obtain
\begin{align*}
h^a_{HN}(x)=\frac{2x+a+\sqrt{4+a^2-4ax}}{2}=2x+a-f^a_{HN}(x),\quad x\in[-1,1].
\end{align*}
Note that when $a=0$ we find that $f^a_{HN}(x) = x-1$ and $h^a_{HN}(x)=x+1$. See Example~\ref{eg:open}.} for $i=0,1$,
\begin{equation}
\label{eq:nonuniqueness}
\int_{-1}^x  z^i \frac{dz}{2} =  \int_{-2}^{f^a_{HN}(x)}  z^i \frac{dz}{4} +  \int_{a}^{h^a_{HN}(x)}  z^i \frac{dz}{4},\quad x\in\R.
\end{equation}
Set $\pi^{f^a_{HN},h^{a}_{HN}}(dx,dy) = \mu(dx)\pi^{f^{a}_{HN}(x),h^a_{HN}(x)}_x(dy)$ where $\pi^{f^{a}_{HN}(x),h^a_{HN}(x)}_x$ is as defined in \eqref{eq:pifh}.
Then $\pi^{f^a_{HN},h^{a}_{HN}} \in \sM(U[-1,1],U[-2,2])$.

For $y \in (a,2]$ we have $\{ x : y \in \supp( \pi_x ) \} = \{ x: y=h^a_{HN}(x) \} = \{ (h^a_{HN})^{-1}(y)\}$ and for $y\in [-2,a)$, $\{ x : y \in \supp( \pi_x ) \}= \{ (f^a_{HN})^{-1}(y)\}$.

Take $\Gamma_\mu = (-1,1)$ and define $\Gamma_\nu = (-2,2) \setminus \{ a \}$. For $x \in \Gamma_\mu$ define $\Gamma_{\pi_x} = \{ f(x),h(x) \} = \supp ( \pi_x) = \supp_I (\pi_x)$.
Then $\Gamma_\mu$, $\Gamma_\nu$ and $\Gamma_{\pi_x}$ are supports of $\mu$, $\nu$ and $\pi_x$ respectively. Then $(\Gamma_\mu,(\pi_x)_{x \in \Gamma_\mu})$ defines a weakly injective martingale coupling of $\mu$ and $\nu$. It is not strongly injective because $\{a \} \notin \{ y : \# \{ x: y \in \supp_I(\pi_x) \} = 1 \}$. However, we show in Example~\ref{eg:open} how the construction may be modified to give a strongly injective martingale coupling.
\end{eg}

The martingale coupling in Hobson and Neuberger~\cite{HobsonNeuberger:12} was created to have other properties and not designed to be an injective coupling, and it seems difficult to extend the construction to the general case whilst maintaining the injectivity property. Nonetheless the example is instructive in describing some of the issues which arise in defining strongly injective martingale couplings. 

The next example illustrates why in the definition of strongly injective we cannot expect to take as the support of the initial law the smallest closed set with full mass. Instead, we must allow ourselves some extra flexibility. 
The example shows that, with this extra flexibility the construction in Example~\ref{eg:HNeuberger} can be modified to give a strongly injective martingale coupling.
Although the ideas work in general for general $a \in [-1,1]$ we focus on the case $a=0$.

\begin{eg}
\label{eg:open}
Continuing Example~\ref{eg:HNeuberger}, suppose we take $a=0$ and for $x \in [-1,1]$ define $\pi_x(dy) = \frac{1}{2} \delta_{x-1}(dy) + \frac{1}{2} \delta_{x+1}(dy)$. Then $\supp(\pi_x) = \supp_I(\pi_x) = \{x-1,x+1\}$.

Let $\Gamma_\mu = [-1,1]$. Then $\Gamma_\mu = \supp(\mu)$. We have $\{ x \in \Gamma_\mu : 0 \in \supp( \pi_x ) \} = \{-1,+1\}$ and the injectivity property is lost.

Conversely, let $\Gamma_\mu=(-1,1)$. Then $\Gamma_\mu$ is a support of $\mu$. For $y \in \supp(\nu) \setminus \{ -2,0,2 \} = (-2,0) \cup (0,2)$,
we have $\# \{ x \in \Gamma_\mu : y \in \supp( \pi_x ) \} = 1$, but for $y\in \{-2,0,2\}$ {(and $y \in \R \setminus \supp(\nu)$)}, $\# \{ x \in \Gamma_\mu : y \in \supp( \pi_x ) \} = 0$. In particular, we have $\# \{ x \in \Gamma_\mu : y \in \supp( \pi_x ) \} \leq 1$ for all $y\in\R$, but there is strict inequality for some $y \in \supp(\nu)$.


Nonetheless, in this example it is possible to modify the construction to obtain a strongly injective martingale coupling using an ad-hoc method.
Let $\Gamma_\mu=(-1,1] \cup \{ - 2 \}$. For $x \in (-1,1]$ define $\pi_x$ as before and define $\pi_{-2} = \delta_{-2}$.
Then
$\# \{ x \in \Gamma_\mu : y \in \supp( \pi_x ) \} \leq 1$ for all $y\in\R$ and $\{ y :  \# \{ x \in \Gamma_\mu : y \in \supp( \pi_x ) \} = 1 \} = \supp(\nu)$.
Note that for all $x\in\Gamma_\mu$, $\pi_x$ is a discrete measure and $\supp(\pi_x) = \supp_I(\pi_x)$. 
\end{eg}

After the various counterexamples above, the next example is an example of a strongly injective martingale coupling. Indeed, our construction has this example at its core, although in order to work in the general setting the construction needs to be extended in many ways.

\begin{eg}
Suppose $\mu \sim U[-1,1]$ and $\nu \sim U[-2,2]$. Define $f,h:[-1,1] \mapsto \R$ by $f(x) = - \frac{3+x}{2}$ and $h(x)= \frac{1+3x}{2}$. Set $\pi_{-1} = \delta_{-1}$ and for $x \in (-1,1]$ set $\pi_x = \pi^{f,h}_x$ where $\pi^{f,h}_x$ is as defined in \eqref{eq:pifh}.

Define $\Gamma_\mu = \supp(\mu) = [-1,1]$.
Then $(\Gamma_\mu,\{ \pi_x \}_{x \in \Gamma_\mu})$ defines a martingale coupling of $\mu$ and $\nu$ (it is the left-curtain coupling of Beiglbock and Juillet~\cite{BeiglbockJuillet:16}). Moreover, it is a strongly injective coupling.
\end{eg}

One natural approach to the general problem with arbitrary $(\mu,\nu)$ in convex order (with $\nu$ continuous), which we indeed follow in Section~\ref{ssec:irreducible} below, is to decompose it into a countable family of simpler problems, to solve those simpler subproblems, and to construct a solution for the original problem by combining these solutions together. However, there are some issues to be sorted in this approach. Some involve the fact that if $\nu = \sum_{k \geq 1} \nu_k$ then we may have a strict inclusion $\cup_{k \geq 1} \supp(\nu_k) \subset \supp(\nu)$. But even for finite decompositions some thought is needed as the next pair of examples show.
Again, these examples motivate aspects of our definition of a strongly injective martingale coupling.

For $\mu \leq_{cx} \nu$ define 
$D=D_{\mu,\nu}:\R \mapsto \R$ by
\begin{equation}\label{eq:Dsec2}
D_{\mu,\nu}(z) = \int_{\R}(z-x)^+ \nu(dx)-\int_{\R}(z-x)^+ \mu(dx),\quad z\in\R.
\end{equation}
Then, $D_{\mu,\nu} \geq 0$ and $\lim_{z \rightarrow \pm \infty} D_{\mu,\nu}(z) = 0$. 
If $x\in\R$ is such that $D_{\mu,\nu}(x)=0$ then, following Hobson~\cite[page 254]{Hobson:98maxmaxpaper}, for any martingale coupling $\pi \in \sM(\mu,\nu)$ we must have that no mass can cross $x$, i.e., $\pi((-\infty,x] \times(x, \infty)) = 0 = \pi([x,\infty) \times(-\infty,x))$.

\begin{eg}
\label{eg:discreteD=0}    
Suppose $b \geq 0$ and $\mu \sim U \{-1-b, 1+b \}$
and $\nu \sim \frac{1}{2} U[-2-b, -b] + \frac{1}{2} U[b,2+b]$.

Then $\sM(\mu,\nu)$ is a singleton, and (since $D_{\mu,\nu}(0)=0$) mass initially at $-1-b$ (respectively $1+b$) must be transported to locations at or below (respectively at or above) zero. Indeed, $\pi_{-1-b} \sim U[-2-b,-b]$ and $\pi_{1+b} \sim U[b,2+b]$. 
Further, $\supp (\pi_{-1-b}) = [-2-b, -b]$ and $\supp (\pi_{1+b}) = [b,2+b]$. 

Suppose $b>0$. Then for all $y \in supp(\nu)$ we have
$\# \{ x \in \supp(\mu) : y \in \supp(\pi_x) \} = 1$.

More pertinently, now suppose $b=0$. Then $0 \in \supp (\pi_{-1}) \cap \supp (\pi_{1})$. This is one of the reasons why we require $\Gamma_{\pi_x} = \supp(\pi_x) \cap I_{\pi_x} = \supp_I(\pi_x)$ (rather than $\Gamma_{\pi_x} = \supp (\pi_x)$) in the definition of a strongly injective coupling. 

Note that $\supp_I(\pi_{-1})   \cap \supp_I(\pi_{1}) = \emptyset$. If we define $\Gamma_\mu = \{ -2, -1, 0 , 1,2 \}$ and define $\pi_{-2} = \delta_{-2}$, $\pi_0 = \delta_0$ and $\pi_2 = \delta_2$ (and $\pi_{\pm 1}$ as before) then 
$(\Gamma_\mu,(\pi_x)_{x \in \Gamma_\mu})$ defines a strongly injective martingale coupling.
\end{eg}

In Example~\ref{eg:discreteD=0} we consider the case where $\mu$ has a discrete uniform distribution. However,
some of the same issues also arise with continuous initial laws. Most especially, the next example is further evidence that (at least if we want to insist that we use $\supp(\nu)$ as the definition of the support set of the target law which we want to cover exactly once) we must give ourselves some flexibility in defining the support of $\mu$. In the next example we decompose the problem into separate parts. We suppose we can find an injective solution on each part, but when we try to combine them we lose the injectivity property.

\begin{eg}
\label{eg:continuousD=0}
Suppose $a \geq 2$ and suppose $\mu \sim \frac{1}{2}U[-1-a,1-a] +\frac{1}{2}U[-1+a,1+a]$ and $\nu \sim \frac{1}{2}U[-2-a,2-a] + \frac{1}{2}U[-2+a,2+a]$. As in the previous example $D_{\mu,\nu}(0)=0$ and in any martingale transport of $\mu$ to $\nu$ no mass can cross zero.

Let $\pi^{-}$ be a strongly injective coupling of $U[-1-a,1-a]$ and $U[-2-a,2-a]$. Similarly, let $\pi^{+}$ be a strongly injective coupling of $U[-1+a,1+a]$ and $U[-2+a,2+a]$.

If $a > 2$ then $\frac{1}{2} \pi^{-} + \frac{1}{2} \pi^+$ is a strongly injective coupling of $\mu$ and $\nu$.

However, if $a=2$ then $\frac{1}{2} \pi^{-} + \frac{1}{2} \pi^+$ is not injective, instead $\# \{ x : 0 \in \supp(\pi_x) \} = 2$.
\end{eg}

\section{Reductions of the problem}
\label{sec:reduction}

First we rule out a special case in which it is easy to see that a strongly injective martingale coupling exists. 
\begin{cor}\label{cor:mu=nu}
If $\mu=\nu$ then there exists a strongly injective martingale coupling of $\mu$ and $\nu$.
\end{cor}
\begin{proof}
	If $\mu=\nu$ then $D_{\mu,\nu}\equiv0$. Set $\Gamma_\mu = \supp(\nu)$ and for each $x \in \Gamma_\mu$ set $\pi_x = \delta_x$. Then $(\Gamma_\mu,(\pi_x)_{x \in \Gamma_\mu})$ defines a strongly injective coupling of $\mu$ and $\nu$.

	\end{proof}

For the rest of the paper we assume that
$\mu \leq_{cx} \nu$ and $\mu \neq \nu$.
We also assume that $\nu$ is continuous.

\subsection{Reduction to irreducible components}
\label{ssec:irreducible}
Recall the definition of $D_{\mu,\nu}$ and define $\sD^+_{\mu,\nu} = \{ x : D_{\mu,\nu}(x)>0 \}$. Then $\sD^+_{\mu,\nu}$ is a disjoint union of open intervals $\sD^+_{\mu,\nu} = \cup_k I_k$. When constructing a (strongly) injective martingale coupling a natural idea is to construct a coupling on each of these intervals and to obtain the global coupling by superposition. However, as Examples~\ref{eg:discreteD=0} and \ref{eg:continuousD=0} show some care is needed.

\begin{defn}
$(\mu,\nu)$ has a single irreducible component if $\sD_{\mu,\nu}^+$ consists of a single interval.
\end{defn}

Note that if $(\mu,\nu)$ has a single irreducible component then since $\nu$ is continuous $I_\nu = \sD_{\mu,\nu}^+$.

In the case of pairs $(\mu,\nu)$ with a single irreducible component the next definition slightly modifies the notion of a strongly injective martingale coupling, but only at the endpoints. The modified definition will help when we try to `add' solutions to subproblems.

\begin{defn} Suppose $(\mu,\nu)$ has a single irreducible component.

We say $(\Gamma_\mu,(\pi_x)_{x \in \Gamma_\mu})$ defines a strongly injective martingale coupling of $\mu$ and $\nu$ on its irreducible component if $(\Gamma_\mu,(\pi_x)_{x \in \Gamma_\mu})$ defines a martingale coupling of $\mu$ and $\nu$, $\Gamma_\mu \subseteq I_\nu = \sD^+_{\mu,\nu}$, and 
\begin{equation}
\label{eq:singlestrong}
\begin{cases}
\# \{ x \in \Gamma_\mu : y \in \supp_I( \pi_x )  \} \in \{0, 1\}, \hspace{10mm} \forall y, \\
\{y : \# \{ x \in \Gamma_\mu : y \in \supp_I( \pi_x ) \} = 1 \} = \supp_I(\nu)  = \supp(\nu) \cap \sD^+_{\mu,\nu} .
\end{cases}
\end{equation}
\end{defn}

\begin{lem}
\label{lem:opentostrong}
Suppose $(\mu,\nu)$ has a single irreducible component.

Suppose there exists a strongly injective martingale coupling of $\mu$ and $\nu$ on its irreducible component. Then there exists a strongly injective martingale coupling of $\mu$ and $\nu$.
\end{lem}

\begin{proof}
Let $(\Gamma_\mu,(\pi_x)_{x \in \Gamma_\mu})$ define a strongly injective martingale coupling of $\mu$ and $\nu$ on its irreducible component.

Let $\sD^E_{\mu,\nu} = \supp(\nu) \setminus I_\nu$ be the set of finite endpoints of $\sD_{\mu,\nu}^+$. Note that $\sD^E_{\mu,\nu}$ is disjoint from $\Gamma_\mu$.

Define $\tilde{\Gamma}_\mu = \Gamma_\mu \cup \sD^E_{\mu,\nu}$ and for $x \in \tilde{\Gamma}_\mu$ define $\tilde{\pi}_x = \pi_x$ if $x \in \Gamma_\mu$ and $\tilde{\pi}_x = \delta_x$ if $x \in\sD^E_{\mu,\nu}$. Then $(\tilde{\Gamma}_\mu, (\tilde{\pi}_x)_{x \in \tilde{\Gamma}_\mu})$ defines a strongly injective martingale coupling of $\mu$ and $\nu$.
\end{proof}

The next lemma decomposes the general problem into a family of irreducible problems.

\begin{lem}[{Beiglb\"{o}ck and Juillet~\cite[Theorem A.4]{BeiglbockJuillet:16}}]\label{lem:irreducible}
Let $\mu,\nu\in\sP$ with $\mu\leq_{cx}\nu$. 
Write $\sD^+_{\mu,\nu}$ as a union of disjoint open intervals, $\sD^+_{\mu,\nu} = \bigcup_{k\geq 1}I_k$. Let $I_0=\R\setminus\bigcup_{k\geq 1}I_k$. Set $\mu_k=\mu\lvert_{I_k}$ so that $\mu=\sum_{k\geq 0}\mu_k$.

There exists a unique decomposition $\nu=\sum_{k\geq 0}\nu_k$ such that $\mu_0=\nu_0$, $\mu_k\leq_{cx}\nu_k$ and $\{x\in\R:D_{\mu_k,\nu_k}(x)>0\}=I_k$  for each $k\geq 1$. Moreover, any martingale coupling $\pi\in\sM(\mu,\nu)$ admits the unique decomposition $\pi=\sum_{k\geq0}\pi_k$ where $\pi_k\in\sM(\mu_k,\nu_k)$ for all $k\geq 0$, and {$\supp(\pi_0) \subseteq \{(x,y)\in\R^2:x=y, D_{\mu,\nu}(x)=0\}$.}
\end{lem}

Given the decomposition in Lemma~\ref{lem:irreducible}, the next result reduces the problem from one of studing the reducible case to the problem of searching for 
strongly injective martingale coupling of $\mu$ and $\nu$ on irreducible components.

\begin{prop}
\label{prop:irreducible}
Suppose that for every pair $(\mu,\nu)$ (with $\mu \leq_{cx} \nu$ and $\nu$ continuous) such that $(\mu,\nu)$ has a single irreducible component
there exists a strongly injective martingale coupling of $\mu$ and $\nu$ on its irreducible component.

Then, for every  (arbitrary) pair $\mu \leq_{cx} \nu$ such that $\nu$ is continuous there exists a strongly injective martingale coupling of $\mu$ and $\nu$.
\end{prop}

\begin{proof}
Let $\sD^+_{\mu,\nu}$ be the disjoint union $\sD^+_{\mu,\nu} = \cup_k I_k$ and let $(\sum_{k \geq 0} \mu_k, \sum_{k \geq 0} \nu_k)$ be the decomposition which arises in Lemma~\ref{lem:irreducible}. For $k \geq 1$ let $\pi^k$ denote a strongly injective martingale coupling of $\mu_k$ and $\nu_k$ on its irreducible component; it may be written as $(\pi^k_x)_{x \in \Gamma_k}$ where $\Gamma_k \subseteq I_k$. Here $\Gamma_k $ and $ I_k$ are shorthand for $\Gamma_{\mu_k}$ and $I_{\nu_k}$ respectively. Note that the sets $(\Gamma_k)_{k \geq 1}$ are disjoint and $\cup_{k \geq 1} \Gamma_k \subseteq \cup_{k \geq 1} I_k.$

Let $\hat{\Gamma} = \supp(\nu) \setminus \cup_{k \geq 1} I_k$ and let $\Gamma_\mu$ be the disjoint union $\Gamma_\mu = \hat{\Gamma} \cup (\cup_{k \geq 1} \Gamma_k)$. For $x \in \hat{\Gamma}$ let $\pi_x = \delta_x$ and for $x \in \Gamma_k$ let $\pi_x = \pi^k_x$.

Then $(\hat{\Gamma},(\pi_x)_{x \in \hat{\Gamma}})$ defines a strongly injective coupling of $\mu$ and $\nu$. Indeed, if $y \in \supp(\nu)$ then either $y \in I_k$ for some unique $k>0$ or $y \in \hat{\Gamma}$. In the latter case $y \in \{ y \} = \supp(\delta_y) = \supp_I( \pi_y )$. In the former case, there exists (a unique) $x \in \Gamma_k$ such that $y \in \supp_I(\pi^k_x) = \supp_I(\pi_x)$. On the other hand, if $y \notin \supp(\nu)$, then there is no $x \in \Gamma_\mu$ such that $y \in \supp(\pi_x)$ since $\cup_{x \in \Gamma_\mu} \supp(\pi_x) \subseteq \supp(\nu)$. Finally, by considering the separate cases, it is easy to see that, for $x, x' \in \Gamma_\mu$ with $x\neq x'$, we have that $\supp_I(\pi_x) \cap \supp_I(\pi_{x'}) = \emptyset$, so strong injectivity follows.
    
\end{proof}

\subsection{Reduction to the case with no atoms in the initial law}

In this section we show that the general problem can be reduced to the case where the initial law is continuous.

\begin{prop}
    \label{prop:reduction}
Suppose that whenever $\mu \leq_{cx} \nu$ and both $\mu$ and $\nu$ are continuous there exists a strongly injective martingale coupling {$(\Gamma_\mu, (\pi_x)_{x \in \Gamma_\mu})$} of $\mu$ and $\nu$. 

Then for every pair $\mu \leq_{cx} \nu$, with $\mu$ arbitrary and $\nu$ continuous there exists a strongly injective martingale coupling of $\mu$ and $\nu$. In particular, Theorem~\ref{thm:mainintro} holds.
\end{prop}

\begin{proof}
First, decompose $\mu$ into $\mu = \mu^a + \mu^c$ where $\mu^a=\sum_{k =1}^N \alpha_i \delta_{x_i}$ is the atomic part of $\mu$ (enumerated so that $\lim_{i}\sum_{k =1}^{N \wedge i} \alpha_i = \sum_{k =1}^{N} \alpha_i = \mu^a(\R)$, where $0 \leq N \leq \infty$) and $\mu^c$ is continuous.
If $N=0$ there is nothing to prove so suppose $N \geq 1$.

Let $\mu_1 = \alpha_1 \delta_{x_1}$ and $\nu_1 = S^\nu(\mu_1)$ (here $S^\nu(\mu_1)$ is the shadow measure of $\mu_1$ in $\nu$; see Section \ref{sec:measures}) . 
We have $\mu_1 \leq_{cx} \nu_1$ and $\mu - \mu_1 \leq_{cx} \nu - \nu_1$ by the associativity property of the shadow measure (see Beiglb\"{o}ck and Juillet~\cite[Theorem 4.8]{BeiglbockJuillet:16} or Beiglb\"{o}ck et al~\cite[Theorem 4.8]{BeiglbockHobsonNorgilas:18}).

We can repeat the construction. For $2 \leq k < N+1$, let $\mu_k = \sum_{j \geq 1}^k \alpha_j \delta_{x_j}$, let 
$\nu^\Delta_k = S^{\nu - \nu_{k-1}}(\alpha_k \delta_{x_k})$ and let $\nu_k = \nu_{k-1} + \nu^\Delta_k$. Then, again by the associativity property of the shadow measure, 
$\nu_k = S^\nu(\mu_k)$. 
By construction, $\alpha_k \delta_{x_k} \leq_{cx} \nu^\Delta_k$, $\mu_k \leq_{cx} \nu_k$ and $\mu - \mu_k \leq_{cx} \nu - \nu_k$.

Let $\pi^a_{x_k} = \frac{1}{\alpha_k} \nu^\Delta_k$. Then $\pi^a_{x_k}\in L^1(x_k)$.  
We have that the sets $(\supp_I(\nu^\Delta_j))_{1 \leq j < N+ 1}$ (and thus also $(\supp_I(\pi^a_{x_k}))_{1 \leq j < N+ 1}$) are disjoint. Here we utilize the fact that, in the case the target measure $\nu$ is continuous, the shadow measure $S^{\tilde \nu}(\alpha_k\delta_{x_k})$ of $\alpha_k\delta_{x_k}$ in $\tilde\nu \leq\nu$, is in fact a restriction of $\tilde\nu$ to an interval; see Beiglb\"{o}ck and Juillet~\cite[Example 4.7]{BeiglbockJuillet:16}.

Let $\nu_0 = \nu - \sum_{k \geq 1}^N \nu^\Delta_k$. Considering potentials or otherwise, it is easy to see that $\mu^c = \lim_{k} (\mu  - \mu_k) \leq_{cx} \lim_{k} (\nu - \nu_k) = \nu_0$. 

By hypothesis, since $\mu^c$ and $\nu_0 \leq \nu$ are continuous there exists a strongly injective coupling of $\mu^c$ and $\nu_0$ defined by some pair $(\Gamma_0,(\pi^0_x)_{x \in \Gamma_0})$. 
Note that, by the strong injectivity of $(\Gamma_0,(\pi^0_x)_{x \in \Gamma_0})$, we have that $\supp_I(\pi^0_x)\subseteq\supp(\nu_0)\subseteq \supp(\nu)$ for all $x\in\Gamma_0$. Furthermore, if $\mu^c(\R) = \nu_0(\R)=0$ then we can (and do) take $\Gamma_0 = \emptyset$.

Let $E = \left( \cup_{k=1}^{N} \supp_I (\nu^\Delta_k)  \right) \cup \supp(\nu_0)$. Note that  $\cup_{k=1}^{N} (\supp_I (\nu^\Delta_k))$ and $\supp(\nu_0)$ are disjoint and $E \subseteq \supp(\nu)$.

Define $\Gamma_a = \left( \cup_{k=1}^{N} \{ x_k \} \right)$ and $\Gamma_b = \supp (\nu) \setminus E$.
Note that $\Gamma_a$ and $\Gamma_b$ are disjoint since if $x_k \in \Gamma_a$ then there exists an interval $(x_k - \epsilon, x_k + \epsilon)$ such that $\nu=\nu_k$ on this interval, and then $\Gamma_b \cap ( x_k -\epsilon, x_k+\epsilon) = \emptyset$.

To prove the main result, suppose first that $\Gamma_0 \cap (\Gamma_a \cup \Gamma_b) = \emptyset$. 
Define $\Gamma_\mu$ to be the disjoint union $\Gamma_\mu = \Gamma_a \cup \Gamma_0 \cup \Gamma_b$.
For $x \in \Gamma_\mu$ define $\pi_x = \pi^a_{x_k}$ if $x=x_k \in \Gamma_a$;
define $\pi_x = \pi^0_x$ if $x \in \Gamma_0$; otherwise define
$\pi_x = \delta_x$ if $x \in \Gamma_b$. Note that $\supp_I(\pi_x)\subseteq\supp(\nu)$ for all $x\in\Gamma_\mu$.
The claim is that $(\Gamma_\mu, (\pi_x)_{x \in \Gamma_\mu})$ defines a strongly injective martingale coupling of $\mu$ and $\nu$.

First note that $1 \geq \mu(\Gamma_\mu) \geq \mu ( \Gamma_a \cup \Gamma_0) = \sum_{k=1}^N \alpha_k + \mu^c(\R) = 1$ so that $\Gamma_\mu$ is a support of $\mu$ and $\mu(\Gamma_b)=0$.
Since $(\Gamma_0,\pi^0)$ and $( \{x_k \} , \pi^a_{x_k})_{1\leq k < N+1}$ define martingale couplings of $(\mu^a,\nu_0)$ and $(\alpha_k\delta_{x_k},\nu^\Delta_k)_{1 \leq k < N+1}$, respectively, it is easy to see that $(\Gamma_a \cup \Gamma_0, (\pi_x)_{x \in \Gamma_a \cup \Gamma_0})$ defines a martingale coupling of $\mu$ and $\nu$. Adding on $\Gamma_b$ (i.e., considering $(\Gamma_a \cup \Gamma_0 \cup \Gamma_b, (\pi_x)_{x \in \Gamma_a \cup \Gamma_0 \cup \Gamma_b)}$) does not change this property since $\Gamma_b$ has zero $\mu$-measure.


It remains to show that the coupling is strongly injective.
First, fix $y \notin \supp(\nu)$. Then 
$y\notin\supp_I(\pi_x)$ for all $x\in\Gamma_\mu$ and it follows that $\{ x \in \Gamma_{\mu} : y \in \supp_I(\pi_x) \} = \emptyset$.
On the other hand, if $y \in \supp(\nu)$ then either $y \in  \supp_I (\nu_k^{\Delta})$ for some (and then exactly one) $k$, or $y \in \supp(\nu_0)$, or $y \in \supp(\nu) \setminus E$, and exactly one of these three possibilities must happen. If $y \in \supp_I (\nu_k^{ \Delta})$ then $y \notin (\cup_{j \neq k} \supp_I (\nu_j^{ \Delta}) ) \cup \supp (\nu_0) \cup \Gamma_b$. Then $\{ x \in \Gamma_\mu : y \in \supp_I (\pi_x) \} = \{ x_k \}$ and $\#\{ x \in \Gamma_\mu : y \in \supp_I (\pi_x) \} = 1$. The arguments for the other cases are similar. Hence, $(\Gamma_\mu, (\pi_x)_{x \in \Gamma_\mu})$ defines a strongly injective martingale coupling.

Now suppose that at least one of $\Gamma_{a,0} := \Gamma_a \cap \Gamma_0$ or $\Gamma_{b,0} := \Gamma_b \cap \Gamma_0$ is non-empty (note that $\Gamma_{a,0}\cap \Gamma_{b,0}=\emptyset$). {Then we must have $\mu^c(\R) = \mu^c(\Gamma_0)>0$.} 

For each $k$ such that $x_k \in \Gamma_{a,0}$ choose $\tilde{x}_k {\in \Gamma_0}$ such that $\tilde{x}_k \notin \Gamma_a \cup \Gamma_b$, such that $\min \{ y : y \in \supp (\pi^0_{x_k}) \} < \tilde{x}_k < \max \{ y : y \in \supp (\pi^0_{x_k}) \}$ and such that $\tilde{x}_k \neq \tilde{x}_j$ for all $j < k$. Define a measure $\tilde{\pi}_{\tilde{x}_k}$ such that $\supp (\tilde{\pi}_{\tilde{x}_k}) = \supp (\pi^0_{x_k})$ (and thus also $\supp_I(\tilde{\pi}_{\tilde{x}_k}) = \supp_I(\pi^0_{x_k})$) and 
$\tilde{\pi}_{\tilde{x}_k} \in L^1(\tilde{x}_k)$. Let $\tilde{\Gamma}_a = \{ \tilde{x}_k \}_{ \{ k : x_k \in \Gamma_a \cap \Gamma_0 \} }$. 
{Note that $\tilde{\Gamma}_a \subseteq \Gamma_0$.}

Define {$\Gamma_\mu =\Gamma_0 \cup \Gamma_a \cup \Gamma_b$ which we can write as} the disjoint union $\Gamma_\mu = \Gamma_a \cup (\Gamma_0 \setminus (\Gamma_a \cup \tilde{\Gamma}_a \cup \Gamma_b)) \cup ( \Gamma_b \setminus \Gamma_0) \cup \Gamma_{b,0} \cup \tilde{\Gamma}_a$.
For $x \in \Gamma_a$ (such that $x=x_k$) set $\pi_{x} = \pi^a_{x_k}$. For $x \in \Gamma_0 \setminus (\Gamma_a \cup \tilde{\Gamma}_a \cup \Gamma_b)$ define $\pi_x = \pi^0_x$. For $x \in\Gamma_b \setminus \Gamma_0$ define $\pi_x = \delta_x$. The new elements are: first, that for $x \in \Gamma_{b,0}$ we define $\pi_x = \frac{1}{2} \delta_x + \frac{1}{2} \pi^0_x$; 
{and second, that for $x = \tilde{x}_k \in \tilde{\Gamma}_a \subseteq \Gamma_0$ we define $\pi_x = \frac{1}{2} \pi^0_{{x}} + \frac{1}{2} \tilde\pi_x$}. Note that we again have that $\supp_I(\pi_x)\subseteq\supp(\nu)$ for all $x\in\Gamma_\mu$.

The claim is that $(\Gamma_\mu, (\pi_x)_{x \in \Gamma_\mu})$
defines a strongly injective martingale coupling of $\mu$ and $\nu$.

First we check that $(\Gamma_\mu, (\pi_x)_{x \in \Gamma_\mu})$ defines a martingale coupling of $\mu$ and $\nu$.
Define $(\hat{\Gamma}, (\hat{\pi}_x)_{x \in \hat{\Gamma}})$ by $\hat{\Gamma} = \Gamma_a \cup \Gamma_0$ and $\hat{\pi}_x = \pi^a_x$ for $x \in \Gamma_a$ and $\hat{\pi}_x = \pi^0_x$ for $x \in \Gamma_0 \setminus \Gamma_a$. Note that $\Gamma_0 \setminus \Gamma_a$ differs from $\Gamma_0$ by a set of $\mu_0$ measure zero. It follows that $(\Gamma_0 \setminus \Gamma_a, (\hat{\pi}_x)_{x \in \Gamma_0 \setminus \Gamma_a})$ defines a martingale coupling of $\mu^c$ and $\nu_0$. Hence $(\hat{\Gamma}, (\hat{\pi}_x)_{x \in \hat{\Gamma}})$ defines a
martingale coupling of $\mu$ and $\nu$. Then, the candidate coupling $(\Gamma_\mu, (\pi_x)_{x \in \Gamma_\mu})$ defines a martingale coupling of $\mu$ and $\nu$ since it only differs from $(\hat{\Gamma}, (\hat{\pi}_x)_{x \in \hat{\Gamma}})$ on a set of $\mu$ measure zero. 



It remains to prove the strong injectivity property.
As in the previous case, we have that $\{x\in\Gamma_\mu:y\in\supp_I(\pi_x)\}=\emptyset$ for all $y\notin\supp(\nu)$. 
So it remains to show that for each $y \in \supp(\nu)$ there exists a unique $x \in \Gamma_\mu$ such that $y \in \supp_I(\pi_x)$. 

Suppose $y \in \supp_I(\nu^\Delta_k)$ for some (unique) $1\leq k<N+1$. Then $y \in \supp_I(\pi_{x_k})$, but $y \notin \supp_I (\pi_{x_j})$ for any other $j$. Since $y \in \supp_I(\nu^\Delta_k)$ it must be the case that $y \notin \supp (\nu_0)$ and so $y \notin \supp_I (\pi_x)\subseteq\supp(\nu_0)$ for any $x \in \Gamma_0 {\setminus \Gamma_b}$. Further, $y \in E$ so that $y \notin \Gamma_b$ and $y \notin \supp (\pi_x)$ for $x \in \Gamma_b \setminus \Gamma_0$. On the other hand, for all $x\in\Gamma_{b,0}$, $\supp_I(\pi_x)=\{x\}\cup\supp_I(\pi^0_x)\subseteq\Gamma_b\cup\supp(\nu_0)$, and thus also $y\notin\supp_I(\pi_x)$ for all $x\in\Gamma_{b,0}$.
Finally, for $x = \tilde{x}_k \in \tilde{\Gamma}_a$, $\supp (\pi_x) = \supp (\tilde{\pi}_{\tilde{x}_k}) \cup \supp (\pi^0_{\tilde{x}_k})=\supp (\pi^0_{{x}_k}) \cup \supp (\pi^0_{\tilde{x}_k})\subseteq\supp(\nu_0)$ and again $y \notin \supp_{I}(\pi_x)$.

Suppose $y \in \supp(\nu_0)$. It follows that $y\notin \left( \cup_{k=1}^{N} \supp_I (\nu^\Delta_k)  \right)\cup \Gamma_b$ and for every $x \in \Gamma_a$ and every $x \in \Gamma_b \setminus \Gamma_0$, $y\notin \supp_I(\pi_x)$. 
Then, since $\Gamma_\mu = \Gamma_0 \cup \Gamma_a \cup \Gamma_b$ we must have {$\{x\in\Gamma_\mu:y\in\supp_I(\pi_x)\}\subseteq \Gamma_0 \setminus \Gamma_a \subseteq \Gamma_0$}.
By our hypothesis on the case where both $\mu$ and $\nu$ are continuous we know that there exists a unique $x \in \Gamma_0$ such that $y \in \supp_I(\pi^0_x)$. Let this $x$ be denoted by $x_y$. 
We must have that either $x_y\in \Gamma_{0} \setminus \Gamma_a$ or $x_y\in\Gamma_{a,0}$.

If $x_y\in \Gamma_{0} \setminus \Gamma_a$,
then $y \in \supp_I(\pi^0_{x_y}) \subseteq \supp_I(\pi_{x_y})$ with the inclusion being strict if $x_y \in \tilde{\Gamma}_{a} \cup\Gamma_{b,0}$,  and therefore $x_y\in\{x\in\Gamma_\mu:y\in\supp_I(\pi_x)\}$. Moreover, for any other $x \in \Gamma_0 \setminus (\Gamma_a \cup \tilde{\Gamma}_a { \cup \Gamma_b})$, $y \notin \supp_I(\pi^0_x) = \supp_I(\pi_x)$ by the uniqueness of $x_y$ and for any $x = \tilde{x}_j \in \tilde{\Gamma}_a$ with $x \neq x_y$, $\supp_I(\pi_x) = \supp_I(\pi_{\tilde{x}_j}) = \supp_I(\pi^0_{\tilde{x}_j}) \cup \supp_I(\tilde{\pi}_{\tilde{x}_j}) = \supp_I(\pi^0_{\tilde{x}_j}) \cup \supp_I(\pi^0_{{x}_j})$ and since ${x_y \notin \{ x_j, \tilde{x}_j \} }$ we conclude {(using the strong injectivity of $(\Gamma_0,(\pi^0_x)_{x \in \Gamma_0})$) that }$y \notin \supp_I(\pi^0_{\tilde{x}_j}) \cup \supp_I(\pi^0_{{x}_j}) = \supp_I(\pi_x)$. On the other hand, if $x\in\Gamma_{b,0}$ with $x\neq x_y$, then $y\notin \Gamma_b\cup\supp_I(\pi_x^0)\supseteq\{x\}\cup\supp_I(\pi_x^0)=\supp_I(\pi_x)$. Combining all cases we conclude that $\{x\in\Gamma_\mu:y\in\supp_I(\pi_x)\} = \{ x_y \}$.

Now suppose that $x_y\in\Gamma_{a,0}$. Then $x_y=x_k$ for some (unique) $k\geq1$, so that $\pi_{x_y}=\pi^a_{x_k}$, and therefore $y\notin\supp_I(\nu_k^\Delta)=\supp_I(\pi_{x_y})$. However, in this case there exists a unique $\tilde x_j\in\tilde\Gamma_a$ such that $y\in\supp_I(\pi_{x_y}^0)=\supp_I(\pi_{\tilde x_j})$. Then, similarly as in the case when $x_y\in\Gamma_0\setminus\Gamma_a$, using the uniqueness of $x_y$ we conclude that $\{x\in\Gamma_\mu:y\in\supp_I(\pi_x)\} = \{ x_j \}$.
%

Combining the cases when $x_y\in\Gamma_0\setminus\Gamma_a$ and $x_y\in\Gamma_{a,0}$, it follows that $\lvert\{x\in\Gamma_\mu:y\in\supp_I(\pi_x)\}\lvert=1$ for all $y\in\supp({\nu_0})$.

Finally, suppose $y \in \supp (\nu) \setminus E$. Then $y \in \Gamma_b$ and $y \in \{ y \} \subseteq \supp_I (\pi_y)$ and 
although this inequality may be strict (if $y \in \Gamma_0$, then $\supp_I(\pi_y)=\{y\}\cup\supp_I(\pi^0_y)$) it does not change the result we need that there exists 
$x \in \Gamma_\mu$ such that $y \in \supp_I(\pi_x)$ (namely $x=y$). Moreover, this $x$ is clearly unique.

\end{proof}

\section{The construction in simple cases}
\label{sec:exampleconstruction}

In this section we describe the main idea which underpins our construction of an injective martingale coupling, in the irreducible case when $\mu$ {and $\nu$ are} continuous. 
We work in a setting where $\mu$ and $\nu$ both have densities.

Suppose $\mu$ and $\nu$ are absolutely continuous probability measures with interval supports $I_\mu = (\alpha_\mu,\beta_\mu) \subseteq I_\nu = (\alpha_\nu,\beta_\nu)$. Suppose that $\mu$ and $\nu$ have densities $\rho_\mu$ and $\rho_\nu$ respectively, which are piecewise continuous and strictly positive on the interval supports. Suppose further that there is a (central, closed) interval $I^C=[x_*,x^*]$ such that $\rho_\mu>\rho_{\nu}>0$ on the interior of $I^C$, and $\rho_{\mu}<\rho_\nu$ on the intersection of $I_\nu$ with the complement of $I^C$. For examples, consider $(\mu ~ \sim U[-1,1],\nu \sim U[-2,2])$ or $(\mu \sim N(0, \sigma^2_\mu), \nu \sim N(0, \sigma_\nu^2))$ with $\sigma_\mu^2 < \sigma_\nu^2$.

To construct the coupling set $x_0=x_*$ and for $x \geq x_0$ look for solutions $f=f(x), h=h(x)$ to
\begin{equation}
\label{eq:firstexample}
\int_{x_0}^x  z^i \rho_\mu(z) dz =  \int_{f(x)}^{h(x)}  z^i \rho_\nu(z) dz; \hspace{20mm} i = 0,1.
\end{equation}
Assuming that the derivatives exist, we find\footnote{Note that in the regular case with densities the support functions of the left-curtain coupling can be found in a very similar way, but then we find that $(T_d,T_u)$ solve $ T_d'(x) =  - 
\frac{T_u(x)-x}{T_u(x)-T_d(x)} \frac{\rho_\mu(x)}{\rho_\nu(T_d(x)) - \rho_\mu(T_d(x))}$ and $T_u'(x) = \frac{x-T_d(x)}{T_u(x)-T_d(x)} \frac{\rho_\mu(x)}{\rho_\nu(T_u(x))}$; see Equations (3.9) and (3.10) of Henry-Labord\`{e}re and Touzi~\cite{HenryLabordereTouzi:16}.} 
that $f,h$ satisfy
\[ f'(x) =  - \frac{h(x)-x}{h(x)-f(x)} \frac{\rho_\mu(x)}{\rho_\nu(f(x))} ,  \hspace{10mm}    h'(x) = \frac{x-f(x)}{h(x)-f(x)} \frac{\rho_\mu(x)}{\rho_\nu(h(x))} ;  \]
subject to $f(x_0)=x_0=h(x_0)$.
Then on an open interval $\tilde{I}$ to the right of $x_0=x_*$ we have that $h(x)>x$ for $x\in\tilde I$, and it follows that $h$ is strictly increasing and $f$ is strictly decreasing on $\tilde I$. In the terminology of Beiglb\"{o}ck and Juillet \cite{BeiglbockJuillet:16}, see also Beiglb\"{o}ck et al. \cite{BeiglbockHobsonNorgilas:18}, for each such $x$, $\mu$ restricted to $[x_0,x]$ is mapped to its `shadow' $S^\nu(\mu|_{[x_0,x]})$ in $\nu$ which is equal to $\nu$ restricted to the interval $[f(x),h(x)]$. If we define $\pi_{x_0}= \delta_{x_0}$ and $\pi_z = \pi^{f,h}_z$ as in \eqref{eq:pifh} otherwise, then {for each $x \in \tilde{I}$} we have that $(\pi_z)_{z \in [x_0,x]}$ defines a martingale coupling of $\mu$ restricted to $[x_0,x]$ and $\nu$ restricted to the interval $[f(x),h(x)]$.

Let $x_1 = \inf \{x \in ( x_0, \beta_\mu ] : h(x) \leq x \} \wedge \beta_\mu$.

There are two cases: either $x_1 = \beta_\mu$ or $x_1 < \beta_\mu$.

In the former case (see Example~\ref{eg:2uniform} below) it turns our that we must have $x_0 = \alpha_\mu$. Then we have continuous functions $f,h : [x_0,x_1] = [\alpha_\mu,\beta_\mu]  \rightarrow [\alpha_\nu,\beta_\nu]$ with $f$ strictly decreasing and $h$ strictly increasing. Further, for $y \in (\alpha_\mu,\beta_\mu]$ we have $\{ x: y \in \supp(\pi_x) \} = \{ h^{-1}(x) \}$, for $y\in [\alpha_\nu,\alpha_\mu)$ we have $\{ x: y \in \supp(\pi_x) \} = \{ f^{-1}(x) \}$ and for $y=x_0=\alpha_\mu$ we have $\{ x: y \in \supp(\pi_x) \} = \{ x_0 \}$. In particular, we have defined an injective martingale coupling in the sense of \eqref{eq:introstrong} with $\Gamma_\mu = [\alpha_\mu,\beta_\mu] = \supp(\mu)$. For each $y \in [\alpha_\nu, \beta_\nu]$ there exists a unique $x \in [\alpha_\mu = x_0, \beta_\mu]$ such that $y \in \supp(\pi_x)$.

In the latter case (see Example~\ref{eg:mixeduniform} below) we have defined an injective martingale coupling of $\mu|_{[x_0,x_1]}$ with $\nu|_{[f(x_1),h(x_1)]} = S^\nu(\mu|_{[x_0,x_1]})$. {See Figure \ref{fig:densities}.} We consider trying to construct an injective martingale coupling of $\mu|_{\R \setminus [x_0,x_1]}$ and $\nu - S^\nu(\mu|_{[x_0,x_1]})$. This time we use the right-curtain coupling, starting from $x_0$. For $x \leq x_0$ we choose $f(x),h(x)$ to solve, for $i=0,1$,
\begin{equation*}
\int_{x}^{x_0}  z^i \rho_\mu(z) dz  =  
\int_{f(x)}^{f(x_0)}  z^i \rho_\nu(z) dz + \int_{h(x_1)}^{h(x)} z^i \rho_\nu(z) dz 
 = 
\int_{f(x)}^{h(x)}  z^i \rho_{\nu - S^{\nu}(\mu|_{[x_0,x_1]}) }(z) dz. 
\end{equation*}
Let $x_2 = \sup \{x < x_0 : f(x) \geq x \} \wedge \alpha_\nu$. 
We have an injective coupling of $\mu|_{[x_2,x_0)}$ with $S^{\nu -S^{\nu}(\mu|_{[x_0,x_1]}}(\mu|_{[x_2,x_0))}$. Combining this with the construction on $[x_0,x_1]$ and taking care at the endpoints, we have an injective coupling of $\mu|_{[x_2,x_1]}$ with $S^\nu (\mu|_{[x_2,x_1]})$. We now proceed inductively. 
The construction may terminate, or it may not. In either case we construct an injective martingale coupling of $\mu$ and $\nu$ in the sense of \eqref{eq:introweak}. We can extend it to become a martingale coupling in the sense of \eqref{eq:introstrong} by dealing with any endpoints.
\begin{figure}[H]
	\centering
\begin{tikzpicture}
  \message{Low sensitivity 2^^J}
  
  \def\q{17.0};
  \def\ex{-9.0};
  \def\h{5.0};
  \def\f{-13.0};
  \def\B{0.0};
  \def\S{18.0};
  \def\Bs{7.0};
  \def\Ss{12.0};
  \def\xmax{\S+3.2*\Ss};
  \def\ymin{{-0.3*gauss(\B,\B,\Bs)}};
  
  \begin{axis}[every axis plot post/.append style={
               mark=none,domain={-1*(\xmax)}:{1*\xmax},samples=\Nn,smooth},
               xmin={-1*(\xmax)}, xmax=\xmax,
               ymin=\ymin, ymax={1.1*gauss(\B,\B,\Bs)},
               axis lines=middle,
               axis line style={draw=none},
               enlargelimits=upper, 
               ticks=none,
               every axis x label/.style={at={(current axis.right of origin)},anchor=north west},
               width=0.7*\textwidth, height=0.5*\textwidth,
              ]
    
    \addplot[blue, name path=B,thick] {gauss(x,\B,\Bs)};
    \addplot[red,  name path=S,thick] {gauss(x,\B,\Ss)};
    \addplot[black,thin]
      coordinates {(\ex, {gauss(\ex,\B,\Bs)}) (\ex, 0)};
    \addplot[black,thin]
      coordinates {(0, {gauss(0,\B,\Bs)}) (0, 0)};
    \addplot[black,thin]
      coordinates {(\h, {gauss(\h,\B,\Ss)}) (\h, 0)};
    \addplot[black,thin]
      coordinates {(\f, {gauss(\f,\B,\Ss)}) (\f, 0)};
    \draw[->,thick]  (0,0.05) to[out=200, in=90] (\f,{gauss(\f,\B,\Ss)+0.002});
    \draw[->,thick]  (0,0.05) to[out=300, in=90] (\h,{gauss(\h,\B,\Ss)+0.002});
    
    \addplot[black,thin]
      (-25,0) -- (25,0); 
    \path[name path=xaxis]
      (\pgfkeysvalueof{/pgfplots/xmin},0) -- (\pgfkeysvalueof{/pgfplots/xmax},0); 
    \addplot[pattern color=white!50!red,pattern=crosshatch]  fill between[of=xaxis and S, soft clip={domain=0:\h}];
      \addplot[pattern color=white!50!blue,pattern=dots]  fill between[of=S and B, soft clip={domain=\ex:0}];
      \addplot[pattern color=white!50!red,pattern=crosshatch]  fill between[of=xaxis and S, soft clip={domain=\f:\ex}];
    \addplot[pattern color=white!50!red,pattern=crosshatch] fill between[of=xaxis and S, soft clip={domain=\ex:0}];
    
    \addplot[gray,densely dashed] (-1*\ex,{gauss(-1*\ex,\B,\Bs)})--(-1*\ex,-0.004);
    \node[below]          at (-1*\ex,-0.0025)       {$x^*$};
    \addplot[gray,densely dashed] (\ex,0)--(\ex,-0.004);
    \node[below]          at (\ex,-0.004)       {$x_*$};
    \addplot[gray,densely dashed] (0,0)--(0,-0.004);
    \node[below]          at (0,-0.004)       {$x$};
    \addplot[gray,densely dashed] (\f,0)--(\f,-0.012);
    \node[below]          at (\f,-0.01)       {$f(x)$};
    \addplot[gray,densely dashed] (\h,0)--(\h,-0.012);
    \node[below]          at (\h,-0.01)       {$h(x)$};
    \node[black!20!red]  at (23,0.015) {$\rho_\nu$};
    \node[black!50!blue]  at (10,0.045) {$\rho_\mu$};
  \end{axis}
\end{tikzpicture}
\caption{Sketch of the densities $\rho_\mu$ and $\rho_\nu$ and the locations of $f = f (x)$, $h = h(x)$ for given
$x\in(x_0=x_*,x_1)$. {Mass in $(x_0= x_*,x)$ according to the initial law in  is mapped to the interval $(f(x),h(x))$ according to the target law. In particular, at the margins, mass at $x$ is mapped to $f(x)$ and $h(x)$ in a way which respects the martingale property.}
}
\label{fig:densities}
\end{figure}
\begin{eg}
\label{eg:2uniform}
Suppose $\mu \sim U[-1,1]$ and $\nu \sim U[-2,2]$. Then $I^C = [-1,1]$. Set $x_0=x_* =-1$ and for $x \in [-1,1]$ let $f:[-1,1] \mapsto [-2,2]$ and $h:[-1,1] \mapsto [-2,2]$ solve
\[ \int_{x_0}^{x} x^i \frac{dx}{2}  = \int_{f(x)}^{h(x)} y^i \frac{dy}{4} \]
with $f(-1)=-1=h(-1)$. We find that on $[-1,1]$, $f,h$  solve
\[ (x+1) = \frac{1}{2} ( h(x)-f(x) ); \hspace{20mm}  x^2 - 1 = \frac{1}{2} (h(x)- f(x))^2, \]
and then $f(x) = \frac{-(x+3)}{2}$ and $h(x)= \frac{3x+1}{2}$.

In this case $\{x \in (x_0,\beta_\mu=1) : h(x) \leq x \} = \emptyset$ and $x_1 = \beta_\mu=1$. 

Define $\Gamma_\mu = [-1,1]$, $\pi_{-1} = \delta_{-1}$ and for $x \in (-1,1]$, $\pi_x= \pi^{f,h}_x$. Then $(\pi_x)_{x \in \Gamma_\mu}$ defines a strongly injective martingale coupling of $\mu$ and $\nu$.

Alternatively, we may set $\Gamma_\mu = [-1,1)$, $\pi_{-1} = \delta_{-1}$ and for $x \in (-1,1)$, $\pi_x= \pi^{f,h}_x$. Then $(\pi_x)_{x \in \Gamma_\mu}$ defines a strongly injective martingale coupling of $\mu$ and $\nu$ on its irreducible component.

\end{eg}

\begin{eg}
\label{eg:mixeduniform}
$ \mu \sim \frac{1}{2} U[-1,1] + \frac{1}{2}U[-2,2] \mapsto \nu \sim U[-2,2]$.

In this case $x_* = -1$ and in the first step of the construction we find that for $-1 < x \leq 1$ we have that $(f =f(x),h=h(x))$ solve
\[ \int_{x_0}^{x} \frac{3 z^i}{8} dz = \int_{f(x)}^{h(x)} \frac{y^i}{4} dy 
\hspace{20mm} i=0,1. \]
We find
\[ f(x) = - \frac{(x+5)}{4}; \hspace{20mm}  h(x) = \frac{(5x+1)}{4}, \]
so that $(f(1)=-\frac{3}{2}, h(1)=\frac{3}{2})$. At this stage we have coupled $\frac{3}{4}U[-1,1]$ with $\frac{3}{4}U[-\frac{3}{2}, \frac{3}{2}]$. It remains to couple $\frac{1}{8}U[-2,-1] + \frac{1}{8}U[1,2]$ with $\frac{1}{8}U[-2,-\frac{3}{2}] + \frac{1}{8}U[\frac{3}{2}, 2]$.

For $x \geq 1$ we define $(f =f(x),h=h(x))$ as solutions to
\[ \int_{1}^{x} \frac{z^i}{8} dz = \int_{f(x)}^{-3/2} \frac{y^i}{4} dy +  \int_{3/2}^{h(x)} \frac{y^i}{4} dy; \hspace{20mm} i=0,1. \]
We find $h(x)-f(x) = \frac{5+x}{2}$ and $h(x)^2 - f(x)^2 = \frac{x^2-1}{2}$. These equations can be solved to give
\[ f(x) = \frac{x^2 - 10x-27}{4(x+5)}; \hspace{20mm}  h(x) = \frac{3x^2 + 10x +23}{4(x+5)}. \]
Now let $x_1$ be the first solution above 1 to $h(x)=x$.
It is straightfoward to calculate that $x_1$ is the unique root in $(1,2)$ of $Q_1$ where $Q_1(x)= x^2 + 10x -23$.

At this stage we have constructed an injective coupling which embeds $\mu|_{[x_0=-1,x_1]}$ in $\nu|_{[f(x_1), h(x_1)=x_1]}$. 
The next stage is to define $f,h$ on $[x_2,x_0)$ where $x_2$ is to be determined. Here, for $x \in [x_2,x_0)$, $f$ and $h$ satisfy
\[ \int_{x}^{x_0=-1} \frac{z^i}{8} dz = \int_{f(x)}^{f(x_1)} \frac{y^i}{4} dy +  \int_{h(x_1)= x_1}^{h(x)} \frac{y^i}{4} dy; \hspace{20mm} i=0,1, \]
and $x_2$ is then chosen so that $f(x_2)=x_2$. We find that 
\begin{eqnarray*} 
h(x)-f(x) & = & x_1 - f(x_1) + \frac{x_0 - x}{2}, \\
h(x)^2-f(x)^2 & = & x_1^2-f(x_1)^2 + \frac{x_0^2 - x^2}{2},
\end{eqnarray*}
and using $f = \frac{h^2 - f^2 - (h-f)^2}{2(h-f)}$ to eliminate $h$ and $f(x_2)=x_2$, it follows that $x_2$ is the solution to
\[ (x_1 - f(x_1)) \left(x_0 + x_2 - 2f(x_1) \right) = \frac{1}{4}(x_0-x_2)^2.  \]
Using that $x_1 - f(x_1) = h(x_1)-f(x_1) = \frac{x_1 + 5}{2}$ and $x_0=-1$ we find $x_2 = x_1 + 4 - 4 \sqrt{x_1 + 2}$.
We have mapped $\mu|_{[x_2,x_1]}$ to $\nu|_{[f(x_2) =x_2, h(x_2)]}$; see Figure \ref{fig:2iter}.

\begin{figure}[H]
	\centering
\begin{tikzpicture}

\begin{axis}[%
width=4.521in,
height=3.566in,
at={(0.758in,0.481in)},
scale only axis,
xmin=-1,
xmax=11,
ymin=-1,
ymax=11,
axis line style={draw=none},
ticks=none
]
\node[below,scale=1] at (4,0) {$x_0$};
\node[below,scale=1] at (5,0) {$x_1$};
\node[below,scale=1] at (2,0) {$x_2$};
\node[below,scale=1] at (9,0) {$x_3$};
\draw[blue,densely dotted, thick] (0,0)--(10,10);
\draw[red,thick] (4,4) to[out=70, in=185] (5,5);
\draw[red,thick] (4,4) to[out=300, in=170] (5,3.5);

\draw[red,thick] (4,3.5) to[out=240, in=20] (2,2);
\draw[red,thick] (4,5) to[out=110, in=350] (2,6);

\draw[red,thick] (5,6) to[out=40, in=190] (9,9);
\draw[red,thick] (5,2) to[out=350, in=180] (9,1);
\draw[gray,thin,dashed]  (2,0)--(2,6) -- (5,6)--(5,0);
\draw[gray,thin,dashed]  (2,2)--(5,2);
\draw[gray,thin,dashed]  (5,5)--(4,5)--(4,3.5)--(5,3.5);
\draw[gray,thin,dashed]  (4,3.5)--(4,0);
\draw[gray,thin,dashed]  (9,9)--(9,0);
\node[red] at (4,7) {$x\mapsto h(x)$};
\node[red] at (7,3) {$x\mapsto f(x)$};
\end{axis}
\end{tikzpicture}
\caption{Stylized plots of functions $f$ and $h$ (on $[x_2,x_3]$) that support the injective coupling of Example \ref{eg:mixeduniform}. Note that $h$ (resp. $f$) is non-decreasing (resp. non-increasing) on $[x_0,x_1]$, non-increasing (resp. non-decreasing) on $[x_2,x_0]$ and again non-decreasing (resp. non-increasing) on $[x_1,x_3]$.}
\label{fig:2iter}
\end{figure}

We now proceed inductively. Having embedded $\mu_{[x_{2j},x_{2j-1}]}$ to $\nu|_{[f(x_{2j}) =x_{2j}, h(x_{2j})]}$ and working to the right from $x_{2j-1}$ we define
$f,h$ as solutions to
\begin{equation}
    \label{eq:eg4.2a}
\int_{x_{2j-1}}^{x} \frac{z^i}{8} dz = \int_{f(x)}^{f(x_{2j})= x_{2j}} \frac{y^i}{4} dy +  \int_{h(x_{2j})}^{h(x)} \frac{y^i}{4} dy; \hspace{20mm} i=0,1. 
\end{equation}
Focusing on the case $i=0$ in \eqref{eq:eg4.2a} and using the fact that $x_{2j-1}$ is chosen so that $h(x_{2j+1}) = x_{2j+1}$ we get that
\[ \frac{x_{2j+1} - x_{2j-1}}{2} = x_{2j} - {f(x_{2j+1})} + x_{2j+1} - h(x_{2j}), \] 
which can be rewritten as
\[ h(x_{2j})  - \frac{x_{2j} + x_{2j-1}}{2} = \frac{x_{2j+1} + x_{2j}}{2} - f(x_{2j+1}) . \]

By a similar analysis of the case working left from $x_{2j-2}$ we define
$f,h$ as solutions to
\begin{equation}
    \label{eq:eg4.2aD}
\int_{x}^{x_{2j-2}} \frac{z^i}{8} dz = \int_{f(x)}^{f(x_{2j-1})} \frac{y^i}{4} dy +  \int_{h(x_{2j-1})=x_{2j-1}}^{h(x)} \frac{y^i}{4} dy; \hspace{10mm} i=0,1. 
\end{equation}
Focusing on the case $i=0$ and using $f(x_{2j})=x_{2j}$ we find that
$(x_{2j}, h(x_{2j}))$ solve
\[ h(x_{2j})  - \frac{x_{2j} + x_{2j-1}}{2} = \frac{x_{2j-1} + x_{2j-2}}{2} - f(x_{2j-1}).  \]
In particular, setting $\Delta = \frac{x_1 + x_0}{2} - f(x_1)$ we find that $\Delta=2$ and for all $j \geq 1$, 
\begin{equation}
    \label{eq:Delta=2}
h(x_{2j})  - \frac{x_{2j} + x_{2j-1}}{2} = 2 = \frac{x_{2j+1} + x_{2j}}{2} - f(x_{2j+1}) .
\end{equation}

Returning to \eqref{eq:eg4.2a} and considering both $i=0$ and $i=1$
we get expressions for $h(x)-f(x)$ and $h(x)^2-f(x)^2$. We can eliminate $f(x)$ using $h = \frac{h^2 - f^2 + (h-f)^2}{2(h-f)}$ and 
then evaluating the expressions at $x_{2j+1}$ and using $h(x_{2j+1}) = x_{2j+1}$
we find that $x_{2j+1}$ solves
\begin{equation}
\label{eq:eg4.2b}
\frac{(x_{2j+1}-x_{2j-1})^2}{4} = (h(x_{2j}) - x_{2j})(2 h(x_{2j}) - x_{2j+1} - x_{2j-1}).  
\end{equation}

Using \eqref{eq:Delta=2} we have $h(x_{2j})-x_{2j} = 2 + \frac{x_{2j-1}- x_{2j}}{2}$ and 
$h(x_{2j}) - \frac{x_{2j-1} + x_{2j+1}}{2} = 2 - \frac{x_{2j+1} - x_{2j}}{2}$.
Rewriting \eqref{eq:eg4.2b} we find
\begin{equation}
\label{eq:eg4.2c}
\frac{(x_{2j+1}-x_{2j-1})^2}{4} = (4 + x_{2j-1}- x_{2j})(4 - x_{2j+1} + x_{2j}).  
\end{equation}

Let $\Gamma_{2j+1} = 4 - x_{2j+1} + x_{2j}$ and $\Gamma_{2j} = 4 + x_{2j} - x_{2j-1}$. Then $x_{2j+1}-x_{2j-1}= \Gamma_{2j+1}-\Gamma_{2j}$ and \eqref{eq:eg4.2c} becomes
\[ \frac{(\Gamma_{2j+1}-\Gamma_{2j})^2}{2} = \Gamma_{2j+1} (8 - \Gamma_{2j}) .    \]
This simplifies to $\Gamma_{2j+1} = \Upsilon (\Gamma_{2j})$ where $\Upsilon(\gamma) = 8 - \sqrt{64 - \gamma^2}$. Note that $\Upsilon(\gamma) < \gamma$ for $\gamma \in (0,8)$.

Returning to \eqref{eq:eg4.2aD} and following a very similar analysis we again find that $\Gamma_{2j} = \Upsilon(\Gamma_{2j-1})$. Hence we have
\[ \Gamma_{k+1} = \Upsilon(\Gamma_{k}); \hspace{10mm} k \geq 1; \]
with initial condition $\Gamma_1 = 5 - x_1$, where $x_1$ is the positive root of $Q_1$.

\end{eg}

\section{Notation and preliminaries}
\label{sec:notation}

\subsection{Convex Hulls}
For a continuous function $\sH :\R \mapsto \R_+$ let $\sH^c$ denote the convex hull of $\sH$. For $z \in \R$ let $X^+_{\sH}(z) = \sup \{ w  : w\leq z , \sH(w) = \sH^c(w) \}$ and $Z^-_{\sH}(z) = \inf \{w : w \geq z : \sH(w) = \sH^c(w) \}$ with the convention that $\sup \emptyset = -\infty$ and $\inf \emptyset = \infty$. Note that $X^+_\sH(z)=z$ if and only if $Z^-_\sH(z)=z$ (and also if and only if $\sH(z) = \sH^c(z)$).
Now, for $z\in\R$, let $\psi_{ -}(z):=(\sH^c)'(z-)$, where $f'(z -)$ denotes the left derivative (if it exists) of a measurable function $f:\R\to\R$, and similarly, let $\psi_{+}(z):=(\sH^c)'(z+)$, where $f'(z +)$ denotes the right derivative. Define $X^-_\sH(z):= \inf \{w : w \leq z , \sH^{c}(w) = L_{\sH^{c}}^{z,\psi_{ -}(z)}(w) \}$ and $Z^+_\sH(z):= \sup \{w : w \geq z , \sH^c(w) = L_{\sH^c}^{z,\psi_{ +}(z)}(w) \}$, where, for $f:\R\to\R$, $\phi\in\R$ and $z\in\R$ we define $L_f^{z,\phi}$ to be the straight line $L_f^{z,\phi}(k) =f(z)+\phi(k-z)$.

\begin{figure}[H]

		\centering
	
		\begin{tikzpicture}
		
		\begin{axis}[%
		width=4.521in,
		height=2.566in,
		at={(0.758in,0.481in)},
		scale only axis,
		xmin=-11,
		xmax=11,
		ymin=-0.2,
		ymax=0.8,
		axis line style={draw=none},
		ticks=none
		]
		\draw[thin] (-10,0)--(10,0);
        \draw[red,thick,dotted] (-10,0.6) to[out=280, in=130] (-8,0.3);
        \draw[red,thick](-8,0.3)--(-5,0.125);
        \draw[red,thick,dotted](-5,0.125) to[out=340, in=180] (-3,0.1)-- (6,0.21)to[out=10, in=260] (10,0.6);
         \draw[red,thick] (-3,0.1)-- (6,0.21);
         \draw[blue,thick,dashed] (-8,0.3) to[out=320, in=160] (-7,0.3) to[out=340, in=140] (-5,0.125);
        \draw[blue,thick,dashed] (-3,0.1) to[out=10, in=180] (-2.5,0.15)  to[out=0, in=190] (-1,0.1244) to[out=10, in=180] (1.5,0.25) to[out=0, in=190] (6,0.21);
        \draw[gray,thin,dashed] (-3,0.1)--(-3,0) ;
        \draw[gray,thin,dashed] (-1,0.1244)--(-1,0) ;
        \draw[gray,thin,dashed] (6,0.21)--(6,0) ;        
        \node[below, scale=0.8] at (-3,0) {$X^-_{\sH}(z)$};
        \node[below, scale=0.8] at (-1,0) {$X^+_{\sH}(z)$};
              \node[below, scale=0.8] at (6,0) {$Z^-_{\sH}(z)=Z^+_{\sH}(z)$};
        \node[below, scale=0.8] at (1.5,0) {$z$};
        \draw[gray,thin,dashed] (1.5,0.25)--(1.5,0) ;
        \node[above, scale=1, blue] at (1.5,0.3) {$y\mapsto \sH(y)$};
        \node[above, scale=1, red] at (-7,0.5) {$y\mapsto \sH^c(y)$};
        \end{axis}
		\end{tikzpicture}%
				\caption{The graphs of a continuous $y\mapsto\sH(y)$ and its convex hull $\sH^c$. The dashed (resp. dotted) curve represents $\sH$ on $\{\sH>\sH^c\}$ (resp. on $\{\sH=\sH^c\}$), while the solid lines correspond to $\sH^c$ on disjoint intervals that belong to $\{\sH>\sH^c\}$. In the figure, $z\in\R$ is such that $\sH(z)>\sH^c(z)$, and then $\sH^c$ is linear on $(X^-_{\sH}(z),Z^-_{\sH}(z)=Z^+_{\sH}(z))\supset(X^+_{\sH}(z),Z^-_{\sH}(z)=Z^+_{\sH}(z))\ni z$ with slope $(\sH^c)'(z)=\psi_+(z)=\psi_-(z)$.}
			\label{fig:convexHull}
		\end{figure}

Note that $X^-_{\sH}(z) \leq X^+_{\sH}(z) \leq z \leq Z^-_{\sH}(z)\leq Z^+_{\sH}(z)$, for all $z\in\R$. Indeed, $X^-_{\sH}(z)\leq z\leq Z^+_{\sH}(z)$ for all $z\in\R$, and thus if $X^+_\sH(z)=z=Z^-_\sH(z)$ the claim is immediate. On the other hand, if $X^+_\sH(z)<Z^-_\sH(z)$, then $\sH^c(z)<\sH(z)$, $\sH^c$ is linear on $(X^+_\sH(z),Z^-_\sH(z))\ni z$, and it follows that $X^-_{\sH}(z) \leq X^+_{\sH}(z) < z < Z^-_{\sH}(z)\leq Z^+_{\sH}(z)$. Also, in the latter case $(\sH^c)'(z)$ is well-defined and $(\sH^c)'(z)=\psi_{ -}(z) { = \psi_+(z)}$. See Figure~\ref{fig:convexHull}.

Suppose $\sH$ is such that its left and right derivatives exist everywhere and that $\sH'(z-) \geq \sH'(z+)$ everywhere. Then $\psi_-=\psi_+$. This is clear at any $z$ for which $\sH(z)>\sH^c(z)$. If $\sH(z)=\sH^c(z)$ then
\[ \psi_-(z) \leq \psi_+(z) = (\sH^c)'(z+) \leq \sH'(z+) \leq \sH'(z-) \leq (\sH^c)'(z-) = \psi_-(z) . \]
where the first inequality is true since $\psi_{\pm}$ are the left/right derivatives of a convex function, the two equalities are by definition, the second and fourth inequalities hold since $\sH(z)=\sH^c(z)$ and the third inequality holds by our hypothesis on $\sH$.

\subsection{Measures and convex order}\label{sec:measures}

For $\chi\in\sP$ we denote by $G_\chi:[0,\chi(\R)]\to\R$ a quantile of function of $\chi$, i.e., a generalized inverse of $x\mapsto F_\chi(x):=\chi((-\infty,x])$. (In what follows we set $F_\chi(-\infty):=\lim_{x\to-\infty} F_\chi(x)=0$ and $F_\chi(\infty):=\lim_{x\to\infty}F_\chi(x)=\chi(\R)$.) There are two canonical versions of $G_\chi$: the left-continuous and right-continuous versions
correspond to $\overrightarrow{G}_\chi(u)=\sup\{k\in\R:\chi((-\infty,k])<u\}$ (with convention $\sup\emptyset=-\infty$) and $\overleftarrow{G}_\chi(u)=\inf\{k\in\R:\chi((-\infty,k])>u\}$ (with convention $\inf\emptyset=\infty$), for $u\in[0,\chi(\R)]$, respectively (the directions of arrows represent the left and right-continuity of $\overrightarrow{G}_\chi$ and $\overleftarrow{G}_\chi$, respectively). In particular, an arbitrary version of the quantile function $G_\chi$ satisfies $\overrightarrow{G}_\chi\leq G_\chi\leq \overleftarrow{G}_\chi$ on $[0,\chi(\R)]$. Note that $G_\chi$ may take values $-\infty$ and $\infty$ at the left and right end-points of $[0,\chi(\R)]$, respectively. For any quantile function we set $G_\chi(0-)=-\infty$ and $G_\chi(\chi(\R)+)=\infty$. {{Note that $\overleftarrow{G}_\chi(0) = G_\chi(0+) = \alpha_\chi$ and $\overrightarrow{G}_\chi( \chi(\R) ) = G_\chi( \chi(\R)-) = \beta_\chi$ where $\alpha_\chi,\beta_\chi$ are the endpoints of the smallest interval containing the support of $\chi$.}}

For $\chi \in \sP$ let $\bar{\chi} = \int_\R z \chi(dz)$ and define $P_\chi:\R \mapsto \R_+$ by $P_\chi(k) = \int_{\R} (k-z)^+ \chi(dz)$ and $C_\chi:\R \mapsto \R_+$ by $C_\chi(k) = \int_{\R} (z-k)^+ \chi(dz)$. Then $P_\chi$ (respectively $C_\chi$) is an increasing (respectively decreasing) convex function. Note that {\em put-call parity} gives that $P_\chi(k) - C_\chi(k) = k \chi(\R) - \bar{\chi}$.

For $\chi \in \sP$ and $x\in\R\cup\{-\infty,\infty\}$, define  $\chi_x\in\sP_{F_\chi(x)}$ by
\begin{equation}\label{eq:mu_x}\chi_{x} = \chi|_{(-\infty,x]} ,\quad x\in\R. 
\end{equation}
For $-\infty\leq x<\alpha_\chi$ (and we also can take $x=\alpha_\chi$ in the case $\chi$ is continuous), we treat $\chi_x$ as the zero measure, whereas $\chi_{x} = \chi$ for all $x\geq \beta_\chi$. Then for $x\in\R\cup\{-\infty,\infty\}$, $\chi_x\leq\chi$, $\chi_x(\R)=F_\chi(x)$ and $P_{\chi_x}(k) = P_\chi(k)$ for $k \leq G_\chi(F_\chi(x)+)$. The measure $\chi_x$ consists of the left-most part of $\chi$ of total mass $F_\chi(x)$.

For a pair of measures $\mu, \nu \in \sP$ such that for $i = 0,1$, $\int_{\R} z^i \mu(dz) = \int_{\R} z^i \nu(dz)$ define $D=D_{\mu,\nu}:\R \mapsto \R$ by $D_{\mu,\nu}(z) = P_{\nu}(z)-P_{\mu}(z) = C_{\nu}(z)-C_{\mu}(z)$, which coincides with \eqref{eq:Dsec2}. (The equality of these two alternative expressions follows from the put-call parity.) Then, $\lim_{z \rightarrow \pm \infty} D_{\mu,\nu}(z) = 0$. 
Moreover, $D_{\mu,\nu} \geq 0$ if and only if $\mu\leq_{cx}\nu$. 


For a pair of measures $\mu, \nu \in \sP$ we write $\mu\leq_{pcx}\nu$ if $\int_\R fd\mu\leq\int_\R fd\nu$ for all non-negative and convex $f:\R\to\R$. Given $\mu\leq_{pcx}\nu$, the set $\{\theta:\mu\leq_{cx}\theta\leq \nu\}$ is non-empty and admits the minimal element w.r.t. $\leq_{cx}$. In particular, the so-called \textit{shadow} measure of $\mu$ in $\nu$, denoted by $S^\nu(\mu)$, is an element of $\{\theta:\mu\leq_{cx}\theta\leq \nu\}$ such that $S^\nu(\mu)\leq_{cx}\chi$ for all $\chi\in\{\theta:\mu\leq_{cx}\theta\leq \nu\}$. See Beiglb\"ock and Juillet \cite[Lemma 4.6]{BeiglbockJuillet:16}.

\section{A building block for injective martingale couplings}
\label{sec:dispersion}

In the light of Proposition~\ref{prop:irreducible}, from now on we assume that $\mu \leq_{cx} \nu$ are such that $\{ x : D_{\mu,\nu}(x) > 0 \}$ is an interval and is equal to $I_\nu=(\alpha_\nu,\beta_\nu)$ where $\alpha_\nu < \beta_\nu$. In addition, thanks to Proposition~\ref{prop:reduction}, we can also assume that both measures $\mu$ and $\nu$ are continuous.
\begin{sass}\label{sass:atomfree}
	$\mu,\nu\in\sP$ are distinct non-zero elements of $\sP$ such that $\mu \leq_{cx}\nu$, $\{ x : D_{\mu,\nu}(x) > 0 \}=I_\nu=(\alpha_\nu,\beta_\nu)$ and both measures are atom-free (i.e., $\mu(\{x\})=0=\nu(\{x\})$ for all $x\in\R$).
	\end{sass}
    
It follows from Standing Assumption~\ref{sass:atomfree}
that $D'_{\mu,\nu}(\alpha_\nu) = 0 = D'_{\mu,\nu}(\beta_\nu)$ whenever $-\infty<\alpha_\nu<\beta_\nu<\infty$, or more generally, $\lim_{k\downarrow\alpha_\nu}D'_{\mu,\nu}(k+)=0=\lim_{k\uparrow\beta_\nu}D'_{\mu,\nu}(k-)$.
In the case of a finite endpoint this follows since, under Standing Assumption~\ref{sass:atomfree}, $D'_{\mu,\nu}$ is continuous (kinks correspond to atoms of $\mu$ and $\nu$) and $D_{\mu,\nu} \equiv 0$ outside $I_\nu$.
{Standing Assumption~\ref{sass:atomfree} will remain operative until we complete the paper with a proof of Theorem~\ref{thm:mainintro} at the very end of Section~\ref{sec:lc}.}

Let $G_\mu$ be an arbitrary quantile function of $\mu$. For $x\in\R$, recall the definition of $\mu_x$ (see \eqref{eq:mu_x}). Since $\mu$ is atom-less, we have that $\mu_x=\mu\lvert_{(-\infty,x)}=\mu\lvert_{(-\infty,x]}$ for all $x\in\R$.

For each $x_1,x_2\in\R$ with $x_1\leq x_2$ define $\sE^{\mu,\nu}_{x_1,x_2}= \sE_{x_1,x_2} :\R\mapsto\R_+$ by 
\begin{equation}
\label{eq:sEdef}
 \sE_{x_1,x_2}(k)= \left\{ \begin{array}{lcl} D_{\mu,\nu}(k)+C_{\mu_{x_1}}(k), & \; & k \leq x_1 \\
 D_{\mu,\nu}(k), && x_1<k < x_2  \\
 P_\nu(k)-P_{\mu_{x_2}}(k),  && k \geq x_2 . \\ \end{array} \right.
\end{equation}
By design $\sE_{x_1,x_2}$ {is non-negative and continuously differentiable}, including at $x_1$ and $x_2$. Note that on $k \leq x_1$ we have that $\sE_{x_1,x_2}$ does not depend on $x_2$; similarly, on $k \geq x_2$, $\sE_{x_1,x_2}$ does not depend on $x_1$.
{Given that $D_{\mu,\nu} = C_{\nu}-C_{\mu} = P_\nu-P_\mu$ there are other ways to write $\sE_{x_1,x_2}$; for example, for $k \leq x_1$ we have $\sE_{x_1,x_2}(k) = C_\nu(k) - C_{\mu - \mu_{x_1}}(k)$ and for $k \geq x_2$ we have $\sE_{x_1,x_2}(k) = D_{\mu,\nu}(k) + P_{\mu - \mu_{x_2}}(k)$.}

We extend the definition $\sE_{x_1,x_2}$ by allowing $x_1,x_2\in\{-\infty,\infty\}$. In particular, we use that for all $-\infty\leq x\leq \alpha_\mu$ we have that $\mu_{x}$ is the zero measure, and thus $C_{\mu_{x}}=P_{\mu_{x}}\equiv0$, whilst for all $\beta_\mu\leq x\leq\infty$ we have that $\mu_{x}=\mu$, and thus $C_{\mu_{x}}=C_\mu$ and $P_{\mu_{x}}=P_\mu$. Note further that if $x_1<x_2$ are such that $\mu\lvert_{(x_1,x_2)}(\R)=\mu_{x_2}(\R)-\mu_{x_1}(\R)=0$, then $\sE_{x_1,x_2}=\sE_{z_1,z_2}$ for all $x_1\leq z_1\leq z_2\leq x_2$.

We will often consider the convex hull $\sE^c_{x_1,x_2}$ of $\sE_{x_1,x_2}$. For this note that, if $x_1\leq x_2$ with $x_2=-\infty$, then $\sE_{x_1,x_2}=\sE_{-\infty,-\infty}=P_\nu=P_\nu^c=\sE^c_{x_1,x_2}$. On the other hand, if $-\infty=x_1<x_2=\infty$, then $\sE_{x_1,x_2}=D_{\mu,\nu}$, and then $\sE^c_{x_1,x_2}\equiv 0$ (since $\mu\leq_{cx}\nu$). Finally, if $x_1=\infty$, then $\sE_{x_1,x_2}=\sE_{\infty,\infty}=C_\nu=C^c_\nu=\sE^c_{x_1,x_2}$.

Recall that, for measures $\eta\leq_{pcx}\nu$, $S^\nu(\eta)$ denotes the shadow of $\eta$ in $\nu$. Note that if $\eta\leq\mu\leq_{cx}\nu$ then $\eta\leq_{pcx}\nu$. {The next lemma describes the shadow of $\mu_{x_2}- \mu_{x_1}$ in terms of $\sE_{x_1,x_2}$ and its convex hull.}
\begin{lem}\label{lem:E_vu}
Suppose $\mu\leq_{cx}\nu$ and let $-\infty\leq x_1\leq x_2\leq \infty$. Then
$$
P_{\nu-S^\nu(\mu_{x_2}-\mu_{x_1})}(k)=\sE^c_{x_1,x_2}(k)-(\overline{\mu_{x_1}}-F_\mu(x_1)k),\quad k\in\R.
$$
In particular, the second (distributional) derivative of $\sE^c_{x_1,x_2}$ corresponds to the measure $\nu-S^\nu(\mu_{x_2}-\mu_{x_1})$.
\end{lem}
\begin{proof}
Let $l_{x_1}:\R\to\R$ be given by $l_{x_1}(k)=\overline{\mu_{x_1}}-F_\mu(x_1)k$, $k\in\R$. {Note that $l_{x_1}$ is the asymptote to $\sE_{x_1,x_2}$ as $k \rightarrow -\infty$.} Note that, if $x_1\leq\alpha_\mu$, then $\mu_{x_1}$ is the zero measure and $l_{x_1}\equiv 0$. By Beiglb\"ock et al. \cite[Theorem 4.7]{BeiglbockHobsonNorgilas:18}
$$
P_{\nu-S^\nu(\mu_{x_2}-\mu_{x_1})}=P_\nu-P_{S^\nu(\mu_{x_2}-\mu_{x_1})}=(P_\nu-P_{\mu_{x_2}-\mu_{x_1}})^c,
$$
and therefore it is enough to show that $\sE^c_{x_1,x_2}-l_{x_1}=(P_\nu-P_{\mu_{x_2}-\mu_{x_1}})^c$. On the other hand, by Beiglb\"ock et al. \cite[Lemma 2.4]{BeiglbockHobsonNorgilas:18} and linearity of $l_{x_1}$ we have that
$$
(\sE_{x_1,x_2}-l_{x_1})^c=(\sE^c_{x_1,x_2}-l_{x_1})^c=\sE^c_{x_1,x_2}-l_{x_1}.
$$

We now show that $(\sE_{x_1,x_2}-l_{x_1})=(P_\nu-(P_{\mu_{x_2}}-P_{\mu_{x_1}}))=(P_\nu-P_{\mu_{x_2}-\mu_{x_1}})$. Suppose $k\leq x_1$. Then
\begin{align*}
\sE_{x_1,x_2}(k)=D_{\mu,\nu}(k)+C_{\mu_{x_1}}(k)&=D_{\mu,\nu}(k)+P_{\mu_{x_1}}(k)+l_{x_1}(k)\\&=D_{\mu,\nu}(k)+P_\mu(k)+l_{x_1}(k)\\&=P_\nu(k)+l_{x_1}(k)\\&=P_\nu(k) -(P_{\mu_{x_2}}(k)-P_{\mu_{x_1}}(k))+l_{x_1}(k),
\end{align*}
where we used that $(P_{\mu_{x_2}}-P_{\mu_{x_1}})=0$ on $(-\infty,G_\mu(F_\mu(x_1)+)]$. Now suppose $x_1<k<x_2$. Then
\begin{align*}
(P_\nu(k)-(P_{\mu_{x_2}}(k)-P_{\mu_{x_1}}(k))&=(P_\nu(k)-(P_{\mu}(k)-P_{\mu_{x_1}}(k))\\&=D_{\mu,\nu}(k)+P_{\mu_{x_1}}(k)=D_{\mu,\nu}(k)-l_{x_1}(k)=\sE_{x_1,x_2}-l_{x_1}(k),
\end{align*}
as required. Finally, if $k\geq x_2$, then
\begin{align*}
\sE_{x_1,x_2}(k)&=P_\nu(k)-P_{\mu_{x_2}}(k)\\&=(P_\nu(k)-(P_{\mu_{x_2}}(k)-P_{\mu_{x_1}}(k))-P_{\mu_{x_1}}(k)=(P_\nu(k)-(P_{\mu_{x_2}}(k)-P_{\mu_{x_1}}(k))+l_{x_1}(k).
\end{align*}
\end{proof}

\begin{cor}\label{cor:E_vu}
Fix $-\infty\leq x_1\leq x_2\leq \infty$. Then $\sE^c_{x_1,x_2}$ is continuously differentiable. Moreover, for any $k\in\R$ with $\sE_{x_1,x_2}(k)=\sE^c_{x_1,x_2}(k)$, we have that $\sE'_{x_1,x_2}(k)=(\sE^c_{x_1,x_2})'(k)$.
\end{cor}
\begin{proof}
	Since $(\nu-S^\nu(\mu_{x_2}-\mu_{x_1}))\leq \nu$, $(\nu-S^\nu(\mu_{x_2}-\mu_{x_1}))$ is also continuous, and therefore $P_{\nu-S^\nu(\mu_{x_2}-\mu_{x_1})}$ is continuously differentiable. Then Lemma \ref{lem:E_vu} implies the desired differentiability of $\sE^c_{x_1,x_2}$.
	
	Now suppose that $k\in\R$ is such that $\sE_{x_1,x_2}(k)=\sE^c_{x_1,x_2}(k)$. Since $\sE_{x_1,x_2}\geq\sE^c_{x_1,x_2}$ on $\R$, $\sE'_{x_1,x_2}(k-)\leq(\sE^c_{x_1,x_2})'(k)\leq\sE'_{x_1,x_2}(k+)$. But if one (or more) of the inequalities is strict, then the definition of $\sE_{x_1,x_2}$ implies that $\nu(\{k\})>0$, a contradiction.
\end{proof}

\begin{cor}\label{cor:X-Z+} Fix $-\infty\leq x_1\leq x_2\leq \infty$. Then
$\sE^c_{x_1,x_2}$ is linear on $\R\cap[X^-_{\sE_{x_1,x_2}}(z),Z^+_{\sE_{x_1,x_2}}(z)]$, $z\in\R$.
\end{cor}
\begin{proof}
	Since, by Corollary \ref{cor:E_vu}, $(\sE^c_{x_1,x_2})'(z-)=(\sE^c_{x_1,x_2})'(z+)$ for each $z\in\R$, the claim follows immediately from the definitions of $X^-_{\sE_{x_1,x_2}}(z)$ and $Z^+_{\sE_{x_1,x_2}}(z)$.
\end{proof}

Our ultimate goal is to define a pair of functions $(M,N) = (M(k),N(k))_{k\in\R}$ such that $(M,N)$ define an injective lifted martingale coupling. See Figure~\ref{fig:MN}. We will do this by defining $(M,N)$ on a sequence of domains, whose union is a set of full mass with respect to $\mu$. In this section we will give some preliminary results which are applicable to a single domain. 

{We begin by defining a family of functions $\overrightarrow{m}_{\cdot,\cdot},\overrightarrow{n}_{\cdot,\cdot}$ which will later be used to define $(M,N)$. See Figure~\ref{fig:MN}.}

		\begin{figure}[H]

		\centering
	
		\begin{tikzpicture}
		
		\begin{axis}[%
		width=4.521in,
		height=3.566in,
		at={(0.758in,0.481in)},
		scale only axis,
		xmin=-10,
		xmax=11,
		ymin=-0.4,
		ymax=1,
		axis line style={draw=none},
		ticks=none
		]
		\draw[thin] (-10,0)--(10,0);
		
		\draw[dashed, gray] (-6,0.2) -- (-6,0);
		\node[below, scale=0.8] at (-6,0) {$x_1$};
		
			\draw[dashed, gray] (-4,0.33) -- (-4,0);
		\node[below, scale=0.8] at (-4,0) {$x_2$};
		
		\draw[dashed, gray] (-1,0.3) -- (-1,0);
		\node[below, scale=0.8] at (-1,0) {$u_1$};
		
		\draw[dashed, gray] (2,0.13) -- (2,0);
		\node[below, scale=0.8] at (2,0) {$u_2$};
		
			\node[below, scale=0.8] at (-3.5,0.7) {$k\mapsto \sE_{x_1,x_2}(k)$};
				\node[below, scale=0.8] at (2,0.7) {$k\mapsto \sE_{x_1,u_1}(k)$};
					\node[below, scale=0.8] at (9.4,0.7) {$k\mapsto \sE_{x_1,u_2}(k)$};
					\node[below, scale=0.8] at (9.4,0.2) {$k\mapsto D_{\mu,\nu}(k)$};
			
		\addplot [color=black, line width=1.0pt, forget plot]
		table[row sep=crcr]{%
			-10	0\\
			-9.58289512426709	0.0021747064911354\\
			-9.09090909090909	0.0103305798596145\\
			-8.61501384398931	0.0239786697604272\\
			-8.18181818181818	0.0413223203462037\\
			-7.73583407671843	0.0640805747071888\\
			-7.27272727272727	0.0929751936163919\\
			-6.77118450691904	0.130315567398738\\
			-6.36363636363636	0.165289309108455\\
			-6					0.2\\
		};

				\addplot [color=black, line width=1.0pt, forget plot]
			table[row sep=crcr]{%
			2					0.133\\
			2.30915892883115	0.124222275606183\\
			2.72727272727273	0.115700963849885\\
			3.17401811478308	0.111425947075356\\
			3.63636363636364	0.11225995377717\\
			4.09389341851855	0.118157016020043\\
			4.54545454545454	0.123276860436698\\
			4.77290290355786	0.124570566100095\\
			5.00035126166118	0.125011797087635\\
			5.11389980988225	0.124890540141315\\
			5.22744835810332	0.124571780674573\\
			5.28422263221385	0.124317105835142\\
			5.34099690632439	0.124029488898184\\
			5.36938404337965	0.123862435416248\\
			5.38357761190729	0.123770376558618\\
			5.39777118043492	0.123641513886\\
			5.40486796469874	0.123636119769622\\
			5.41196474896255	0.123583068823564\\
			5.41551314109446	0.123561224863682\\
			5.41728733716042	0.123545994457059\\
			5.41906153322637	0.123537899254345\\
			5.44390027814973	0.123521145745828\\
			5.44478737618271	0.123448686636093\\
			5.45645053592806	0.12335290598088\\
			5.45740307661936	0.123256789702043\\
			5.45835561731066	0.123253536298351\\
			5.46026069869326	0.123232148158654\\
			5.46216578007586	0.123216828939791\\
			5.46597594284106	0.123189911190704\\
			5.46978610560626	0.123160847124307\\
			5.47740643113667	0.1223097452537767\\
			5.48502675666707	0.122075226858961\\
			6.36363636363636	0.109504875562451\\
			6.58018947110948	0.104194179261455\\
			6.7967425785826	0.0980968078077034\\
		};

         \draw[black, line width=1.0pt] (6.7967425785826, 0.0980968078077034) to[out=350, in=180] (10,0);

		\addplot [color=blue, dashdotted, line width=1.0pt, forget plot]
		table[row sep=crcr]{%
		-4	0.334490372228037\\
			-3.87242897848034	0.344256777865846\\
			-3.63636363636364	0.356335997024029\\
			-3.38283822498829	0.372411440859339\\
			-3.12931281361294	0.39170097049629\\
			-2.92829277044283	0.409280676686708\\
			-2.72727272727273	0.428879748325416\\
			-2.48352168796022	0.455355497407544\\
			-2.23977064864772	0.484801786390314\\
			-2.02897623341477	0.512664773268342\\
			-1.81818181818182	0.54274548268606\\
			-1.33974520189821	0.619273984155864\\
			-0.909090909090909	0.697934715939885\\
			-0.458761721180472	0.79011372068568\\
			0	0.894442622478866\\
		};
	
		\addplot [color=black!30!green, dashed, line width=1.0pt, forget plot]
		table[row sep=crcr]{%
			-1	0.318223457712913\\
			-0.909090909090909	0.31346433415721\\
			-0.68392631513569	0.303871578045254\\
			-0.458761721180472	0.29681424324858\\
			-0.229380860590236	0.292231111554413\\
			0	0.290276914380949\\
			0.227332318762937	0.29093859940037\\
			0.454664637525873	0.294182380741105\\
			0.681877773308391	0.300006516236132\\
			0.909090909090909	0.308414319874636\\
			1.1453974636039	0.319894132270746\\
			1.38170401811689	0.334168013442113\\
			1.59994291814936	0.349829938490292\\
			1.81818181818182	0.367870942512977\\
			2.06367037350648	0.391013711852386\\
			2.30915892883115	0.41716828022349\\
			2.51821582805194	0.441817287764421\\
			2.72727272727273	0.468650868764268\\
			3.17401811478308	0.533318600877411\\
			3.63636363636364	0.61075696588341\\
			4.09389341851855	0.697962925831238\\
			4.54545454545454	0.794178222265186\\
			5.00035126166118	0.901472159080861\\
		};
		\addplot [color=red,dotted, line width=1.0pt, forget plot]
		table[row sep=crcr]{%
			-10	0.400000066355534\\
			-9.79144756213354	0.359920408976741\\
			-9.58289512426709	0.32310304597337\\
			-9.33690210758809	0.283869073183637\\
			-9.09090909090909	0.249173540451596\\
			-8.8529614674492	0.219933092015665\\
			-8.61501384398931	0.194936064630922\\
			-8.50671492844653	0.18498459146072\\
			-8.39841601290375	0.175873330315903\\
			-8.29011709736097	0.167662224165566\\
			-8.18181818181818	0.160330555224466\\
			-8.07032215554324	0.153701590417529\\
			-7.95882612926831	0.148004884473334\\
			-7.84733010299337	0.143240573454677\\
			-7.73583407671843	0.139408633568469\\
			-7.62005737572064	0.136416137781256\\
			-7.50428067472285	0.13442917776788\\
			-7.38850397372506	0.133447524772796\\
			-7.27272727272727	0.133471063654861\\
			-7.14734158127521	0.134653251651527\\
			-7.02195588982316	0.136969198594515\\
			-6.8965701983711	0.140486872584341\\
			-6.77118450691904	0.145183715844862\\
			-6.66929747109837	0.149872873796319\\
			-6.5674104352777	0.15533227544789\\
			-6.46552339945703	0.161574315415873\\
			-6.36363636363636	0.168595092338883\\
			-6.23698628624693	0.178407518301184\\
			-6.11033620885749	0.189422838029598\\
			-5.85703605407862	0.213984060805278\\
			-5.45454545454545	0.249999897780453\\
			-5.2352891080966	0.267537017816371\\
			-5.01603276164775	0.283604827931299\\
			-4.78074365355115	0.299214111819187\\
			-4.54545454545454	0.311755236319915\\
			-4.32697443302579	0.324539922451977\\
			-4.10849432059704	0.334490372228037\\
			-3.87242897848034	0.343601282497902\\
			-3.63636363636364	0.351010010420574\\
			-3.38283822498829	0.357070094871146\\
			-3.12931281361294	0.361166531032035\\
			-2.92829277044283	0.363019425902994\\
			-2.72727272727273	0.363636422201081\\
			-2.48352168796022	0.36272843649788\\
			-2.23977064864772	0.360004825788297\\
			-2.02897623341477	0.356187667436133\\
			-1.81818181818182	0.351009993159971\\
			-1.57896351004001	0.343492051124193\\
			-1.33974520189821	0.334223457712913\\
			-1.12441805549456	0.32438511856772\\
			-0.909090909090909	0.313131459903253\\
			-0.458761721180472	0.285015318067985\\
			0	0.249999136603171\\
			0.454664637525873	0.21469482694629\\
			0.909090909090909	0.184573597760398\\
			1.38170401811689	0.158722362127491\\
			1.59994291814936	0.148670280326231\\
			1.81818181818182	0.139806151212122\\
			2.06367037350648	0.131312539992621\\
			2.30915892883115	0.125417015697178\\
			2.51821582805194	0.122772397899856\\
			2.72727272727273	0.122312530130396\\
			2.95064542102791	0.124237674393088\\
			3.17401811478308	0.128654923964204\\
			3.40519087557336	0.135857575036822\\
			3.63636363636364	0.145731022350523\\
			3.8651285274411	0.158129710079006\\
			4.09389341851855	0.173145559646625\\
			4.31967398198655	0.190530146717026\\
			4.54545454545454	0.210466904145501\\
			4.77290290355786	0.233128330458324\\
			5.00035126166118	0.258371477639871\\
			5.22744835810332	0.286165247096235\\
			5.28422263221385	0.293514390356015\\
			5.34099690632439	0.30102173185397\\
			5.36938404337965	0.304838791757942\\
			5.39777118043492	0.308655654348015\\
			5.7					0.35\\
			5.99080491043187	0.402148341270755\\
			5.99162793792239	0.401954508866035\\
			5.99327399290343	0.401569386360698\\
			5.99492004788446	0.401187237803638\\
			5.99821215784653	0.400416026439283\\
			6.08627609933192	0.414565968621003\\
			6.08709912682244	0.414711662895201\\
			6.08792215431296	0.414846522034477\\
			6.08874518180347	0.414991475579689\\
			6.09039123678451	0.41527318999262\\
			6.09368334674658	0.415836295944068\\
			6.10026756667072	0.416961801565813\\
			6.113436006519	0.419227636022212\\
			6.12660444636728	0.421497831339883\\
			6.15294132606385	0.426067823173405\\
			6.20561508545698	0.43532248459322\\
			6.25828884485011	0.444715501700503\\
			6.36363636363636	0.463912548975129\\
			6.58018947110948	0.5051179133238\\
			6.7967425785826	0.548659549398957\\
			7.27272727272727	0.652616316281482\\
			7.74970618458635	0.768151583389662\\
			8.18181818181818	0.882649262423928\\
		};	
		\draw[black!30!green] (-7.,.136)--(0.66,0.297);
	\draw[gray,dashed] (-7.,.136)--(-7,-0.15);
	\node[below, scale=0.8] at (-6,-0.15) {$\protect\overrightarrow{m}_{x_1,x_2}(u_1)$};
	\draw[gray,dashed] (0.66,.297)--(0.66,-0.15);
	\node[below, scale=0.8] at (0.66,-0.15) {$\protect\overrightarrow{n}_{x_1,x_2}(u_1)$};
	
	\draw[red] (-7.4,.13)--(2.67,0.12);	
	\draw[gray,dashed] (-7.4,.13)--(-7.4,-0.25);
	\node[below, scale=0.8] at (-7.4,-0.25) {$\protect\overrightarrow{m}_{x_1,x_2}(u_2)$};
	\draw[gray,dashed] (2.67,0.12)--(2.67,-0.02);
	\draw[gray,dashed] (2.67,-0.11)--(2.67,-0.25);
	\node[below, scale=0.8] at (2.67,-0.25) {$\protect\overrightarrow{n}_{x_1,x_2}(u_2)$};
		
		\end{axis}
		\end{tikzpicture}%
				\caption{The construction of $\protect\overrightarrow{m}_{x_1,x_2}$ and $\protect\overrightarrow{n}_{x_1,x_2}$. For $x_1<x_2<u_1<u_2$, the dotted curve represents $\sE_{x_1,u_2}$, the dashed curve corresponds to $\sE_{x_1,u_1}$, while the dash-dotted curve represents $\sE_{x_1,x_2}$. The solid curve corresponds to $D_{\mu,\nu}$. Note that $D_{\mu,\nu}\leq\sE_{x_1,x_2}\leq\sE_{x_1,u_1}\leq\sE_{x_1,u_2}$ everywhere, $\sE_{x_1,w}=D_{\mu,\nu}$ on $[x_1,w]$ for $w\in\{x_2,u_1,u_2\}$, and $\sE_{x_1,u_i}=\sE_{x_1,x_2}$ on $(-\infty,x_1]$ for $i=1,2$. Furthermore, for $i=1,2$, the straight line going through $(\protect\overrightarrow{m}_{x_1,x_2}(u_i),\sE_{x_1,u_i}(\protect\overrightarrow{m}_{x_1,x_2}(u_i)))$ and $(\protect\overrightarrow{n}_{x_1,x_2}(u_i),\sE_{x_1,u_i}(\protect\overrightarrow{n}_{x_1,x_2}(u_i)))$ corresponds to the linear section on $[\protect\overrightarrow{m}_{x_1,x_2}(u_i),\protect\overrightarrow{n}_{x_1,x_2}(u_i)]$ of the convex hull $\sE^c_{x_1,u_i}$ of $\sE_{x_1,u_i}$. In particular, $\protect\overrightarrow{m}_{x_1,x_2}(u_2)\leq\protect\overrightarrow{m}_{x_1,x_2}(u_1)\leq\protect\overrightarrow{n}_{x_1,x_2}(u_1)\leq\protect\overrightarrow{n}_{x_1,x_2}(u_2)$.}
			\label{fig:MN}
		\end{figure}

	\begin{defn}\label{def:mn_v}
		Fix $-\infty\leq x_1\leq x_2< \infty$. Define $\overrightarrow{m}_{x_1,x_2},\overrightarrow{n}_{x_1,x_2}:\R\cap[x_2,\infty)\to\R$ by
		$$
	\overrightarrow{m}_{x_1,x_2}(l)=X^-_{\sE_{x_1,l}}(l)\quad\textrm{and}\quad\overrightarrow{n}_{x_1,x_2}(l)=Z^+_{\sE_{x_1,l}}(l),\quad l\in\R\cap[x_2,\infty).	$$
	\end{defn}
	\begin{lem}\label{lem:E_vw-linear}
		Fix $-\infty\leq x_1\leq x_2\leq l<\infty$. Then, for all {$w\in[l,\beta_\mu]$}, $S^\nu(\mu_w-\mu_{x_1})=\nu$ on $(\overrightarrow{m}_{x_1,x_2}(l),\overrightarrow{n}_{x_1,x_2}(l))$ and $\sE^c_{v,w}$ is linear on $\R\cap[\overrightarrow{m}_{x_1,x_2}(l),\overrightarrow{n}_{x_1,x_2}(l)]$.
	\end{lem}
	\begin{proof}
If $\overrightarrow{m}_{x_1,x_2}(l) = \overrightarrow{n}_{x_1,x_2}(l)$ then there is nothing to prove so we may assume $\overrightarrow{m}_{x_1,x_2}(l)<\overrightarrow{n}_{x_1,x_2}(l)$. By the associativity of the shadow measure (see Beiglb\"ock and Juillet \cite[Theorem 4.8]{BeiglbockJuillet:16}), for $x_1\leq x_2\leq l\leq w$, $S^\nu(\mu_w - \mu_{x_1}) = S^\nu(\mu_l - \mu_{x_1}) + S^{\nu - S^\nu(\mu_l - \mu_{x_1})}(\mu_w - \mu_l)$ and therefore $S^\nu(\mu_l - \mu_{x_1}) \leq S^\nu(\mu_w-\mu_{x_1}) \leq \nu$. Hence,
$$
		\nu-S^\nu(\mu_w-\mu_{x_1})\leq \nu - S^\nu(\mu_l-\mu_{x_1}).
		$$
{By Corollary~\ref{cor:X-Z+},} $\sE^c_{x_1,l}$ is linear on $[\overrightarrow{m}_{x_1,x_2}(l),\overrightarrow{n}_{x_1,x_2}(l)]$, {and then, see Lemma \ref{lem:E_vu},} $\nu - S^\nu(\mu_l-\mu_{x_1})$ (and thus also $\nu-S^\nu(\mu_w-\mu_{x_1})$) does not charge $(\overrightarrow{m}_{x_1,x_2}(l),\overrightarrow{n}_{x_1,x_2}(l))$. Hence, by Lemma \ref{lem:E_vu} again, $\sE^c_{x_1,w}$ is linear on $[\overrightarrow{m}_{x_1,x_2}(l),\overrightarrow{n}_{x_1,x_2}(l)]$. Moreover, since $(\nu-S^\nu(\mu_w-\mu_{x_1}))((\overrightarrow{m}_{x_1,x_2}(l),\overrightarrow{n}_{x_1,x_2}(l)))=0$ it follows that $S^\nu(\mu_w-\mu_{x_1})((\overrightarrow{m}_{x_1,x_2}(l),\overrightarrow{n}_{x_1,x_2}(l)))=\nu((\overrightarrow{m}_{x_1,x_2}(l),\overrightarrow{n}_{x_1,x_2}(l)))$. Then, since $S^\nu(\mu_w-\mu_{x_1})\leq\nu$ we conclude that $S^\nu(\mu_w-\mu_{x_1})=\nu$ on $(\overrightarrow{m}_{x_1,x_2}(l),\overrightarrow{n}_{x_1,x_2}(l))$.
	\end{proof}

\begin{cor}\label{cor:mn1_bounds}
	Fix $-\infty\leq x_1\leq x_2\leq l\leq \infty$. If $l\in(\alpha_\mu,\beta_\mu)$ then $\alpha_\nu < \overrightarrow{m}_{x_1,x_2}(l)\leq \overrightarrow{n}_{x_1,x_2}(l)<\beta_\nu$.
\end{cor}
\begin{proof} Note that $\overrightarrow{m}_{x_1,x_2}(l)\leq l \leq \overrightarrow{n}_{x_1,x_2}(l)$ and {$\alpha_\nu \leq \alpha_\mu < l <  \beta_\mu \leq \beta_\nu$. Hence,} $\overrightarrow{m}_{x_1,x_2}(l)<\beta_\nu$ and $\alpha_\nu<\overrightarrow{n}_{x_1,x_2}(l)$.

{Observe that if $\overrightarrow{m}_{x_1,x_2}(l)\leq\alpha_\nu$ and $\beta_\nu\leq \overrightarrow{n}_{x_1,x_2}(l)$ then}
by Lemma \ref{lem:E_vw-linear}, $0=\nu(\R)-S^\nu(\mu_l-\mu_{x_1})(\R)=\mu(\R)-(\mu_l(\R)-\mu_{x_1}(\R))=F_\mu(x_1)+\mu(\R)-F_\mu(l)>F_\mu(x_1)\geq0$, a contradiction.
{Hence, in order to prove that $\alpha_\nu < \overrightarrow{m}_{x_1,x_2}(l)\leq \overrightarrow{n}_{x_1,x_2}(l)<\beta_\nu$ it is sufficient to show that neither $\alpha_\nu<\overrightarrow{m}_{x_1,x_2}(l)<\beta_\nu\leq\overrightarrow{n}_{x_1,x_2}(l)$ nor $\overrightarrow{m}_{x_1,x_2}(l)\leq\alpha_\nu<\overrightarrow{n}_{x_1,x_2}(l)<\beta_\nu$.}
	
	Suppose $\alpha_\nu<\overrightarrow{m}_{x_1,x_2}(l)<\beta_\nu\leq\overrightarrow{n}_{x_1,x_2}(l)$. Then by Lemma \ref{lem:E_vw-linear} and the continuity of $\nu$ we have that $\nu-S^\nu(\mu_l-\mu_{x_1})$ is concentrated on $(\alpha_\nu,\overrightarrow{m}_{x_1,x_2}(l))\subseteq (\alpha_\nu,l)$. On the other hand, by the associativity of the shadow measure $\mu_{x_1}+(\mu-\mu_l)=\mu-(\mu_l-\mu_{x_1})\leq_{cx}\nu-S^\nu(\mu_l-\mu_{x_1})$. But, since $l\in(\alpha_\mu,\beta_\mu)$, $(\mu-\mu_l)(\R)>0$ and $\mu-\mu_l$ is concentrated on $[l,\infty)$. It follows that $\mu-\mu_l$ cannot be embedded in $\nu-S^\nu(\mu_l-\mu_{x_1})$ in a way which respects the martingale property. {\em A fortiori}, $\mu_{x_1} + (\mu - \mu_l)$ cannot be embedded in $\nu-S^\nu(\mu_l-\mu_{x_1})$ in a way which respects the martingale property, a contradiction to the fact that $\mu-(\mu_l-\mu_{x_1})\leq_{cx}\nu-S^\nu(\mu_l-\mu_{x_1})$. We conclude that $\overrightarrow{n}_{x_1,x_2}(l)<\beta_\nu$ whenever $\alpha_\nu<\overrightarrow{m}_{x_1,x_2}(l)$.
	
Now suppose that $\overrightarrow{m}_{x_1,x_2}{(l)} \leq\alpha_\nu<\overrightarrow{n}_{x_1,x_2}(l)<\beta_\nu$. Then, {by Lemma~\ref{lem:E_vw-linear},} $\nu-S^\nu(\mu_l-\mu_{x_1})$ is concentrated on $(\overrightarrow{n}_{x_1,x_2}(l),\beta_\nu)$. Suppose that $x_1>\alpha_\mu$. Then $\mu_{x_1}(\R)=F_\mu(x_1)>0$ and $\mu_{x_1}$ is concentrated on $(-\infty,x_1]\subseteq(-\infty,\overrightarrow{n}_{x_1,x_2}(l)]$. In this case $\mu_{x_1}$ cannot be embedded in $\nu-S^\nu(\mu_l-\mu_{x_1})$ in a martingale way, a contradiction. Finally suppose that $x_1=\alpha_\mu$, so that $\mu-(\mu_l-\mu_{x_1})=\mu-\mu_l$ and $\sE_{x_1,l}=\sE_{\alpha_\mu,l}=D_{\mu,\nu}$ on $(-\infty, \overrightarrow G_\mu(F_\mu(l)+)]$.
Then, since $\overrightarrow{m}_{x_1,x_2}\leq\alpha_\nu$, we have that $\sE^c_{\alpha_\mu,l}(\overrightarrow{m}_{x_1,x_2})=D_{\mu,\nu}(\overrightarrow{m}_{x_1,x_2})=0$ (or $\lim_{k\downarrow \overrightarrow{m}_{x_1,x_2}}\sE^c_{\alpha_\mu,l}(k)=\lim_{k\downarrow \overrightarrow{m}_{x_1,x_2}}D_{\mu,\nu}(k)=0$ in the case $ \overrightarrow{m}_{x_1,x_2}=-\infty$) and  $(\sE^c_{\alpha_\mu,l})'(\overrightarrow{m}_{x_1,x_2})=D_{\mu,\nu}'(\overrightarrow{m}_{x_1,x_2})=0$ (or $(\sE^c_{\alpha_\mu,l})'(\overrightarrow{m}_{x_1,x_2}+)=0$ in the case $ \overrightarrow{m}_{x_1,x_2}=-\infty$). But then $D_{\mu,\nu}(\overrightarrow{n}_{x_1,x_2}(l))=\sE_{\alpha_\mu,l}(\overrightarrow{n}_{x_1,x_2}(l))=\sE^c_{\alpha_\mu,l}(\overrightarrow{n}_{x_1,x_2}(l))=0$, a contradiction since $\overrightarrow{n}_{x_1,x_2}(l)<\beta_\nu$ and $D_{\mu,\nu}>0$ on $(\alpha_\nu,\beta_\nu)$.

\end{proof}
	\begin{cor}\label{cor:n_vu-increasing} Fix $-\infty\leq x_1\leq x_2< \infty$.
		$\overrightarrow{n}_{x_1,x_2}(\cdot)$ is non-decreasing on $\R\cap[x_2,\infty)$.
	\end{cor}
	\begin{proof}
		Note that {since $\overrightarrow{n}_{x_1,x_2}$ only depends on $x_2$ via the domain on which it is defined,} it is enough to prove the claim for $x_2=x_1$.

  If $x_1=x_2=-\infty$, then $\sE_{x_1,x_2}=P_\nu$. By the convexity of $P_\nu$ and the definition of $\overrightarrow{n}_{x_1,x_2}$ we then immediately have that $\overrightarrow{n}_{x_1,x_2}$ is non-decreasing on $\R$.
  
  Let $-\infty<x_1= x_2\leq w_1\leq w_2<\infty$. Either $w_2 \geq \overrightarrow{n}_{x_1,x_1}(w_1)$ and then $\overrightarrow{n}_{x_1,x_1}(w_1) \leq w_2 \leq \overrightarrow{n}_{x_1,x_1}(w_2)$ or $w_2 < \overrightarrow{n}_{x_1,x_1}(w_1)$ and then $\overrightarrow{m}_{x_1,x_1}(w_1) \leq w_1 \leq w_2 < \overrightarrow{n}_{x_1,x_1}(w_1)$. In the latter case, since $[\overrightarrow{m}_{x_1,x_1}(w_2),\overrightarrow{n}_{x_1,x_1}(w_2)]$ is the largest interval containing $w_2$ on which $\sE^c_{x_1,w_2}$ is linear, and since (by Lemma~\ref{lem:E_vw-linear}) $\sE^c_{x_1,w_2}$ is linear on $[\overrightarrow{m}_{x_1,x_1}(w_1),\overrightarrow{n}_{x_1,x_1}(w_1)]$ we must have $\overrightarrow{m}_{x_1,x_1}(w_2) \leq \overrightarrow{m}_{x_1,x_1}(w_1) \leq \overrightarrow{n}_{x_1,x_1}(w_1) \leq \overrightarrow{n}_{x_1,x_1}(w_2)$. Hence in both cases we have $\overrightarrow{n}_{x_1,x_1}(w_1) \leq \overrightarrow{n}_{x_1,x_1}(w_2)$ so that $\overrightarrow{n}_{x_1,x_1}$ is non-decreasing.
	\end{proof}

\begin{lem}\label{lem:mvu-decreasingPrelim}
	Let $x_1,x_2\in\R\cup\{-\infty\}$ and ${r_1,r_2}\in\R$ be such that $x_1\leq x_2 \leq r_1 \leq r_2$. If $\overrightarrow{m}_{x_1,x_2}(r_2) \leq \overrightarrow{n}_{v,u}(r_1)$ then $\overrightarrow{m}_{x_1,x_2}(r_2) \leq \overrightarrow{m}_{x_1,x_2}(r_1)$.
\end{lem}
\begin{proof}
	If 
$r_1=r_2$ then there is nothing to prove. So suppose  $r_1<r_2$.
	
	
If $\overrightarrow{m}_{x_1,x_2}(r_1)=\overrightarrow {n}_{v,u}(r_1)$ then, by hypothesis, $\overrightarrow{m}_{x_1,x_2}(r_2)\leq \overrightarrow{n}_{x_1,x_2}(r_1)=\overrightarrow{m}_{x_1,x_2}(r_1)$ and we are done. Hence in the rest of the proof we suppose that $\overrightarrow{m}_{x_1,x_2}(r_1)<\overrightarrow {n}_{x_1,x_2}(r_1)$.

If $\overrightarrow{m}_{x_1,x_2}(r_2)=r_2$ then either $\overrightarrow{m}_{x_1,x_2}(r_2)=r_2<\overrightarrow{m}_{x_1,x_2}(r_1)$ (and there is nothing to prove) or $\overrightarrow{m}_{x_1,x_2}(r_1)\leq r_2= \overrightarrow{m}_{x_1,x_2}(r_2)\leq \overrightarrow{n}_{x_1,x_2}(r_1)$. In the latter case, since $[\overrightarrow{m}_{x_1,x_2}(r_2),\overrightarrow{n}_{x_1,x_2}(r_2)]$ is the largest interval (containing $r_2$) on which $\sE^c_{x_1,r_2}$ is linear (note that, by Lemma \ref{lem:E_vw-linear},  $\sE^c_{x_1,r_2}$ is linear on $[\overrightarrow{m}_{x_1,x_2}(r_1),\overrightarrow{n}_{x_1,x_2}(r_1)]$), we conclude that $\overrightarrow{m}_{x_1,x_2}(r_2)\leq \overrightarrow{m}_{x_1,x_2}(r_1)$. Hence in the rest of the proof we suppose that $\overrightarrow{m}_{x_1,x_2}(r_2)<r_2$, so that $\overrightarrow{m}_{x_1,x_2}(r_2)<\overrightarrow {n}_{v,u}(r_2)$.
	
We now have that $\sE^c_{x_1,r_2}$ is linear on both $[\overrightarrow{m}_{x_1,x_2}(r_1),\overrightarrow{n}_{x_1,x_2}(r_1)]$ and $[\overrightarrow{m}_{x_1,x_2}(r_2),\overrightarrow{n}_{x_1,x_2}(r_2)]$ where (by Corollary \ref{cor:n_vu-increasing})
	$\overrightarrow{n}_{x_1,x_2}(r_1) \leq \overrightarrow{n}_{x_1,x_2}(r_2)$. Since $\overrightarrow{m}_{x_1,x_2}(r_2) \leq \overrightarrow{n}_{x_1,x_2}(r_1)$ these intervals either overlap or meet at a point. In the latter case,
	since $\nu$ is atom-less, the slopes of both linear sections of $\sE^c_{x_1,r_2}$ must be the same at $\overrightarrow{m}_{x_1,x_2}(r_2) \equiv \overrightarrow{n}_{x_1,x_2}(r_1)$ (i.e., $\sE^c_{x_1,r_2}$ cannot have kink at $\overrightarrow{m}_{x_1,x_2}(r_2)$). Therefore in both cases $\sE^c_{x_1,r_2}$ is linear on $[\overrightarrow{m}_{x_1,x_2}(r_1),\overrightarrow{n}_{x_1,x_2}(r_2)]$. But $r_2\in[\overrightarrow{m}_{x_1,x_2}(r_1),\overrightarrow{n}_{x_1,x_2}(r_2)]$, and therefore we must have that $\overrightarrow{m}_{x_1,x_2}(r_2)\leq \overrightarrow{m}_{x_1,x_2}(r_1)$.
\end{proof}

	For $x_1,x_2\in \R\cup\{-\infty\}$ with $x_1\leq x_2$ define
\[ \overline{w}_{x_1,x_2}:=\inf\{w\in(x_2,\beta_\mu):\overrightarrow{n}_{x_1,x_2}(w)\leq \overrightarrow G_\mu(F_\mu(w)+)\}\wedge\beta_\mu \]
with convention $\inf\emptyset=\infty$. 	
\begin{lem}\label{lem:w_vu>u}
Suppose $-\infty\leq x_1\leq x_2<\infty$ with $x_2>-\infty$. If $\overrightarrow{n}_{x_1,x_2}(x_2)>\overrightarrow{G}_\mu(F_\mu(x_2)+)$ then $\overline{w}_{x_1,x_2}>x_2$
{and moreover $F_\mu(\overline{w}_{x_1,x_2})>F_\mu(x_2)$.}

	\end{lem}
\begin{proof}
	Since $x_1\leq\overrightarrow{G}_\mu(F_\mu(x_2)+)<\overrightarrow{n}_{x_1,x_2}(x_2)$, by the right-continuity of $\overrightarrow{G}_\mu(\cdot+)$ (also note that $x\mapsto F_\mu(x)$ is continuous, since $\mu$ is atom-less) there exists $\overline\epsilon>0$ such that $F_\mu(x_2+\overline\epsilon)>F_\mu(x_2)$, and for all $\epsilon\in {( }0,\overline\epsilon]$
	$$
	\overrightarrow{G}_\mu(F_\mu(x_2)+)\leq \overrightarrow{G}_\mu(F_\mu(x_2+\epsilon))\leq \overrightarrow{G}_\mu(F_\mu(x_2+\epsilon)+)<\overrightarrow{n}_{x_1,x_2}(x_2).
	$$
Then, by the monotonicity of $\overrightarrow{n}_{x_1,x_2}$,  for $\epsilon \in [0,\overline\epsilon]$ we have $\overrightarrow{G}_\mu(F_\mu(x_2+\epsilon)+)<\overrightarrow{n}_{x_1,x_2}(x_2) \leq\overrightarrow{n}_{x_1,x_2}(x_2+\epsilon)$ and therefore $\overline{w}_{x_1,x_2}\geq x_2+\epsilon>x_2$.

{Clearly, if $\overline{w}_{x_1,x_2}>x_2$ then $F_\mu(\overline{w}_{x_1,x_2}) \geq F_\mu(x_2)$. Suppose there is equality. From the definition of $\overline{w}_{x_1,x_2}$, either $\overrightarrow{n}_{x_1,x_2}(\overline{w}_{x_1,x_2}) \leq \overrightarrow{G}_\mu(F_\mu(\overline{w}_{x_1,x_2})+)$ or there exists $w_n \downarrow \overline{w}_{x_1,x_2}$ such that $\overrightarrow{n}_{x_1,x_2}(w_n) \leq \overrightarrow{G}_\mu(F_\mu(w_n)+)$. In the former case we have $\overrightarrow{n}_{x_1,x_2}(x_2) > \overrightarrow{G}_\mu(F_\mu(x_2)+) = \overrightarrow{G}_\mu(F_\mu(\overline{w}_{x_1,x_2})+) \geq \overrightarrow{n}_{x_1,x_2}(\overline{w}_{x_1,x_2}) \geq \overrightarrow{n}_{x_1,x_2}(x_2)$ a contradiction. In the latter case we have 
$\overrightarrow{n}_{x_1,x_2}(x_2) > \overrightarrow{G}_\mu(F_\mu(x_2)+) = \lim_n \overrightarrow{G}_\mu(F_\mu(w_n)+) \geq \lim_n \overrightarrow{n}_{x_1,x_2}(w_n) \geq \overrightarrow{n}_{x_1,x_2}(x_2)$, which again is a contradiction. Hence $F_\mu(\overline{w}_{x_1,x_2}) > F_\mu(x_2)$.} 
\end{proof}

\begin{lem}
\label{lem:splitpart2}
    Suppose $-\infty \leq  x_1 \leq x_2<\overline{w}_{x_1,x_2}\leq \beta_\mu$. 
Then
			$
			 \overline{w}_{x_1,x_2}\leq \overrightarrow{n}_{x_1,x_2}(\overline{w}_{x_1,x_2})\leq \overrightarrow{G}_\mu(F_\mu(\overline{w}_{x_1,x_2})+),
			$ {and $D_{\mu,\nu}(\overrightarrow{n}_{x_1,x_2}(\overline{w}_{x_1,x_2})) = \sE_{x_1,x_2}(\overrightarrow{n}_{x_1,x_2}(\overline{w}_{x_1,x_2}))=\sE^c_{x_1,x_2}(\overrightarrow{n}_{x_1,x_2}(\overline{w}_{x_1,x_2}))$.}
\end{lem}
\begin{proof}
For the second part, suppose $-\infty \leq  x_1 \leq x_2<\overline{w}_{x_1,x_2}\leq \beta_\mu$. If $\overline{w}_{x_1,x_2}= \beta_\mu=\infty$, then $\overline{w}_{x_1,x_2}=\lim_{z\to\infty}z\leq \lim_{z\to\infty} \overrightarrow{n}_{x_1,x_2}(z)=:\overrightarrow{n}_{x_1,x_2}(\infty)\leq \infty=\lim_{z\to\infty}\overrightarrow{G}_\mu(F_\mu(z)+)=\overrightarrow{G}_\mu(F_\mu(\infty)+)$ as required (in fact, we have equality throughout). In the rest of the proof we suppose that $\overline{w}_{x_1,x_2}$ is finite.

That $\overline{w}_{x_1,x_2}\leq \overrightarrow{n}_{x_1,x_2}(\overline{w}_{x_1,x_2})$ is clear from the definition of $\overrightarrow{n}_{x_1,x_2}$, and thus we now prove that $\overrightarrow{n}_{x_1,x_2}(\overline{w}_{x_1,x_2}) \leq \overrightarrow{G}_\mu(F_\mu(\overline{w}_{x_1,x_2})+)$. This is immediate if $\overline{w}_{x_1,x_2}=\beta_\mu$ {since we have defined $G_\mu(\mu(\R)+) = \infty$ for any quantile function $G_\mu$ of $\mu$.  Suppose $\overline{w}_{x_1,x_2}<\beta_\mu$, so that $\{w\in(x_2,\beta_\mu):\overrightarrow{n}_{x_1,x_2}(w) \leq \overrightarrow{G}_\mu(F_\mu(w)+)\}\neq\emptyset$.  Either $\overrightarrow{n}_{x_1,x_2}(\overline{w}_{x_1,x_2}) \leq \overrightarrow{G}_\mu(F_\mu(\overline{w}_{x_1,x_2})+)$ and we are done, or there exists $(w_k)_{k\geq1}$ in $(x_2,\beta_\mu)$ such that $w_k\downarrow \overline{w}_{x_1,x_2}$ (as $k\uparrow\infty$) and $\overrightarrow{n}_{x_1,x_2}(w_k) \leq \overrightarrow{G}_\mu(F_\mu(w_k)+)$ for each $k\geq 1$. Then, by Corollary \ref{cor:n_vu-increasing}, $\overrightarrow{n}_{x_1,x_2}(\cdot)$ is monotonic increasing on $[x_2,\infty)\cap\R$ and we have that
$$ \overrightarrow{n}_{x_1,x_2}(\overline{w}_{x_1,x_2}) \leq \lim_{ k \uparrow \infty} \overrightarrow{n}_{x_1,x_2}(w_k) \leq \lim_{ k \uparrow \infty}
\overrightarrow{G}_\mu(F_\mu(w_k)+) = \overrightarrow{G}_\mu(F_\mu(\overline{w}_{x_1,x_2})+).
$$

The last assertion, i.e., $D_{\mu,\nu}(\overrightarrow{n}_{x_1,x_2}(\overline{w}_{x_1,x_2})) = \sE_{x_1,x_2}(\overrightarrow{n}_{x_1,x_2}(\overline{w}_{x_1,x_2}))=\sE^c_{x_1,x_2}(\overrightarrow{n}_{x_1,x_2}(\overline{w}_{x_1,x_2}))$,
follows immediately from the definitions of $D_{\mu,\nu},~\sE_{x_1,x_2},~\sE^c_{x_1,x_2},~\overrightarrow{n}_{x_1,x_2}$ and $\overline{w}_{x_1,x_2}$ (in the case $\overrightarrow{n}_{x_1,x_2}(\overline{w}_{x_1,x_2})=\infty$, we take limits and use the continuity of $D_{\mu,\nu},~ \sE_{x_1,x_2},~ \sE^c_{x_1,x_2}$).}

\end{proof}

{We have shown that $\overrightarrow{n}_{x_1,x_2}$ is increasing on $[x_2,\infty)$ and now we would like to show that $\overrightarrow{m}_{x_1,x_2}$ is decreasing. This is not true on $[x_2,\infty)$; however, we will show that it is true on $[x_2,\overline{w}_{x_1,x_2}]$. }
	\begin{lem}\label{lem:m_vu-decreasing}
		Fix $x_1,x_2\in\R\cup\{-\infty\}$ with $x_1\leq x_2$.  Then $\overrightarrow{m}_{x_1,x_2}(\cdot)$ is non-increasing on $[x_2,\overline{w}_{x_1,x_2}]\cap\R$. 
	\end{lem}
	
	\begin{proof}
{If $x_2= \overline{w}_{x_1,x_2}$ there is nothing to prove. So, suppose $x_2<\overline{w}_{x_1,x_2}$.}

Fix $l,w \in (x_2, \overline{w}_{x_1,x_2}]\cap\R$ with $l<w$. We show that $\overrightarrow{m}_{x_1,x_2}(w) \leq \overrightarrow{m}_{x_1,x_2}(l)$. Later we show the result can be extended to allow $l=x_2$.	
	
Suppose that $\overrightarrow{m}_{x_1,x_2}(w) \leq \overrightarrow{n}_{x_1,x_2}(l)$. Then by Lemma~\ref{lem:mvu-decreasingPrelim} with $r_1=l$ and $r_2=w$, we have $\overrightarrow{m}_{x_1,x_2}(w) \leq \overrightarrow{m}_{x_1,x_2}(l)$ as required.
		
The alternative is that $\overrightarrow{n}_{x_1,x_2}(l)<\overrightarrow{m}_{x_1,x_2}(w)$. We show that this case cannot happen by finding a contradiction. Define $I_{x_1,x_2,w} = \{k\in(l,w]:\overrightarrow{n}_{x_1,x_2}(l)<\overrightarrow{m}_{x_1,x_2}(k)\}$. Clearly, {since} $\overrightarrow{n}_{x_1,x_2}(l)<\overrightarrow{m}_{x_1,x_2}(w)$ {we have} $w \in I_{x_1,x_2,w}$.  Define $\tilde l:=\inf\{k: k \in I_{x_1,x_2,w} \}$. We show first that $\tilde{l} > l$, second that $\tilde{l} \notin I_{x_1,x_2,w}$ (so that $\tilde{l}<w$) and third that if $\tilde{l}<w$ then there exists $\tilde{\epsilon}>0$ such that $[\tilde{l}, \tilde{l}+ \tilde{\epsilon}) \cap I_{x_1,x_2,w} = \emptyset$. But, this contradicts the definition of $\tilde{l}$ as the infimum of elements of $I_{x_1,x_2,w}$.

So, suppose $\overrightarrow{n}_{x_1,x_2}(l)<\overrightarrow{m}_{x_1,x_2}(w)$. Since $l< \overline{w}_{x_1,x_2}$ we have that
$l\leq \overrightarrow G_\mu(F_\mu(l)+)<\overrightarrow{n}_{x_1,x_2}(l)$, and there exists $\bar\epsilon>0$ with $l+\bar\epsilon< w$ such that $\overrightarrow{m}_{x_1,x_2}(l+\epsilon)\leq l+\epsilon\leq \overrightarrow{n}_{x_1,x_2}(l)$ for all $\epsilon\in[0,\bar\epsilon]$. In particular, $\overrightarrow{m}_{x_1,x_2}(l+\epsilon)\leq\overrightarrow{n}_{x_1,x_2}(l)$ for all $\epsilon\in[0,\bar\epsilon]$, and it follows that $\tilde l \geq l + \bar{\epsilon} > l$.
		
Now we show that $\overrightarrow{m}_{x_1,x_2}(\cdot)$ is non-increasing on $[l,\tilde l)$. Let $l_1,l_2\in[l,\tilde l)$ with $l_1<l_2$.
Since $l_1<\overline{w}_{x_1,x_2}$ we have that $\overrightarrow{m}_{x_1,x_2}(l_1) \leq l_1 \leq \overrightarrow G_\mu(F_\mu(l_1)+) <\overrightarrow{n}_{x_1,x_2}(l_1)$ and, since $l_1 < \tilde{l}$,
$\overrightarrow{n}_{x_1,x_2}(l) \geq \overrightarrow{m}_{x_1,x_2}(l_1)$.
Then, applying Lemma~\ref{lem:mvu-decreasingPrelim} with $r_1=l$ and $r_2=l_1$ we have that
$\overrightarrow{m}_{x_1,x_2}(l_1)\leq \overrightarrow{m}_{x_1,x_2}(l)$. Since $l_2 < \tilde{l}$, similarly as for $l_1$, we have that $\overrightarrow{m}_{x_1,x_2}(l_2)\leq { \overrightarrow{m}_{x_1,x_2}(l) \leq} \overrightarrow{n}_{x_1,x_2}(l)$, and therefore $\overrightarrow{m}_{x_1,x_2}(l_2)\leq \overrightarrow{n}_{x_1,x_2}(l_1)$. Then, by Lemma~\ref{lem:mvu-decreasingPrelim} with $r_1=l_1$ and $r_2=l_2$, $\overrightarrow{m}_{x_1,x_2}(l_2)\leq \overrightarrow{m}_{x_1,x_2}(l_1)$ as required.
		
We now show that $\tilde l\notin I_{x_1,x_2,w}$. Note that $\sE^c_{x_1,\tilde l}$ is linear on $[\overrightarrow{m}_{x_1,x_2}(k),\overrightarrow{n}_{x_1,x_2}(k)]$ for all $k\in[l,\tilde l)$. By the monotonicity of $\overrightarrow{m}_{x_1,x_2}(\cdot)$ and $\overrightarrow{n}_{x_1,x_2}(\cdot)$ on $[l,\tilde l)$, by Lemma~\ref{lem:E_vw-linear} {and the fact that for each $k\in[l,\tilde l)$, $\nu-S^\nu(\mu_{ k}-\mu_{x_1})$ does not charge $(\overrightarrow{m}_{x_1,x_2}(k),\overrightarrow{n}_{x_1,x_2}(k))$)}
we have that $\sE^c_{x_1,\tilde l}$ is linear on $[\overrightarrow{m}_{x_1,x_2}(\tilde l-),\overrightarrow{n}_{x_1,x_2}(\tilde l-)]$. Since $\overrightarrow{m}_{x_1,x_2}(k)\leq k\leq \overrightarrow G_\mu(F_\mu(k)+)<\overrightarrow{n}_{x_1,x_2}(k)$ for all $k\in[l,\tilde l)$, $\overrightarrow{m}_{x_1,x_2}(\tilde l-)\leq \tilde l\leq\overrightarrow{n}_{x_1,x_2}(\tilde l-)$. But $[\overrightarrow{m}_{x_1,x_2}(\tilde l),\overrightarrow{n}_{x_1,x_2}(\tilde l)]$ is the largest interval (containing $\tilde l$) on which $\sE^c_{x_1,\tilde l}$ is linear. It follows that $\overrightarrow{m}_{x_1,x_2}(\tilde l)\leq\overrightarrow{m}_{x_1,x_2}(\tilde l-)\leq \overrightarrow{m}_{x_1,x_2}(l)<\overrightarrow{n}_{x_1,x_2}(l)$, and therefore $\tilde{l} \notin I_{x_1,x_2,w}$ and $\tilde l < w$.
		
		Finally, we have that
		$$
		\overrightarrow{m}_{x_1,x_2}(\tilde l)\leq \tilde l\leq \overrightarrow G_\mu(F_\mu(\tilde l)+)<\overrightarrow{n}_{x_1,x_2}(\tilde l).
		$$
		Since $\tilde{l}< w$, there exists $0<\tilde\epsilon< (w-\tilde l)$ such that, for all $\epsilon\in[0,\tilde \epsilon]$, $(\tilde l+\epsilon)\leq \overrightarrow{n}_{x_1,x_2}(\tilde l)$. But, since $[\overrightarrow{m}_{x_1,x_2}(\tilde l+\epsilon),\overrightarrow{n}_{x_1,x_2}(\tilde l+\epsilon)]$ is the largest interval (containing $(\tilde l+\epsilon)$) on which $\sE^c_{x_1,\tilde l+\epsilon}$ is linear, we have that $\overrightarrow{m}_{x_1,x_2}(\tilde l+\epsilon)\leq \overrightarrow{m}_{x_1,x_2}(\tilde l)<\overrightarrow{n}_{x_1,x_2}(\tilde l)$. Hence $[\tilde{l},\tilde{l}+\tilde{\epsilon}) \cap I_{x_1,x_2,w} = \emptyset$ and $\tilde{l}$ is not the infimum of elements of $I_{x_1,x_2,w}$.

  We now extend the claim to $[x_2,\overline{w}_{x_1,x_2}]\cap\R$. Since $\overrightarrow{m}_{x_1,x_2}(\cdot)$ is non-increasing on $(x_2,\overline{w}_{x_1,x_2}]\cap\R$, it is enough to show that $\overrightarrow{m}_{x_1,x_2}(x_2+)\leq\overrightarrow{m}_{x_1,x_2}(x_2)$. Suppose not, so that $\overrightarrow{m}_{x_1,x_2}(x_2)<\overrightarrow{m}_{x_1,x_2}(x_2+)$. Since $\overrightarrow{m}_{x_1,x_2}(z)\leq z$ for all $z\in(x_2,\overline{w}_{x_1,x_2}]\cap\R$, we have that $\overrightarrow{m}_{x_1,x_2}(x_2)<\overrightarrow{m}_{x_1,x_2}(x_2+)\leq x_2$. Then there exists $x_2<\tilde w\leq \overline{w}_{x_1,x_2}$ such that for all $w\in(x_2,\tilde w]$ we have that
  $$
  \overrightarrow{m}_{x_1,x_2}(x_2)<\overrightarrow{m}_{x_1,x_2}(w)\leq\overrightarrow{m}_{x_1,x_2}(x_2+)\leq x_2<w\vee \overrightarrow{n}_{x_1,x_2}(x_2)\leq \overrightarrow{n}_{x_1,x_2}(w).
  $$
  By Lemma \ref{lem:E_vw-linear}, $\sE^c_{x_1,w}$ is linear on both $[\overrightarrow{m}_{x_1,x_2}(x_2),\overrightarrow{n}_{x_1,x_2}(x_2)]$ and $[\overrightarrow{m}_{x_1,x_2}(w),\overrightarrow{n}_{x_1,x_2}(w)]$. But $\overrightarrow{m}_{x_1,x_2}(x_2)<\overrightarrow{m}_{x_1,x_2}(w)\leq x_2\leq \overrightarrow{n}_{x_1,x_2}(x_2)$, and thus $[\overrightarrow{m}_{x_1,x_2}(x_2),\overrightarrow{n}_{x_1,x_2}(x_2)]\cap[\overrightarrow{m}_{x_1,x_2}(w),\overrightarrow{n}_{x_1,x_2}(w)]=[\overrightarrow{m}_{x_1,x_2}(w),\overrightarrow{n}_{x_1,x_2}(x_2)]$. Since both measures $(\mu,\nu)$ are atom-less, $\sE^c_{x_1,w}$ is differentiable and therefore $\sE^c_{x_1,w}$ must be linear on $[\overrightarrow{m}_{x_1,x_2}(x_2),\overrightarrow{n}_{x_1,x_2}(w)]$. But $\overrightarrow{m}_{x_1,x_2}(x_2)<w\leq \overrightarrow{n}_{x_1,x_2}(w)$, and since $[\overrightarrow{m}_{x_1,x_2}(w),\overrightarrow{n}_{x_1,x_2}(w)]$ is the largest interval on which $\sE^c_{x_1,w}$ is linear, we must have that $\overrightarrow{m}_{x_1,x_2}(w)\leq \overrightarrow{m}_{x_1,x_2}(x_2)$, a contradiction.
  
	\end{proof}

\begin{lem}
		\label{lem:properties1}
		Suppose $-\infty\leq x_1\leq x_2<\overline{w}_{x_1,x_2}$. {Suppose that either} $x_1=x_2$, or $x_1<x_2$ and $\overrightarrow{m}_{x_1,x_2}(x_2+)\leq \overrightarrow G_\mu(F_\mu(x_1)+)$. Then  
		\begin{enumerate}
			\item[(i)] $S^\nu(\mu_l-\mu_{x_1})=\nu\lvert_{(\overrightarrow{m}_{x_1,x_2}(l),\overrightarrow{n}_{x_1,x_2}(l))}$, {for all $l \in (x_2,\overline{w}_{x_1,x_2}]$}.
{		\item[(ii)] Suppose that $r_1,r_2\in(x_2,\overline{w}_{x_1,x_2}]\cap\R$ with $r_1<r_2$.  
 Then $\overrightarrow{n}_{x_1,x_2}(r_1) \leq \overrightarrow{n}_{x_1,x_2}(r_2)$ and $\overrightarrow{m}_{x_1,x_2}(r_1) \geq \overrightarrow{m}_{x_1,x_2}(r_2)$. Moreover, if $F_\mu(r_1)<F_\mu(r_2)$ then the inequalities are strict.}			
		\end{enumerate}
	\end{lem}
	
	\begin{proof}
		(i) {Fix $l\in(x_2,\overline{w}_{x_1,x_2})$. By} Lemma \ref{lem:E_vw-linear}, $S^\nu(\mu_l-\mu_{x_1})=\nu$ on $(\overrightarrow{m}_{x_1,x_2}(l),\overrightarrow{n}_{x_1,x_2}(l))$. In order to conclude that $S^\nu(\mu_l-\mu_{x_1})=\nu\lvert_{(\overrightarrow{m}_{x_1,x_2}(l),\overrightarrow{n}_{x_1,x_2}(l))}$ it remains to show that $S^\nu(\mu_l-\mu_{x_1})(\R\setminus(\overrightarrow{m}_{x_1,x_2}(l),\overrightarrow{n}_{x_1,x_2}(l)))=0$.
		
		Suppose that $\overrightarrow{m}_{x_1,x_2}(l)\leq x_1$. Note that the second derivative of $\sE_{x_1,l}$ on $(-\infty, x_1)\cup(l,\infty)$ corresponds to $\nu$. Then since $\nu$ is continuous and $\sE^c_{x_1,l}=\sE_{x_1,l}$ on $(-\infty,\overrightarrow{m}_{x_1,x_2}(l)]\cup[\overrightarrow{n}_{x_1,x_2}(l),\infty)$ we have that $(\nu-S^\nu(\mu_l-\mu_{x_1}))=\nu$ on $(-\infty,\overrightarrow{m}_{x_1,x_2}(l)]\cup[\overrightarrow{n}_{x_1,x_2}(l),\infty)$, and therefore $S^\nu(\mu_l-\mu_{x_1})((\overrightarrow{m}_{x_1,x_2}(l),\overrightarrow{n}_{x_1,x_2}(l)))=S^\nu(\mu_l-\mu_{x_1})(\R)${, and we are done}.
		
		Now suppose $\overrightarrow{m}_{x_1,x_2}(l)> x_1$. Since $\overrightarrow{m}_{x_1,x_2}(\cdot)$ is non-increasing on $(x_2,\overline{w}_{x_1,x_2})$ (see Lemma \ref{lem:m_vu-decreasing}), $\overrightarrow{m}_{x_1,x_2}(l)\leq \overrightarrow{m}_{x_1,x_2}(x_2+)\leq \overrightarrow G_\mu(F_\mu(x_1)+)$ ({the second inequality} follows directly if $x_1=x_2$, or by assumption in the case $x_1<x_2$). It follows that $\mu$ does not  charge $(x_1,\overrightarrow{m}_{x_1,x_2}(l))\subseteq (x_1,\overrightarrow G_\mu(F_\mu(x_1)+))$. Therefore, $\sE_{x_1,l}$ is convex on $(-\infty,\overrightarrow{m}_{x_1,x_2}(l))$ and its second (distributional) derivative on $(-\infty,\overrightarrow{m}_{x_1,x_2}(l))$ corresponds to $\nu$.
{Since $\sE^c_{x_1,l} = \sE_{x_1,l}$ on $[\overrightarrow{n}_{x_1,x_2}(l),\infty)$ and hence $\nu - S^\nu(\mu_l - \mu_{x_1}) = \nu$ on $[\overrightarrow{n}_{x_1,x_2}(l),\infty)$, we} again conclude that $S^\nu(\mu_l-\mu_{x_1})((\overrightarrow{m}_{x_1,x_2}(l),\overrightarrow{n}_{x_1,x_2}(l)))=S^\nu(\mu_l-\mu_{x_1})(\R)$.
		
		We now show that the assertion holds for $l=\overline{w}_{x_1,x_2}<\infty$ as well. By the monotonicity of $\overrightarrow{m}_{x_1,x_2}$ and $\overrightarrow{n}_{x_1,x_2}$ on $(x_2,\overline{w}_{x_1,x_2}]$ we have that $\overrightarrow{m}_{x_1,x_2}(\overline{w}_{x_1,x_2})\leq \overrightarrow{m}_{x_1,x_2}(\overline{w}_{x_1,x_2}-)$ and $\overrightarrow{n}_{x_1,x_2}(\overline{w}_{x_1,x_2}-)\leq \overrightarrow{n}_{x_1,x_2}(\overline{w}_{x_1,x_2})$. On the other hand, $S^\nu(\mu_{\overline{w}_{x_1,x_2}}-\mu_{x_1})=\nu$ on $((\overrightarrow{m}_{x_1,x_2}(\overline{w}_{x_1,x_2}-),\overrightarrow{n}_{x_1,x_2}(\overline{w}_{x_1,x_2}-))$ and \begin{align*}S^\nu(\mu_{\overline{w}_{x_1,x_2}}-\mu_{x_1})(\R)&=F_\mu(\overline{w}_{x_1,x_2})-F_\mu(x_1)=\lim_{l\uparrow \overline{w}_{x_1,x_2}} F_\mu(l)-F_\mu(x_1)=\lim_{l\uparrow \overline{w}_{x_1,x_2}}S^\nu(\mu_{l}-\mu_{x_1})(\R)\\&=\lim_{l\uparrow \overline{w}_{x_1,x_2}}\nu((\overrightarrow{m}_{x_1,x_2}(l),\overrightarrow{n}_{x_1,x_2}(l))) =\nu((\overrightarrow{m}_{x_1,x_2}(\overline{w}_{x_1,x_2}-),\overrightarrow{n}_{x_1,x_2}(\overline{w}_{x_1,x_2}-))).\end{align*} It follows that $S^\nu(\mu_{\overline{w}_{x_1,x_2}}-\mu_{x_1})=\nu\lvert_{(\overrightarrow{m}_{x_1,x_2}(\overline{w}_{x_1,x_2}-),\overrightarrow{n}_{x_1,x_2}(\overline{w}_{x_1,x_2}-))}
=\nu\lvert_{(\overrightarrow{m}_{x_1,x_2}(\overline{w}_{x_1,x_2}),\overrightarrow{n}_{x_1,x_2}(\overline{w}_{x_1,x_2}))}$, where the last equality follows from the fact that $\sE^c_{x_1,\overline{w}_{x_1,x_2}}$ is linear on $[\overrightarrow{m}_{x_1,x_2}(\overline{w}_{x_1,x_2}),\overrightarrow{n}_{x_1,x_2}(\overline{w}_{x_1,x_2})]$ and therefore, in the case $\overrightarrow{m}_{x_1,x_2}(\overline{w}_{x_1,x_2})<\overrightarrow{m}_{x_1,x_2}(\overline{w}_{x_1,x_2}-)$ (resp. $\overrightarrow{n}_{x_1,x_2}(\overline{w}_{x_1,x_2}-)<\overrightarrow{n}_{x_1,x_2}(\overline{w}_{x_1,x_2})$), $\nu((\overrightarrow{m}_{x_1,x_2}(\overline{w}_{x_1,x_2}),\overrightarrow{m}_{x_1,x_2}(\overline{w}_{x_1,x_2}-)])=0$ (resp. $\nu((\overrightarrow{n}_{x_1,x_2}(\overline{w}_{x_1,x_2}-),\overrightarrow{n}_{x_1,x_2}(\overline{w}_{x_1,x_2})])=0$).

		(ii) Monotonicity follows from Corollary~\ref{cor:n_vu-increasing} and Lemma~\ref{lem:mvu-decreasingPrelim} so it only remains to prove the statement about strict monotonicity. Let $x_2<r_1<r_2\leq \overline{w}_{x_1,x_2}$ with $F_\mu(r_1)<F_\mu(r_2)$. We cannot have both $\overrightarrow{m}_{x_1,x_2}(r_1)=\overrightarrow{m}_{x_1,x_2}(r_2)$ and $\overrightarrow{n}_{x_1,x_2}(r_1)=\overrightarrow{n}_{x_1,x_2}(r_2)$ else $F_\mu(r_2)-F_\mu(r_1)=S^{\nu-S^\nu(\mu_{r_1}-\mu_{x_1})}(\mu_{r_2}-\mu_{r_1})(\R)=0$, a contradiction. Suppose $\overrightarrow{m}_{x_1,x_2}(r_2)<\overrightarrow{m}_{x_1,x_2}(r_1)$ and $\overrightarrow{n}_{x_1,x_2}(r_1)=\overrightarrow{n}_{x_1,x_2}(r_2)$. Then $S^{\nu-S^\nu(\mu_{r_1}-\mu_{x_1})}(\mu_{r_2}-\mu_{r_1})=\nu\lvert_{(\overrightarrow{m}_{x_1,x_2}(r_2),\overrightarrow{m}_{x_1,x_2}(r_1))}$. But this cannot hold since $\mu_{r_2}-\mu_{r_1}$ places all its mass at or to the right of $r_1$ and $r_1\geq\overrightarrow{m}_{x_1,x_2}(r_1)$ and so the shadow measure $S^{\nu-S^\nu(\mu_{r_1}-\mu_{x_1})}(\mu_{r_2}-\mu_{r_1})$, being a measure in convex order with respect to $\mu_{r_2}-\mu_{r_1}$, must place some mass to the right of {$\overrightarrow{m}_{x_1,x_2}(r_1)$}.
		
		Now we show that $\overrightarrow{m}_{x_1,x_2}(r_2)=\overrightarrow{m}_{x_1,x_2}(r_1)$ and $\overrightarrow{n}_{x_1,x_2}(r_1)<\overrightarrow{n}_{x_1,x_2}(r_2)$ cannot happen either. Since $r_1<\overline{w}_{x_1,x_2}$, $r_1\leq \overrightarrow G_\mu(F_\mu(r_1)+)<\overrightarrow{n}_{x_1,x_2}(r_1)$. Then there exists $\tilde\epsilon>0$ such that $r_1+\tilde\epsilon<w$, $F_\mu(r_1)<F_\mu(r_1+\tilde\epsilon)$, and for all $\epsilon\in[0,\tilde\epsilon]$ we have that $\overrightarrow{m}_{x_1,x_2}(r_2)=\overrightarrow{m}_{x_1,x_2}(r_1)=\overrightarrow{m}_{x_1,x_2}(r_1+\epsilon) \leq \overrightarrow G_\mu(F_\mu(r_1)+) \leq  r_1+\epsilon < \overrightarrow{n}_{x_1,x_2}(r_1)$. But then $S^{\nu-S^\nu(\mu_{r_1}-\mu_{x_1})}(\mu_{r_1+\epsilon}-\mu_{r_1}$ concentrates on $(\overrightarrow{n}_{x_1,x_2}(r_1),\overrightarrow{n}_{x_1,x_2}(r_1+\epsilon))$ and thus to the right of $(r_1+\epsilon)$, while $\mu_{r_1+\epsilon}-\mu_{r_1}$ places all its mass at or to the left of $(r_1+\epsilon)$, a contradiction.
	
		\end{proof}
		

The next result extends part (ii) of Lemma \ref{lem:properties1} to $r_1,r_2\in[x_2,\overline{w}_{x_1,x_2}]$
\begin{cor}\label{cor:strictmon}
Suppose $-\infty\leq x_1\leq x_2<\overline{w}_{x_1,x_2}$. Suppose that either $x_1=x_2$, or $x_1<x_2$ and $\overrightarrow{m}_{x_1,x_2}(x_2+)\leq \overrightarrow G_\mu(F_\mu(x_1)+)$. Let $r\in(x_2,\overline{w}_{x_1,x_2}]\cap\R$ with $F_\mu(x_2)<F_\mu(r)$. {Then $\overrightarrow{n}_{x_1,x_2}(x_2)<\overrightarrow{n}_{x_1,x_2}(r)$
and $\overrightarrow{m}_{x_1,x_2}(x_2)>\overrightarrow{m}_{x_1,x_2}(r)$.}
\end{cor}

\begin{proof} Monotonicity follows from {Corollary~\ref{cor:n_vu-increasing} and} Lemma~\ref{lem:mvu-decreasingPrelim}; strict monotonicity on $[x_2,\overline{w}_{x_1,x_2}]\cap\R$ then follows from strict monotonicity on $(x_2,\overline{w}_{x_1,x_2}]\cap\R$ (see Lemma \ref{lem:properties1}).
\end{proof}

	Let $-\infty<x_1\leq x_2\leq\infty$. Define $\overleftarrow{m}_{x_1,x_2},\overleftarrow{n}_{x_1,x_2}:(-\infty,x_1]\cap\R\to\R$ by
	$$
	\overleftarrow{m}_{x_1,x_2}(l)=X^-_{\sE_{l,x_2}}(l)\quad\textrm{and}\quad\overleftarrow{n}_{x_1,x_2}(l)=Z^+_{\sE_{l,x_2}}(l),\quad l\in(-\infty,x_1]\cap\R.
	$$
	The following results follow by symmetry, {working right-to-left instead of left-to-right.}
	
	\begin{lem}\label{lem:E_wv-linear}
		Fix $-\infty< l\leq x_1\leq x_2\leq \infty$. Then, for all $w\in(-\infty,l]$, $S^\nu(\mu_{x_2}-\mu_w)=\nu$ on $(\overleftarrow{m}_{x_1,x_2}(l),\overleftarrow{n}_{x_1,x_2}(l))$  and $\sE^c_{w,x_2}$ is linear on $[\overleftarrow{m}_{x_1,x_2}(l),\overleftarrow{n}_{x_1,x_2}(l)]$.
	\end{lem}

\begin{cor}\label{cor:mn2_bouds}
	Fix $-\infty\leq l\leq x_1\leq x_2\leq\infty$. If $l\in(\alpha_\mu,\beta_\mu)$ then $\alpha_\nu < \overleftarrow{m}_{x_1,x_2}(l)\leq \overleftarrow{n}_{x_1,x_2}(l)<\beta_\nu$.
\end{cor}
\begin{cor}\label{m_w-decreasing} Fix $-\infty<x_1\leq x_2\leq\infty$.
		$\overleftarrow{m}_{x_1,x_2}(\cdot)$ is non-decreasing on $\R\cap(-\infty,x_1]$.
	\end{cor}
{For $x_1,x_2\in \R \cup \{\infty \}$ with $x_1\leq x_2$ define $\underline{w}_{x_1,x_2}:=\sup\{w\in (\alpha_\mu,x_2):\overleftarrow{m}_{x_1,x_2}(w)\geq \overleftarrow G_\mu(F_\mu(w))\} \vee \alpha_\mu$ with convention $\sup\emptyset=-\infty$.}
	
\begin{lem}\label{lem:w_uv<u}
			Suppose $-\infty< x_1\leq x_2\leq\infty$ with $x_1<\infty$. If $\overleftarrow{m}_{x_1,x_2}(x_1)<\overleftarrow{G}_\mu(F_\mu(x_1)-)$ then $\underline{w}_{x_1,x_2}<x_1$ {and moreover $F_\mu(\underline{w}_{x_1,x_2})<F_\mu(x_1)$}.
\end{lem}
\begin{lem}
Suppose $\alpha_\mu\leq \underline{w}_{x_1,x_2}<x_1\leq x_2\leq\infty$.  {Then
			$ \overleftarrow{G}_\mu(F_\mu(\underline{w}_{x_1,x_2})-)\leq \overleftarrow{m}_{x_1,x_2}(\underline{w}_{x_1,x_2})\leq \underline{w}_{u,v}
			$
and $D_{\mu,\nu}(\overleftarrow{m}_{x_1,x_2}(\underline{w}_{x_1,x_2})) = \sE_{x_1,x_2}(\overleftarrow{m}_{x_1,x_2}(\underline{w}_{x_1,x_2}))=\sE^c_{x_1,x_2}(\overleftarrow{m}_{x_1,x_2}(\underline{w}_{x_1,x_2}))$.}
		\end{lem}
	
	\begin{lem}\label{lem:n_uv-decreasing}
		Fix $x_1,x_2\in\R\cup\{\infty\}$ with $x_1\leq x_2$. Then $\overleftarrow{n}_{x_1,x_2}(\cdot)$ is non-increasing on $[\underline{w}_{x_1,x_2},x_1)\cap\R$.
	\end{lem}
	\begin{cor}\label{cor:n_uv-decreasing}
		Suppose $-\infty<x_1\leq x_2\leq\infty$ with $x_1<\infty$. If $\overleftarrow{m}_{x_1,x_2}(x_1)<\overleftarrow G_\mu(F_\mu(x_1)-)$ then {$\underline{w}_{x_1,x_2}<x_1$ and} $\overleftarrow{n}_{x_1,x_2}(\cdot)$ is non-increasing on $[\underline{w}_{x_1,x_2},x_1]\cap\R$.
	\end{cor}
	\begin{lem}
		\label{lem:properties2}
		Suppose $\underline{w}_{x_1,x_2}<x_1\leq x_2\leq \infty$. Suppose that either $x_1=x_2$, or $x_1<x_2$ and $\overleftarrow G_\mu(F_\mu(x_2)-)\leq \overleftarrow{n}_{x_1,x_2}(x_1-)$. Then
		\begin{enumerate}
			\item[(i)] $S^\nu(\mu_{x_2}-\mu_{l})=\nu\lvert_{(\overleftarrow{m}_{x_1,x_2}(l),\overleftarrow{n}_{x_1,x_2}(l))}$, for all $l \in [\underline{w}_{x_1,x_2},x_1)$.

			\item[(ii)] Suppose that $r_1,r_2\in[\underline{w}_{x_1,x_2},x_1)$ with $r_1<r_2$ and $F_\mu(r_1)<F_\mu(r_2)$. Then $\overleftarrow{n}_{x_1,x_2}(r_1)>\overleftarrow{n}_{x_1,x_2}(r_2)$ and $\overleftarrow{m}_{x_1,x_2}(r_1)<\overleftarrow{m}_{x_1,x_2}(r_2)$.
			
		\end{enumerate}
	\end{lem}
\begin{cor}\label{cor:2strictmon}
Suppose 
$\underline{w}_{x_1,{x_2}}<x_1\leq x_2\leq\infty$. Suppose that either $x_1=x_2$, or $x_1<x_2$ and $\overleftarrow G_\mu(F_\mu(x_2)-)
\leq \overleftarrow{n}_{x_1,x_2}(x_1{-})$. Let $r\in[\underline{w}_{x_1,{x_2}},x_1)$ with $F_\mu(r)<F_\mu(x_1)$. Then $\overrightarrow{n}_{x_1,x_2}(r)>\overrightarrow{n}_{x_1,x_2}(x_1)$ and $\overrightarrow{m}_{x_1,x_2}(r)<\overrightarrow{m}_{x_1,x_2}(x_1)$.
\end{cor}

\section{The construction in the regular case}
\label{sec:construction}

{The goal of this section is to construct a strongly injective martingale coupling of $\mu$ and $\nu$ in a fairly general, but still regular case. We may restrict attention to the continuous, irreducible case, so our general setting is pairs of measures in the following set.}
\begin{defn}\label{def:K}
{Suppose $\mu, \nu$ are distinct, non-zero elements of $\sP$. Then $(\mu,\nu) \in \sK$ if $\nu$ and $\mu$ are continuous, if $\mu \leq_{cx} \nu$ and if $\{ x: D_{\mu,\nu}(x)>0 \}= I_\nu= (\alpha_\nu,\beta_\nu)$.}
\end{defn}

We begin by outlining the principles which govern our approach; note that the condition $(\mu,\nu)\in\sK$ is equivalent to the Standing Assumption \ref{sass:atomfree}, and thus in the case $(\mu,\nu)\in\sK$ the results of Section \ref{sec:dispersion} apply.

Given $(\mu,\nu) \in \sK$ the first step is to choose a suitable starting point $x_0 \in [\alpha_\mu, \beta_\mu)\cap\R$ and then to consider $(\sE_{x_0,k})_{x_0 \leq k \leq \beta_\mu}$ and to define $x_1 = \overline{w}_{x_0,x_0}$. Assuming that $x_1>x_0$, we define $M$ and $N$ on $[x_0,x_1]$ via $M(k) = \overrightarrow{m}_{x_0,x_0}(k)$ and $N(k) = \overrightarrow{m}_{x_0,x_0}(k)$. These functions are monotonic (on the domain where they have been defined) and moreover, for each $k \in [x_0,x_1]$ we have that $S^\nu(\mu_k - \mu_{x_0}) = \nu|_{(M(k),N(k))}$. In particular, the functions $M$ and $N$ can be used to define a martingale coupling of $\mu_{x_1} - \mu_{x_0}$ and $\nu|_{(M(x_1),N(x_1))}$ {via $\pi_x = \pi^{M,N}_x$} (where $\pi^{M,N}_x$ is defined as in \eqref{eq:pifh}). Moreover, {(with some care over dealing with intervals where either $\mu$ or $\nu$ has no support) the monotonicity of $M$ and $N$ can be used to show that} this coupling is injective.

If $x_0=\alpha_\mu>-\infty$ it will then follow by Lemma~\ref{lem:w=muR} below that $x_1= \beta_\mu$ and then the construction is complete, and we have described an injective martingale coupling. More generally, {if $x_0>\alpha_\mu$ then this will not be the case}. Then we proceed by induction. If $x_0>\alpha_\mu$ (and then $x_1 \leq \beta_\mu$) we consider $(\sE_{k,x_1})_{\alpha_\mu \leq k \leq {x_0}}$ and set $x_2 = \underline{w}_{x_0,x_1}$. It will turn out that $x_2<x_0<x_1$ and we can extend the definitions of $M$ and $N$ to $[x_2,x_1]$ such that $M$ and $N$ are monotonic on $[x_2,x_0)$ and such that $S^\nu(\mu_{x_1} - \mu_{k}) = \nu|_{(M(k),N(k))}$. In particular, $S^\nu(\mu_{x_1} - \mu_{x_2}) = \nu|_{(M(x_2),N(x_2))}$ and $M$ and $N$ now defined on $[x_2,x_0]$ can be used to define a martingale coupling of $\mu_{x_1} - \mu_{x_2}$ and $\nu|_{(M(x_2),N(x_2))}$. It follows from the properties of $M$ and $N$ that the coupling is injective.

If $x_1 = \beta_\mu$ then $x_2=\alpha_\mu$ and the construction terminates. Otherwise we consider $(\sE_{x_2,k})_{k \geq x_1}$ and proceed by induction working alternately from left-to-right and then right-to-left.

Suppose we have $\alpha_\mu \leq x_{2k} < x_{2k-2} < \ldots x_0 < x_1 < \ldots x_{2k-1}<\beta_\mu$ and $M$ and $N$ defined on $[x_{2k},x_{2k-1}]$. Then we define $x_{2k+1} \in (x_{2k-1},\beta_\mu]$ by $x_{2k+1} = \overline{w}_{x_{2k},x_{2k-1}}$, and extend the domain of definition of $M$ and $N$ to $[x_{2k},x_{2k+1}]$ such that for $l \in (x_{2k-1},x_{2k+1}]$, $S^\nu(\mu_{l} - \mu_{x_{2k}}) = \nu|_{(M(l),N(l))}$. Either $x_{2k}=\alpha_\mu$ and then $x_{2k+1}=\beta_\mu$ and the construction terminates, or
$\alpha_\mu < x_{2k} < x_{2k-2} < \ldots < x_0 < x_1 < \ldots < x_{2k-1} < x_{2k+1} \leq \beta_\mu$. Then we define $x_{2k+2} \in [\alpha_\mu,x_{2k})$ by $x_{2k+2} = \underline{w}_{x_{2k},x_{2k+1}}$ and extend the definitions of $M$ and $N$ to $[x_{2k+2},x_{2k+1}]$.

Either $x_{j}=\alpha_\mu, x_{j+1} = \beta_\mu$ for some (even) $j$, or $x_{j}=\beta_\mu, x_{j+1} =\alpha_\mu$ for some (odd) $j$, and we have a pair of functions $\{M,N\}$ defined on $[\alpha_\mu, \beta_\mu]\cap\R$, or we have a pair of decreasing and increasing sequences
such that $\alpha_\mu < \ldots < x_{2k} < x_{2k-2} < \ldots < x_0 < x_1< \ldots < x_{2k-1} < x_{2k+1} < \ldots < \beta_\mu$ and a pair of functions $\{M,N\}$ defined on $(x_\infty,x^\infty)$ where $x_\infty := \lim_{k} x_{2k}$ and $x^\infty := \lim_{k} x_{2k+1})$.

The main issues are: first, to argue that $x_1 > x_0$ and thereafter $x_{2k}<x_{2k-2}$ and $x_{2k+1} > x_{2k-1}$ at least until $x_{2k}$ reaches $\alpha_\mu$ or $x_{2k+1}$ reaches $\beta_\mu$ and the construction terminates; second, to show that if the construction does not terminate then $x_\infty =\alpha_\mu$ and $x^\infty=\beta_\mu$; third to prove that $S^\nu(\mu_{x_{2k + 1}} - \mu_{x_{2k}}) = \nu|_{(M(x_{2k+1}),N(x_{2k+1}))}$ and that $(M(k),N(k))_{\alpha_\mu < k < \beta_\mu}$ define a  martingale coupling; and fourth to justify that $M$ and $N$ have appropriate monotonicity properties so that the coupling
{defined via $\pi_x = \pi^{M,N}_x$ is a strongly injective martingale coupling on its irreducible component (perhaps after some modification at points of the sequence $\{x_k\}_{k \geq 0}$ and at ends of intervals where $\mu$ or $\nu$ has no support).}

In fact, it is not the case that the construction will work as described in the general case.
For this reason we introduce a subspace $\sK_* \subseteq \sK$. We will show that the construction outlined above defines an injective coupling on $\sK_*$. Later we show that the general case can be reduced to this case. As an intermediate step we introduce $\sK_R$ with $\sK_* \subseteq \sK_R \subseteq \sK$.

Let $G_\mu$ be any quantile function of $\mu$. Note that, for all $u \in [0, \mu(\R)]$, $G_\mu(u-)$ and $G_\mu(u+)$ do not depend on the choice of $G_\mu$. By our conventions, $G_\mu(0-)=-\infty$ and $G_\mu(\mu(\R)+)=\infty$.

\begin{defn}\label{def:KR}
	Suppose $(\mu, \nu) \in \sK$. Then $(\mu,\nu) \in \sK_{R}$ if {(see Figure~\ref{fig:DERS})}
	\begin{enumerate}
		\item $\exists a \in \R$ such that $\mu \leq \nu$ on $(-\infty,a)$; let $x_0$ be the largest such value, so that $G_\mu(F_\mu(x_0)+) = x_0$; let $\underline{x}_0 = \inf \{ a : \mbox{$\mu = \nu$ on $({a},x_0)$} \}$, with the convention $\inf\emptyset=x_0$. 
\item The tangent line $L_{x}$ to $D_{\mu,\nu}$ at $x\in[\underline{x}_0,x_0]$ is such that $D_{\mu,\nu} < L_{x}$ on $(x_0,\beta_\nu)$.
		\item for all $b \in (\alpha_\nu ,\underline{x}_0)$ the tangent line $L_b$ to $D_{\mu,\nu}$ at $b$, given by $L_b(k) = D_{\mu,\nu}(b) + (k-b)D'_{\mu,\nu}(b)$, is such that there exists $c$ with $c>x_0$ such that $D_{\mu,\nu}>L_b$ on $(x_0,c)$, $D_{\mu,\nu}<L_b$ on $(c,\infty)$ and $D'_{\mu,\nu}(b) > D'_{\mu,\nu}(c)$.
	\end{enumerate}
\end{defn}
Note that, since $\mu$ and $\nu$ are continuous (and $\mu\leq\nu$ on $(-\infty,x_0)$), $D_{\mu,\nu}$ is continuously differentiable on $\R$ (and convex on $(-\infty,x_0)$), and thus has a unique tangent at each $k\in\R$.
Moreover the tangent line is the same for all $x \in [\underline{x}_0,x_0]$.

\begin{lem}\label{lem:u_0=1}
	Suppose $(\mu,\nu)\in\sK_{R}$. Then $x_0\in[\alpha_\mu,\beta_\mu)\cap\R$.
\end{lem}
\begin{proof}
That $x_0>-\infty$ and $x_0\geq\alpha_\mu$ follows immediately from Definition \ref{def:KR}.

	If $x_0=\beta_\mu$ then $\mu \leq \nu$ on $\R$ and therefore $\mu=\nu$, since $\mu\leq_{cx}\nu$. Then $D_{\mu,\nu} \equiv 0$, but this is not possible since $(\mu,\nu)\in\sK$.
\end{proof}
\begin{figure}[H]
	\centering
\begin{tikzpicture}

\begin{axis}[%
width=4.521in,
height=3.066in,
at={(0.758in,0.481in)},
scale only axis,
xmin=-11,
xmax=11,
ymin=-0.1,
ymax=0.8,
axis line style={draw=none},
ticks=none
]
\addplot [color=black, line width=1.0pt, forget plot]
table[row sep=crcr]{%
	-10	0\\
	-9.58289512426709	0.0021747064911354\\
	-9.09090909090909	0.0103305798596145\\
	-8.61501384398931	0.0239786697604272\\
	-8.18181818181818	0.0413223203462037\\
	-7.73583407671843	0.0640805747071888\\
	-7.27272727272727	0.0929751936163919\\
	-6.77118450691904	0.130315567398738\\
	-6.36363636363636	0.165289309108455\\
	-5.85703605407862	0.213984060805278\\
	-5.45454545454545	0.249999897780453\\
	-5.2352891080966	0.267537017816371\\
	-5.01603276164775	0.283604827931299\\
	-4.78074365355115	0.299214111819187\\
	-4.54545454545454	0.311755236319915\\
	-4.32697443302579	0.324539922451977\\
	-4.10849432059704	0.334490372228037\\
	-3.87242897848034	0.343601282497902\\
	-3.63636363636364	0.351010010420574\\
	-3.38283822498829	0.357070094871146\\
	-3.12931281361294	0.361166531032035\\
	-2.92829277044283	0.363019425902994\\
	-2.72727272727273	0.363636422201081\\
	-2.48352168796022	0.36272843649788\\
	-2.23977064864772	0.360004825788297\\
	-2.02897623341477	0.356187667436133\\
	-1.81818181818182	0.351009993159971\\
	-1.57896351004001	0.343492051124193\\
	-1.33974520189821	0.334223457712913\\
	-1.12441805549456	0.32438511856772\\
	-0.909090909090909	0.313131459903253\\
	-0.458761721180472	0.285015318067985\\
	0	0.249999136603171\\
	0.454664637525873	0.21469482694629\\
	0.909090909090909	0.184573597760398\\
	1.38170401811689	0.158722362127491\\
	1.81818181818182	0.139806151212122\\
	2.30915892883115	0.124222275606183\\
	2.72727272727273	0.115700963849885\\
	3.17401811478308	0.111425947075356\\
	3.63636363636364	0.11225995377717\\
	4.09389341851855	0.118157016020043\\
	4.54545454545454	0.123276860436698\\
	4.77290290355786	0.124570566100095\\
	5	0.125011797087635\\
};
\draw [ line width=1.0pt] (5,0.125011797087635) to [out=1,in=180] (10,0);
\draw [thick] (-11,0)--(-10,0);
\draw [thick] (11,0)--(10,0);
\draw[thin] (-10,0)--(10,0);
\draw[blue,thick] (-9.1, -0.09753086)-- (0,0.79);
\draw[thin, gray, dashed] (-6.9,0.1) -- (-6.9,0);
\draw[thin, gray, dashed] (-5.6,0.23) -- (-5.6,0);
\node[below,scale=1] at (-6.9,0) {$\underline{x}_0$};
\node[below,scale=1] at (-5.6,0) {$x_0$};
\draw[thin, gray, dashed] (-9,0.01) -- (-9,0);
\node[below,scale=1] at (-9,0) {$b$};
\node[below,scale=1,blue] at (-4,0.7) {$k\mapsto L_{x_0}(k)$};
\node[below,scale=1,red] at (2,0.4) {$k\mapsto L_b(k)$};
\node[below,scale=1] at (8,0.2) {$k\mapsto D_{\mu,\nu}(k)$};
\draw[thin, gray, dashed] (0.34,0.2) -- (0.34,0);
\node[below,scale=1] at (0.34,0) {$c$};
\draw[red,thick] (-11, -0.032)-- (5,0.33);

\end{axis}
\end{tikzpicture}
\caption{Plot of $D_{\mu,\nu}$ for $(\mu,\nu) \in \sK_R$. The line $L_{x_0}$ satisfies $D_{\mu,\nu}>L_{x_0}$ on $(-\infty,\underline{x}_0)$, $D_{\mu,\nu}=L_{x_0}$ on $[\underline{x}_0,x_0]$ and $D_{\mu,\nu}<L_{x_0}$ on $(x_0,\infty)$. On the other hand, the line $L_b$, that is tangent to $D_{\mu,\nu}$ at $b<\underline{x}_0$, crosses $D_{\mu,\nu}$ at $c>x_0$ and satisfies $D_{\mu,\nu}\geq L_{b}$ on $(-\infty,c]$ and $D_{\mu,\nu}< L_{b}$ on $(c,\infty)$. }
\label{fig:DERS}
\end{figure}

\begin{lem}\label{lem:strongermonotonicity}
		Suppose $(\mu,\nu)\in\sK_{R}$.
		\begin{enumerate}
			\item[(i)] Suppose $\alpha_\mu<v\leq x_0\leq u<\overline{w}_{v,u}$. {Then $\overline{w}_{v,u}< \infty$ and $
			\overrightarrow{n}_{v,u}(\overline{w}_{v,u})<\beta_\nu$.} If, in addition, $\overrightarrow{m}_{v,u}(\overline{w}_{v,u})<\overrightarrow{G}_\mu(F_\mu(v)+) \wedge \underline{x}_0$, then
			$
			\overrightarrow{m}_{v,u}(\overline{w}_{v,u})<\overrightarrow{G}_\mu(F_\mu(v))$. 
			\item[(ii)] Suppose $\underline{w}_{u,v}<u\leq x_0\leq v< \beta_\mu$. Then {$-\infty <\underline{w}_{u,v}$ and $\overleftarrow{m}_{u,v}(\underline{w}_{u,v})>\alpha_\nu$.} If, in addition, $\overleftarrow{m}_{u,v}(\underline{w}_{u,v}) <\tilde x_0$ and $\overleftarrow{n}_{u,v}(\underline{w}_{u,v})>\overleftarrow{G}_\mu(F_\mu(v)-)$, then
			$
			\overleftarrow{n}_{u,v}(\underline{w}_{u,v})>\overleftarrow{G}_\mu(F_\mu(v)).
			$ 
		\end{enumerate}
	\end{lem}
\begin{proof}
	

We prove $(i)$; part $ii)$ follows similarly. 

{ We begin by showing that $\overline{w}_{v,u}<\infty$. If $\beta_\mu<\infty$ then there is nothing to prove, so suppose $\beta_\mu = \infty$. Then $\beta_\nu = \infty$ also. We suppose $\overline{w}_{v,u}=\infty$ and look for a contradiction. If $\overline{w}_{v,u}=\infty$ then using Lemma \ref{lem:E_vw-linear} we have that $\sE^c_{v, \overline{w}_{v,u}} = \sE^c_{v,\infty}$ is linear on
\[ \left( \lim_{l \rightarrow \overline{w}_{v,u}=\infty} \overrightarrow{m}_{v,u}(l) ,  \lim_{l \rightarrow \overline{w}_{v,u}=\infty} \overrightarrow{n}_{v,u}(l) = \infty\right) 
.\]
But, $\lim_{ l \rightarrow \infty} D(l) = \lim_{l \rightarrow \infty} D'(l)=0$ and thus necessarily $\lim_{l\to \overline{w}_{v,u}} D(\overrightarrow{m}_{v,u}(l))=D(\lim_{l\to \overline{w}_{v,u}}\overrightarrow{m}_{v,u}(l))=0$. Since, $D>0$ on $(\alpha_\nu,\beta_\nu)$, we then further have that $\lim_{l\to \overline{w}_{v,u}}\overrightarrow{m}_{v,u}(l)=\alpha_\nu$. It follows that $\sE^c_{v,\infty}\equiv 0$ and $\nu-S^\nu(\mu-\mu_v)$ is the zero measure. But this contradicts the fact that $v>\alpha_\mu$. We conclude that $\overline{w}_{v,u}<\infty$.

Next we show that $\overrightarrow{n}_{v,u}(\overline{w}_{v,u})<\beta_\nu$.
}
Suppose that $\overrightarrow{n}_{v,u}(\overline{w}_{v,u})\geq \beta_\nu$. Then $\lim_{k\to\overrightarrow{n}_{v,u}(\overline{w}_{v,u})}D(k)=0$ and the line joining $\sE^c_{v,\overline{w}_{v,u}}(\overrightarrow{m}_{v,u}(\overline{w}_{v,u}))$ and $\sE^c_{v,\overline{w}_{v,u}}(\overrightarrow{n}_{v,u}(\overline{w}_{v,u}))$  must have slope equal to zero, from which it follows that $\sE_{v,u}(\overrightarrow{m}_{v,u}(\overline{w}_{v,u}))=0$. But this contradicts the fact that $\sE_{v,u}>0$ on $(-\infty,\beta_\nu)$, since $v> { \alpha_\mu}$.


Now suppose that $\overrightarrow{m}_{v,u}(\overline{w}_{v,u})<\overrightarrow{G}_\mu(F_\mu(v)+) \wedge \underline{x}_0$. Suppose further that $\overrightarrow{m}_{v,u}(\overline{w}_{v,u}) \geq\overrightarrow{G}_\mu(F_\mu(v))$.
Then
\begin{equation}\label{eq:Ec=E=D}
\sE^c_{v,\overline{w}_{v,u}}(k)=\sE_{v,\overline{w}_{v,u}}(k)=D_{\mu,\nu}(k)
\end{equation}
for $k = \overrightarrow{m}_{v,u}(\overline{w}_{v,u})$.  Furthermore, by the results of Lemma~\ref{lem:w_vu>u},
\eqref{eq:Ec=E=D} also holds at $k = \overrightarrow{n}_{v,u}(\overline{w}_{v,u})$. Then if $b:=\overrightarrow{m}_{v,u}(\overline{w}_{v,u})$ (where $b<\underline{x}_0$ by hypothesis) we find that  $\overrightarrow{n}_{v,u}(\overline{w}_{v,u})$ plays the role of $c$ in {Definition \ref{def:KR}} in the sense that $D_{\mu,\nu} > L_b$ on $(a_0,c)$ and $D_{\mu,\nu} < L_b$ on $(c,\infty)$. However, $D'_{\mu,\nu}(b)=D'_{\mu,\nu}(c)$ and so the condition $D'_{\mu,\nu}(b)>D'_{\mu,\nu}(c-)$ is not satisfied. Hence $(\mu,\nu) \notin \sK_R$, a contradiction. We conclude that $\overrightarrow{m}_{v,u}(\overline{w}_{v,u})<\overrightarrow{G}_\mu(F_\mu(v))$.

Part $(ii)$ follows similarly, where again we have $b={\overleftarrow{m}}_{v,u}(\overline{w}_{v,u}) < \tilde{x}_0$ by hypothesis.
\end{proof}

\begin{lem} \label{lem:w=muR}
Suppose $(\mu,\nu) \in \sK_{R}$.
\begin{enumerate}
\item[i)]
Suppose $\alpha_\mu\leq x_0 \leq u < \beta_\mu$ and $\overline{w}_{\alpha_\mu,u}>u$. Then $\overline{w}_{\alpha_\mu,u} = \beta_\mu$.
Further, $\lim_{u \rightarrow \beta_\mu} {\overrightarrow{m}}_{\alpha_\mu,u}(u) = \alpha_\nu$ and $\lim_{u \rightarrow \beta_\mu} {\overrightarrow{n}}_{\alpha_\mu,u}(u) = \beta_\nu$.
\item[ii)]
Suppose ${\alpha_\mu < u \leq x_0< \beta_\mu}$ and $\underline{w}_{u,\beta_\mu}<u$. Then $\underline{w}_{u,\beta_\mu} = \alpha_\mu$. Further, $\lim_{u \rightarrow \alpha_\mu}  {\overleftarrow{m}}_{u,\beta_\mu}(u) = \alpha_\nu$ and $\lim_{u \rightarrow \alpha_\mu}  {\overleftarrow{n}}_{u,\beta_\mu}(u) = \beta_\nu$.
\end{enumerate}
\end{lem}

\begin{proof}
We prove $i)$. The proof of $ii)$ is similar. Suppose $\overline{w}_{\alpha_\mu,u}<\beta_\mu$. It follows that $\sE_{\alpha_\mu,\overline{w}_{\alpha_\mu,u}}=D_{\mu,\nu}$ on $(-\infty, \overrightarrow G_\mu(F_\mu(\overline{w}_{\alpha_\mu,u})+)]$. Note that $\overrightarrow{G}_\mu(F_\mu(\alpha_\mu)+)=\alpha_\mu\leq \tilde x_0$. Then, since $\overline{w}_{\alpha_\mu,u}\in(\alpha_\mu,\beta_\mu)$, by Corollary \ref{cor:mn1_bounds} we have that $b:=\overrightarrow{m}_{\alpha_\mu,u}(\overline{w}_{\alpha_\mu,u})> \alpha_\nu$, so that $\sE^c_{\alpha_\mu,\overline{w}_{\alpha_\mu,u}}(b)=D(b)>0$. On the other hand, by Lemma \ref{lem:w_vu>u}, $c:=\overrightarrow{n}_{\alpha_\mu,u}(\overline{w}_{\alpha_\mu,u})\in[\overline{w}_{\alpha_\mu,u},\overrightarrow{G}_\mu(F_\mu(\overline{w}_{\alpha_\mu,u})+)] $, and therefore $\sE^c_{\alpha_\mu,\overline{w}_{\alpha_\mu,u}}(c) = \sE_{\alpha_\mu,\overline{w}_{\alpha_\mu,u}}(c)=D(c)$. Further, $D_{\mu,\nu}(c)>0$, since $\overline{w}_{\alpha_\mu,u}<\beta_\mu$ and thus $c\leq \overrightarrow{G}_\mu(F_\mu(\overline{w}_{\alpha_\mu,u})+)<\beta_\mu\leq\beta_\nu $.

Since $D_{\mu,\nu}$ is convex on $(-\infty,x_0]$, we must have that there exists $\tilde b \in(-\infty,\tilde x_0)$ such that $\sE^c_{\alpha_\mu,\overline{w}_{\alpha_\mu,u}}=\sE_{\alpha_\mu,\overline{w}_{\alpha_\mu,u}}=D_{\mu,\nu}$ on $(-\infty,\tilde b]$ and $\sE^c_{\alpha_\mu,\overline{w}_{\alpha_\mu,u}}<\sE_{\alpha_\mu,\overline{w}_{\alpha_\mu,u}}=D_{\mu,\nu}$ on $(\tilde b,\tilde x_0)$. (Indeed, since $\sE^c_{\alpha_\mu,\overline{w}_{\alpha_\mu,u}}\leq\sE_{\alpha_\mu,\overline{w}_{\alpha_\mu,u}}$ everywhere, the other two possibilities are that either $\sE^c_{\alpha_\mu,\overline{w}_{\alpha_\mu,u}}=\sE_{\alpha_\mu,\overline{w}_{\alpha_\mu,u}}=D_{\mu,\nu}$ on $(-\infty,\tilde x_0]$ or $\sE^c_{\alpha_\mu,\overline{w}_{\alpha_\mu,u}}<\sE_{\alpha_\mu,\overline{w}_{\alpha_\mu,u}}=D_{\mu,\nu}$ on $(-\infty,x_0)$. In the first case, $\sE^c_{\alpha_\mu,\overline{w}_{\alpha_\mu,u}}\geq L_{\tilde x_0}$ on $(\tilde x_0,\infty)$, and thus, by the second statement of Definition \ref{def:KR}, $\sE^c_{\alpha_\mu,\overline{w}_{\alpha_\mu,u}}$ cannot be equal to $D_{\mu,\nu}$ at any point $k\in(x_0,\infty)$. This contradicts the fact that $x_0\leq u<\overline{w}_{\alpha_\mu,u}\leq c<\infty$ and $\sE^c_{\alpha_\mu,\overline{w}_{\alpha_\mu,u}}(c)=D(c)$. In the second case we have that $\sE^c_{\alpha_\mu,\overline{w}_{\alpha_\mu,u}}=0$ on $(-\infty,x_0]$. Since $\sE^c_{\alpha_\mu,\overline{w}_{\alpha_\mu,u}}(c)=D(c)>0$, we must have that there exists $k\in(x_0,c)$ with $\sE^c_{\alpha_\mu,\overline{w}_{\alpha_\mu,u}}(k)=\sE_{\alpha_\mu,\overline{w}_{\alpha_\mu,u}}(k)=D(k)=0$, but this violates the assumption that $(\mu,\nu)\in\sK$.)

Let $L_{\tilde b}$ be the line tangent to $\sE^c_{\alpha_\mu,\overline{w}_{\alpha_\mu,u}}$ at $\tilde b$. Note that, by the third statement of Definition \ref{def:KR}, $L_{\tilde b}$ meets $D_{\mu,\nu}$ on $(x_0,\infty)$ only once; let $\tilde c$ be such point. Then $\tilde c=c$. (Indeed, since $\sE^c_{\alpha_\mu,\overline{w}_{\alpha_\mu,u}}\geq L_{\tilde b}>D_{\mu,\nu}$ on $(\tilde c,\infty)$, we must have that $c\leq\tilde c$. Suppose $c<\tilde c$. Then $\sE^c_{\alpha_\mu,\overline{w}_{\alpha_\mu,u}}(c)=\sE_{\alpha_\mu,\overline{w}_{\alpha_\mu,u}}(c)=D(c)>L_{\tilde b}(c)$. On the other hand, $\sE^c_{\alpha_\mu,\overline{w}_{\alpha_\mu,u}}=L_{\tilde b}$ on $[\tilde b,x_0]$, and hence there exists $\tilde k\in(x_0,c)$ such that $\sE^c_{\alpha_\mu,\overline{w}_{\alpha_\mu,u}}(\tilde k)=\sE_{\alpha_\mu,\overline{w}_{\alpha_\mu,u}}(\tilde k)=D(\tilde k)=L_{\tilde b}(\tilde k)$, which contradicts the third statement of Definition \ref{def:KR}.)

It follows that $b\leq\tilde b$, and $D'(b)=D'(\tilde b)=D'(c)$. This contradicts the third statement of Definition \ref{def:KR} and we conclude that $\overline{w}_{\alpha_\mu,u}=\beta_\mu$.

For the final statement, by Lemma~\ref{lem:E_vw-linear} with $x_1 = \alpha_\mu$, $S^\nu(\mu|_{(\alpha_\mu, l)}) = \nu|_{(\overrightarrow{m}_{\alpha_\mu,l}(l), {\overrightarrow{n}}_{\alpha_\mu,l}(l))}$ and letting $l$ increase to $\beta_\mu$ we have that $\lim_{l \rightarrow \beta_\mu}\overrightarrow{m}_{\alpha_\mu,l}(l) \leq \alpha_\nu $ and $\lim_{l \rightarrow \beta_\mu}\overrightarrow{n}_{\alpha_\mu,l}(l) \geq \beta_\nu $. But, the reverse inequalities follow from Corollary~\ref{cor:mn1_bounds}.
\end{proof}

Before introducing the inductive step we give some definitions. In the next and subsequent definitions we take $y_{-1}=y_0$.

\begin{defn}
For $j\geq2$ even, let $
\sY_j=\{(y_0,y_1,...,y_j):y_{j-1}<\infty,~\alpha_\mu \leq y_j < y_{j-2} < \ldots < y_0 < y_1 < \ldots y_{j-1} \leq \beta_\mu\}$. For $j\geq1$ odd, let $
\sY_j=\{(y_0,y_1,...,y_j):y_{j-1}>-\infty,~\alpha_\mu \leq y_{j-1} < \ldots < y_0 < y_1 < \ldots y_{j} \leq \beta_\mu\}$.

{Write $\sY_j$ as the disjoint union $\sY_j = \sY^{\{j-1,j\}} \cup \sY^{\{j\}} \cup \sY^{\{j-1\}} \cup \sY^{\emptyset}$ where the superscript is the set of indices $k \in \{0,1,\ldots j \}$ such that $y_k \in \{ \alpha_\mu,\beta_\mu \}$. For example, for $j$ even, $\sY^{ \{j-1 \} } = \{ (y_0,\ldots y_j) : \alpha_\mu <y_j < y_{j-2} < \ldots < y_0 < y_1 < \ldots <y_{j-1} = \beta_\mu<\infty\}$.}

Let 
\begin{eqnarray*}
{{\sY}_\infty} & = & \{ (y_0, y_1, \ldots) : y_0 \in (\alpha_\mu,\beta_\mu), \mbox{ for $j \geq 2$ even, $y_j \in (\alpha_\mu, y_{j-2})$}, \\
&& \hspace{5mm} \mbox{for $j \geq 1$ odd, $y_j \in (y_{j-2}, \beta_\mu)$} \}.
\end{eqnarray*}
Let $\sY^*_\infty=\{y\in \sY_\infty:\lim_k y_{2k}=\alpha_\mu,~\lim_k y_{2k+1}=\beta_\mu\}$.
\end{defn}

Note that, for each $j\geq1$ and $y\in\sY_j$ we have $y_k\in\R$ for all $k=0,...,j-1$. In particular, if $y\in\sY_j^{ \{ j-1 \} }\cup\sY_j^{ \{j-1,j \}}$, then necessarily $\alpha_\mu>-\infty$ (resp. $\beta_\mu<\infty$) if $j$ is odd (resp. even).

\begin{defn}
\label{def:psi}
Fix $1\leq j<\infty$ and $y\in\sY_j$.

Define $\phi_j : \R \cap [y_{j-1}\wedge y_j,y_{j-1}\vee y_j] \to [0,y_{j-1}\vee y_j-y_{j-1}\wedge y_j]\cap\R$ as follows: \\
$\phi_j(y_0)= 0$;
for $y_0<z<y_{j-1}\vee y_j$, so that $z \in (y_{2k-1}, y_{2k+1}]$ for some $k\geq0$, $\phi_j(z) = z - y_{2k}$; for $y_{j-1}\wedge y_j < z < y_0$ so that $z \in [y_{2k+2}, y_{2k})$ for some $k\geq0$, $\phi_j(z) = y_{2k+1}-z$; finally, if $j$ is odd, set $\phi_j(y_{j-1}\wedge y_j) = \phi_j(y_{j-1}) = y_{j-2}-y_{j-1}$ and (in the case $y_j<\infty$) $\phi_j(y_{j-1}\vee y_j) = \phi_j(y_j) =
y_j-y_{j-1}$, and, if $j$ is even, set $\phi_j(y_{j-1}\vee y_j) =\phi_j(y_{j-1}) =y_{j-1}-y_{j-2}$ and (in the case $y_j>-\infty$) $\phi_j(y_{j-1}\wedge y_j) = \phi_j(y_j) = y_{j-1}-y_{j}$.

Let $\psi_j: [0,y_{j-1}\vee y_j-y_{j-1}\wedge y_j] \cap \R_+ \to \R \cap [y_{j-1}\wedge y_j,y_{j-1}\vee y_j]$ be given by $\psi_j = \phi_j^{-1}$.

For $0 \leq z < y_{j-1}\vee y_j-y_{j-1}\wedge y_j$ define $H_j(z) = \mu ( \inf_{w \leq z} \psi_j(w),\sup_{w \leq z} \psi_j(w)) = F_\mu(\sup_{w \leq z} \psi_j(w)) - F_\mu(\inf_{w \leq z} \psi_j(w))$. 

Now fix $y\in \sY^*_\infty$.

Define $\phi_\infty :  [\alpha_\mu,\beta_\mu]\cap\R \to [0,\beta_\mu-\alpha_\mu]\cap\R$ as follows: \\
$\phi_\infty(y_0)= 0$;
for $y_0<z<\beta_\mu$, so that $z \in (y_{2k-1}, y_{2k+1}]$ for some $k\geq0$, $\phi_\infty(z) = z - y_{2k}$; for $\alpha_\mu < z < y_0$ so that $z \in [y_{2k+2}, y_{2k})$ for some $k\geq0$, $\phi_\infty(z) = y_{2k+1}-z$; finally, if $\beta_\mu<\infty$ (resp. $\alpha_\mu>-\infty$), set $\phi_\infty(\beta_\mu) =\beta_\mu-\alpha_\mu$ (resp. $\phi_\infty(\alpha_\mu) =\beta_\mu-\alpha_\mu$).

Let $\psi_\infty: [0,\beta_\mu-\alpha_\mu]\cap{\R_+} \to [\alpha_\mu,\beta_\mu]\cap\R$ be given by $\psi_\infty = \phi_\infty^{-1}$.

For $0 \leq z < \beta_\mu-\alpha_\mu$ define $H_\infty(z) = \mu ( \inf_{w \leq z} \psi_\infty(w),\sup_{w \leq z} \psi_\infty(w)) = F_\mu(\sup_{w \leq z} \psi_\infty(w)) - F_\mu(\inf_{w \leq z} \psi_\infty(w))$. 
\end{defn}


\begin{figure}[H]
	\centering
\begin{tikzpicture}

\begin{axis}[%
width=6.521in,
height=3.566in,
at={(0.758in,0.481in)},
scale only axis,
xmin=-2.5,
xmax=21,
ymin=-3,
ymax=11,
axis line style={draw=none},
ticks=none
]
\draw[thin,gray] (0,0)--(20,0);
\draw[thin,gray] (20,0)--(20,10);
\draw[thin,gray] (20,10)--(0,10);
\draw[thin,gray] (0,10)--(0,0);
\draw[thin,gray] (10,0)--(10,10);
\node[left,scale=1] at (0,7) {$x_3-x_2$};
\node[left,scale=1] at (0,3) {$x_1-x_2$};
\node[left,scale=1] at (0,1) {$x_1-x_0$};
\node[left,scale=1] at (0,0) {$0$};
\node[below,scale=1] at (4,0) {$x_0$};
\node[below,scale=1] at (5,0) {$x_1$};
\node[below,scale=1] at (2,0) {$x_2$};
\node[below,scale=1] at (9,0) {$x_3$};
\node[below,scale=0.9] at (13.5,0) {$x_0=\psi_j(0)$};
\node[below,scale=0.9] at (15,-1) {$x_1=\psi_j(x_1-x_0)$};
\node[below,scale=0.9] at (12,-2) {$x_2=\psi_j(x_1-x_2)$};
\node[below,scale=0.9] at (19,0) {$x_3=\psi_j(x_3-x_2)$};
\draw[thick,black!30!green] (4,0)--(5,1);
\draw[thick,black!30!green] (4,1)--(2,3);
\draw[thick,black!30!green] (5,3)--(9,7);
\draw[blue,densely dotted, thick] (10,0)--(20,10);
\node[black!30!green] at (5,6) {$x\mapsto\phi_j(x)$};
\draw[gray,thin,dashed]  (0,1)--(5,1) -- (5,3)--(0,3);
\draw[gray,thin,dashed]  (4,0)--(4,1);
\draw[gray,thin,dashed]  (5,0)--(5,1);
\draw[gray,thin,dashed]  (0,7)--(9,7)--(9,0);
\draw[red,thick] (14,4) to[out=70, in=185] (15,5);
\draw[red,thick] (14,4) to[out=300, in=170] (15,3.5);

\draw[red,thick] (14,3.5) to[out=240, in=20] (12,2);
\draw[red,thick] (14,5) to[out=110, in=350] (12,6);

\draw[red,thick] (15,6) to[out=40, in=190] (19,9);
\draw[red,thick] (15,2) to[out=350, in=180] (19,1);
\draw[gray,thin,dashed]  (12,-2)--(12,6) -- (15,6)--(15,-1);
\draw[gray,thin,dashed]  (12,2)--(15,2);
\draw[gray,thin,dashed]  (15,5)--(14,5)--(14,3.5)--(15,3.5);
\draw[gray,thin,dashed]  (14,3.5)--(14,0);
\draw[gray,thin,dashed]  (19,9)--(19,0);
\node[red] at (14,7) {$x\mapsto N(x)$};
\node[red] at (17,3) {$x\mapsto M(x)$};
\node[circle,fill=black!30!green,inner sep=0pt,minimum size=5pt] at (4,0) {};
\node[circle,fill=black!30!green,inner sep=0pt,minimum size=5pt] at (5,1) {};
\node[circle,fill=black!30!green,inner sep=0pt,minimum size=5pt] at (2,3) {};
\node[circle,fill=black!30!green,inner sep=0pt,minimum size=5pt] at (9,7) {};
\draw [color=black!30!green,fill=white] (4,1) circle[radius= 0.25 em];
\draw [color=black!30!green,fill=white] (5,3) circle[radius= 0.25 em];
\draw[gray,thin,dashed]  (2,3)--(2,0);
\end{axis}
\end{tikzpicture}
\caption{Plots of $\phi_j$ (see the left figure) and increasing and decreasing functions $N\circ\psi_j$ and $M\circ\psi_j$ (see the right figure) for $j=3$ and $(x_0,x_1,x_2,x_3)\in\sY_j$. In the right figure the solid curve above (resp. below) the diagonal corresponds to the graph of $x\mapsto N(x)$ (resp. $x\mapsto M(x)$). Note that $N$ (resp. $M$) is non-decreasing (resp. non-increasing) on $[x_0,x_1]$, non-increasing (resp. non-decreasing) on $[x_2,x_0]$ and again non-decreasing (resp. non-increasing) on $[x_1,x_3]$.}
\label{fig:phiMN}
\end{figure}
For each fixed $1\leq j<\infty$ and $y\in\sY_j$ (resp. $y\in\sY^*_\infty$) we have that $\phi_j : \R \cap [y_{j-1}\wedge y_j,y_{j-1}\vee y_j] \to [0,y_{j-1}\vee y_j-y_{j-1}\wedge y_j]\cap\R$ (resp. $\phi_\infty :  [\alpha_\mu,\beta_\mu]\cap\R \to [0,\beta_\mu-\alpha_\mu]\cap\R$) is a bijection and thus $\psi_j=\phi_j^{-1}$ (resp. $\psi_\infty=\phi^{-1}_\infty$) is well-defined (see Figure \ref{fig:phiMN}), while $H_j:[0, y_{j-1}\vee y_j-y_{j-1}\wedge y_j) \to [0,F_\mu(y_{j-1}\vee y_j)-F_\mu(y_{j-1}\wedge y_j))\subseteq[0, \mu(\R))$ (resp. $H_\infty:[0,\beta_\mu-\alpha_\mu) \to [0, \mu(\R))$) is continuous and increasing. Furthermore, for $k \geq 1$, if $(y_{j+1}, \ldots y_{j+k})$ (resp. $(y_{j+1},y_{j+2}...)$) is such that $(y,y_{j+1}, \ldots y_{j+k})\in\sY_{j+k}$ (resp. $(y,y_{j+1},y_{j+2},...)\in\sY^*_\infty$), then $\phi_{j+k}= \ldots \phi_{j+1}=\phi_j$ (resp. $\phi_{\infty}=\phi_j$) on  $[y_{j-1}\wedge y_j,y_{j-1}\vee y_j]$, $\psi_{j+k} = \ldots =\psi_{j+1}=\psi_j$ (resp. $\psi_{\infty}=\psi_j$) on $[0,y_{j-1}\vee y_j-y_{j-1}\wedge y_j]$, and $H_{j+k}=H_j$ (resp. $H_{\infty}=H_j$ ) on $[0,y_{j-1}\vee y_j-y_{j-1}\wedge y_j)$.

Having constructed the appropriate spaces and given the necessary definitions, now we connect these ideas with our construction.

Suppose $(\mu,\nu) \in \sK_*$ so that $\overline{w}_{x_0,x_0}>x_0$.
Set $x_1 = \overline{w}_{x_0,x_0}$. We define $M,N$ on $[x_0, x_1=\overline{w}_{x_0,x_0}]\cap\R$ via $M(x) = \overrightarrow{m}_{x_0,x_0}(x)$ and $N(x) = \overrightarrow{n}_{x_0,x_0}(x)$. 
{Then either $-\infty<x_0=\alpha_\mu$, in which case, by Lemma~\ref{lem:w=muR}, $x_1=\beta_\mu$ and $(x_0,x_1)\in\sY^{\{0,1\}}_1$, or $x_0>\alpha_\mu$, so that $(x_0,x_1)\in \sY^{\emptyset}_1 \cup \sY^{\{1\}}_1$. In the former case $M$ and $N$ are defined on $[\alpha_\mu,\beta_\mu]\cap\R$ and the construction terminates.}
In the latter case  by Lemma \ref{lem:strongermonotonicity} we must have that $x_1<\infty$ and we then define $x_2 = \underline{w}_{x_0,x_1}$ (which will satisfy $x_2<x_0$), and $M(x) = \overleftarrow{m}_{x_0,x_1}(x)$ and $N(x) = \overleftarrow{n}_{x_0,x_1}(x)$ on $[x_2,x_0)\cap\R$.
If $(x_0,x_1) \in \sY_1^{\{1\}}$, so that $x_1 = \beta_\mu <\infty$, then $x_2=\alpha_\mu$ (again by Lemma \ref{lem:w=muR}), $(x_0, x_1, x_2) \in \sY_2^{ \{1,2 \}}$, $M$ and $N$ are defined on $[\alpha_\mu,\beta_\mu]\cap\R$ and the construction ends. Otherwise, $(x_0,x_1) \in \sY_1^{\emptyset}$ (equivalently $x_1 < \beta_\mu$) and $(x_0, x_1, x_2) \in \sY_2^{ \emptyset } \cup \sY_2^{ \{2 \}}$ (and since $x_1 < \beta_\mu\leq\infty$, by Lemma \ref{lem:strongermonotonicity} again, we have that $x_2>-\infty$). 

We proceed inductively: given $j$ even with $(x_0,...,x_j)\in \sY^{ \emptyset }_j \cup \sY^{ \{j \}}_j$ and $x_j>-\infty$, we define $x_{j+1}>x_{j-1}$ via $x_{j+1} = \overline{w}_{x_j,x_{j-1}}$ and $M,N$ on $(x_{j-1},x_{j+1}]\cap\R$ by $M(x) = \overrightarrow{m}_{x_j,x_{j-1}}(x)$ and $N(x) = \overrightarrow{n}_{x_j,x_{j-1}}(x)$; given $j$ odd with $(x_0,...,x_j)\in \sY^{ \emptyset }_j \cup \sY^{ \{j \}}_j$ and $x_j<\infty$, we define $x_{j+1}<x_{j-1}$ via $x_{j+1} = \underline{w}_{x_{j-1},x_{j}}$ and $M,N$ on $[x_{j+1},x_{j-1})\cap\R$ by $M(x) = \overleftarrow{m}_{x_{j-1},x_{j}}(x)$ and $N(x) = \overleftarrow{n}_{x_{k-1},x_{k}}(x)$. In this way we construct $(x_0,x_1,\ldots x_{j+1}) \in  \sY^{ \emptyset }_{j+1} \cup \sY^{ \{j+1 \}}_{j+1}\cup \sY^{ \{j,j+1 \}}_{j+1}$ (in the case $(x_0,x_1,\ldots x_{j+1}) \in  \sY^{ \emptyset }_{j+1}\cup\sY_{j+1}^{ \{ j+1 \} }$, Lemma \ref{lem:strongermonotonicity} ensures that $\lvert x_{j+1}\lvert <\infty$) and $M,N$ defined on $[x_{j+1},x_{j}]\cap\R$ or $[x_{j},x_{j+1}]\cap\R$.

If $(x_0,...,x_j)\in\sY^{ \{j-1, j\}}_j$ for some $j\geq1$, then the construction terminates. Otherwise the construction continues indefinitely. Note that, due to Lemma \ref{lem:w=muR}, we never have $(x_0,...,x_j)\in\sY^{ \{j-1\}}_j$, for if 
$x_{j-1} \in \{\alpha_\mu,\beta_\mu\}$, then $x_{j} \in \{\alpha_\mu,\beta_\mu\}$ and thus $(x_0,...,x_j)\in\sY^{ \{j-1, j\}}_j$.

Fix (a finite) $j\in\mathbb N$. 
For $j$ even with $j \geq 2$ let $\sP_E(j,(x_0, x_1, \ldots, x_j))$ be the statement 
\begin{enumerate}
\item $(x_0, \ldots x_j) \in  \sY^{\emptyset}_j \cup \sY^{ \{ j \}}_j\cup\sY^{ \{ j-1,j \}}_j$; 
{\item $\overleftarrow{m}_{x_{j-2},x_{j-1}}(x_j) < \underline{x}_0$;}
\item {$N \circ \psi_j$ (respectively, $M \circ \psi_j$) is increasing (respectively, decreasing) on $[0,x_{j-1}-x_j]\cap\R$;} 
\item $S^\nu(\mu_{x_{j-1}} - \mu_{x_j}) = \nu|_{(M(x_j),N(x_j))}$;
\item if $(x_0, \ldots x_j) \in \sY^{\emptyset}_j \cup \sY^{ \{j \} }_j$, then {$x_j > -\infty$} and $N(x_j):=\overleftarrow{n}_{x_{j-2},x_{j-1}}(x_j) > \overleftarrow{G}_\mu(F_\mu(x_{j-1}))$ and $\overleftarrow{G}_\mu(F_\mu(x_{j})-) \leq M(x_j) = \overleftarrow{m}_{x_{j-2},x_{j-1}}(x_j) \leq \overleftarrow{G}_\mu(F_\mu(x_{j}))  $.
\end{enumerate}

For $j$ odd let $\sP_O(j,(x_0, x_1, \ldots, x_j))$ be the statement {(to cover $j=1$ we define $x_{-1}:=x_0$)}
\begin{enumerate}
\item $(x_0, \ldots x_j) \in {\sY^{\emptyset}_j \cup \sY^{ \{ j \}}_j\cup\sY^{ \{ j-1,j \}}_j}$
\item {if $j \geq 3$}, $\overleftarrow{m}_{x_{j-3},x_{j-2}}(x_{j-1}) < \underline{x}_0$; 
\item {$N \circ \psi_j$ (respectively, $M \circ \psi_j$) is increasing (respectively, decreasing) on $[0,x_{j}-x_{j-1}]\cap\R$}; 
\item $S^\nu(\mu_{x_{j}} - \mu_{x_{j-1}}) = \nu|_{(M(x_j),N(x_j))}$;
\item if $(x_0, \ldots x_j) \in { \sY^{\emptyset}_j \cup \sY^{ \{j \} }_j}$, then {$x_j < \infty$} and $M(x_j):=\overrightarrow{m}_{x_{j-1},x_{j-2}}(x_j) < \overrightarrow{G}_\mu(F_\mu(x_{j-1}))$ and $\overrightarrow{G}_\mu(F_\mu(x_{j})) \leq x_{j} \leq N(x_j) = \overrightarrow{n}_{x_{j-1},x_{j-2}}(x_j) \leq \overrightarrow{G}_\mu(F_\mu(x_{j})+)$.
\end{enumerate}

See Figure \ref{fig:phiMN} for the stylized graphs of the increasing and decreasing maps $N\circ\psi_j$ and $M\circ\psi_j$, respectively. 

\begin{lem}\label{lem:m=mn=n}Fix $j\geq1$ and suppose $(x_0,...,x_j)\in { \sY^{\emptyset}_j \cup \sY^{ \{j \} }_j}$.

Suppose $j$ is even and $\sP_E(j,(x_0, x_1, \ldots, x_j))$ holds. Then $\overrightarrow{n}_{x_j,x_{j-1}}(x_{j-1}) = \overleftarrow{n}_{x_{j-2},x_{j-1}}(x_j)$ and $\overrightarrow{m}_{x_j,x_{j-1}}(x_{j-1}) = \overleftarrow{m}_{x_{j-2},x_{j-1}}(x_j)$.

Suppose $j$ is odd and $\sP_O(j,(x_0, x_1, \ldots, x_j))$ holds. Then for $j \geq 3$, $\overleftarrow{m}_{x_{j-1},x_{j}}(x_{j-1}) = \overrightarrow{m}_{x_{j-1},x_{j-2}}(x_j)$ and $\overleftarrow{n}_{x_{j-1},x_{j}}(x_{j-1}) = \overrightarrow{n}_{x_{j-1},x_{j-2}}(x_j)$. For $j=1$, $\overleftarrow{m}_{x_{0},x_{1}}(x_{0}) = \overrightarrow{m}_{x_0,x_{0}}(x_1)$ and $\overleftarrow{n}_{x_{0},x_{1}}(x_{0}) = \overrightarrow{n}_{x_{0},x_{0}}(x_1)$.
\end{lem}

\begin{proof}
Suppose $j\geq 2$ is even. Note that, since $\sP_E(j,(x_0, x_1, \ldots, x_j))$ holds, $x_{j-1} < \beta_{\mu}$ and by Lemma~\ref{lem:strongermonotonicity} $x_j> - \infty$. Then, by hypothesis,
$$
\overleftarrow{m}_{x_{j-2},x_{j-1}}(x_j)\leq x_j\leq x_{j-1}<\overleftarrow{n}_{x_{j-2},x_{j-1}}(x_j).
$$
Note that, by Lemma \ref{lem:E_wv-linear}, $\sE^c_{x_j,x_{j-1}}$ is linear on $[\overleftarrow{m}_{x_{j-2},x_{j-1}}(x_j),\overleftarrow{n}_{x_{j-2},x_{j-1}}(x_j)]$. But, by Lemma \ref{lem:E_vw-linear}, $[\overrightarrow{m}_{x_j,x_{j-1}}(x_{j-1}), \overrightarrow{n}_{x_j,x_{j-1}}(x_{j-1})]$ is the largest interval containing $x_{j-1}$ on which $\sE^c_{x_j,x_{j-1}}$ is linear, and therefore
$$
\overrightarrow{m}_{x_j,x_{j-1}}(x_{j-1})\leq \overleftarrow{m}_{x_{j-2},x_{j-1}}(x_j)<\overleftarrow{n}_{x_{j-2},x_{j-1}}(x_j)\leq \overrightarrow{n}_{x_j,x_{j-1}}(x_{j-1}).
$$
{Since $(\sE^c_{x_j,x_{j-1}})'$ is continuous everywhere and constant on  $[\overrightarrow{m}_{x_j,x_{j-1}}(x_{j-1}), \overrightarrow{n}_{x_j,x_{j-1}}(x_{j-1})]$, and since $x_j$ and $x_{j-1}$ lie in this interval, we have that} $L_{\sE^c_{x_j,x_{j-1}}}^{x_{j-1},(\sE^c_{x_j,x_{j-1}})'(x_{j-1})}
=L_{\sE^c_{x_j,x_{j-1}}}^{x_{j},(\sE^c_{x_j,x_{j-1}})'(x_{j})}$. Then
$$
\overrightarrow{m}_{x_j,x_{j-1}}(x_{j-1})=X^-_{\sE^{}_{x_j,x_{j-1}}}(x_{j-1})=X^-_{\sE^{}_{x_j,x_{j-1}}}(x_{j})= \overleftarrow{m}_{x_{j-2},x_{j-1}}(x_j)$$
and
$$\overleftarrow{n}_{x_{j-2},x_{j-1}}(x_j)=Z^+_{\sE^{}_{x_j,x_{j-1}}}(x_{j})=Z^+_{\sE^{}_{x_j,x_{j-1}}}(x_{j-1})= \overrightarrow{n}_{x_j,x_{j-1}}(x_{j-1}).
$$

The result for $j$ odd follows symmetrically.
\end{proof}

\begin{prop}\label{prop:PEtoPO} Suppose $(\mu,\nu) \in \sK_{R}$.

Suppose $j\geq 2$ is even and $(x_0, x_1, \ldots, x_{j}) \in {\sY^{\emptyset}_{j}}$. Suppose $\sP_E(j,(x_0, x_1, \ldots, x_{j}))$ holds. Then $(x_0, x_1, \ldots, x_{j},\overline{w}_{x_{j},x_{j-1}}) \in {\sY^{\emptyset}_{j+1} \cup \sY^{ \{ j+1 \} }_{j+1}}$ and
$\sP_O(j+1,(x_0, x_1, \ldots, x_{j},\overline{w}_{x_{j},x_{j-1}}))$ holds.

{Suppose $j\geq 2$ is even and $(x_0, x_1, \ldots, x_{j}) \in \sY^{ \{ j \} }_{j}$. Suppose $\sP_E(j,(x_0, x_1, \ldots, x_{j}))$ holds. Then $(x_0, x_1, \ldots, x_{j},\overline{w}_{x_{j},x_{j-1}}) \in \sY^{ \{ j, j+1 \} }_{j+1}$ and
$\sP_O(j+1,(x_0, x_1, \ldots, x_{j},\overline{w}_{x_{j},x_{j-1}}))$ holds.}

\end{prop}
\begin{proof}
Since $(x_0,...,x_j)\in  \sY^{\emptyset}_j \cup \sY^{ \{j \}}_j$ we have that $x_{j-1} < \beta_\mu$,  and, by Lemma~\ref{lem:strongermonotonicity}, $x_j>-\infty$.

Set $x_{j+1} = \overline{w}_{x_j,x_{j-1}}$. 
Since $\sP_E(j,(x_0, x_1, \ldots, x_{j}))$ holds $N(x_j)=\overleftarrow{n}_{x_{j-2},x_{j-1}}(x_j) > \overrightarrow{G}_\mu(F_\mu(x_{j-1})+)$ 
and then by Lemma~\ref{lem:m=mn=n}, $\overrightarrow{n}_{x_j,x_{j-1}}(x_{j-1}) { = \overleftarrow{n}_{x_{j-2},x_{j-1}}(x_j) }> \overrightarrow{G}_\mu(F_\mu(x_{j-1})+)$. 
It follows from Lemma~\ref{lem:w_vu>u} that $x_{j+1} = \overline{w}_{x_j,x_{j-1}}>x_{j-1}$ (and that $F_\mu(x_{j+1}) > F_{\mu}(x_{j-1})$) and from Lemma~\ref{lem:w=muR} that if $-\infty<x_j = \alpha_\mu$ then $x_{j+1}= \beta_\mu>x_{j-1}$. 
Hence it follows that 
{if $(x_0, x_1, \ldots, x_{j}) \in \sY^{ \{ j \} }_j$ then $(x_0, x_1, \ldots, x_{j},\overline{w}_{x_{j},x_{j-1}}) \in {\sY}^{ \{j,j+1\} }_{j+1}$, and if $(x_0, x_1, \ldots, x_{j}) \in \sY^{ \emptyset }_j$ then $(x_0, x_1, \ldots, x_{j},\overline{w}_{x_{j},x_{j-1}}) \in \sY^{ \emptyset}_{j+1} \cup \sY^{ \{ j+1 \} }_{j+1}$, according as either $\overline{w}_{x_{j},x_{j-1}} < \beta_\mu$ or $\overline{w}_{x_{j},x_{j-1}} = \beta_\mu$. }


The condition that $\overleftarrow{m}_{x_{ j-2},x_{ j-1}}(x_{ j}) < \underline{x}_0$ is inherited directly from the inductive hypothesis.



{By Lemma~\ref{lem:m_vu-decreasing}} we have
$\overrightarrow{m}_{x_j,x_{j-1}}(x_{j-1}+) \leq \overrightarrow{m}_{x_j,x_{j-1}}(x_{j-1})$ and then by Lemma~\ref{lem:m=mn=n} and the inductive hypothesis,
\begin{equation}
\label{eq:mGcomparison}
\overrightarrow{m}_{x_j,x_{j-1}}(x_{j-1}+) \leq \overrightarrow{m}_{x_j,x_{j-1}}(x_{j-1}) = \overleftarrow{m}_{x_{j-2},x_{j-1}}(x_{j}) \leq x_{j} = \overrightarrow{G}_\mu(F_\mu(x_{j})+) .
\end{equation}
Then by Corollary~\ref{cor:n_vu-increasing}, $\overrightarrow{n}_{x_j,x_{j-1}}$ 
is increasing on $[x_{j-1},x_{j+1}]\cap\R$.
Moreover,  $\overrightarrow{n}_{x_j,x_{j-1}}(x_{j-1}) = \overleftarrow{n}_{x_{j-2},x_{j-1}}(x_j)$ and $N\circ \psi_j$ is increasing on $[0,x_{j-1}-x_j]$. Using these facts we now show that $N\circ\psi_{j+1}$ is increasing on $[0, x_{j+1}-x_j]\cap\R$.

Since $N\circ \psi_j$ is increasing on $[0,x_{j-1}-x_j]$
it is sufficient to show that for $w \in \R$ with $x_{j-1} - x_j \leq v < w \leq x_{j+1} - x_j$ and for $w \in \R$ with $0 \leq v \leq x_{j-1} - x_j < w \leq x_{j+1} - x_j$ 
we have that $N \circ \psi_{j+1} (v) \leq N \circ \psi_{j+1} (w)$. But, for $w \in \R$ with $x_{j-1} - x_j \leq v < w \leq x_{j+1} - x_j$ the result follows from the monotonicity of
$\overrightarrow{n}_{x_j,x_{j-1}}$ whence 
\[ N \circ \psi_{j+1}(v) = N(v - (x_0-x_j)) = \overrightarrow{n}_{x_j,x_{j-1}}(v-(x_0 - x_j)) \leq \overrightarrow{n}_{x_j,x_{j-1}}(w-(x_0 - x_j)) = N \circ \psi_{j+1}(w). \] Similarly, for $w \in \R$ with $x_j \leq v \leq x_{j-1} < w \leq x_{j+1}$ 
we have $N \circ \psi_{j+1} (v)=N\circ\psi_{j}(v) \leq N \circ \psi_j (x_{j-1} - x_j) = \overleftarrow{n}_{x_{j-2},x_{j-1}}(x_{j}) = \overrightarrow{n}_{x_j,x_{j-1}}(x_{j-1}) \leq \overrightarrow{n}_{x_j,x_{j-1}}(w - (x_{j-1}-x_j)) = N\circ \psi_{j+1} (w)$, where the first inequality follows from the fact that $N \circ \psi_{j}$ is increasing on $[0, x_{j-1}-x_j]$ and the second inequality follows from the monotonicity of $\overrightarrow{n}_{x_j,x_{j-1}}$ on $[x_{j-1},x_{j+1}]\cap\R$.

The proof that $M\circ \psi_{j+1}$ is decreasing is similar.

The fact that $S^\nu(\mu_{x_{j+1}} - \mu_{{x_j}}) = \nu|_{(M(x_{j+1}),N(x_{j+1}))}$ follows from Lemma~\ref{lem:properties1} {where we use the result $\overrightarrow{m}_{x_j,x_{j-1}}(x_{j-1}+) \leq \overrightarrow{G}_\mu(F_\mu(x_{j})+)$ derived in \eqref{eq:mGcomparison} to verify the hypotheses of the lemma.}

{If $-\infty<x_j=\alpha_\mu$ then $\overline{w}_{x_j,x_{j-1}} = \beta_\mu$ and $(x_0, x_1, \ldots, x_{j}=\alpha_\mu,x_{j+1}= \beta_\mu) \in \sY^{ \{j,j+1 \} }_{j+1}$. In that case we do not need to check the final statement of the inductive hypothesis and the proof is complete. So, suppose that $x_j > \alpha_\mu$.
Then,} $(x_0, x_1, \ldots, x_{j},\overline{w}_{x_{j},x_{j-1}})\in \sY_{j+1}^{\emptyset} \cup \sY_{j+1}^{ \{j+1 \} }$. 
In order to show that 
$\sP_O(j+1,(x_0, x_1, \ldots, x_{j},\overline{w}_{x_{j},x_{j-1}}))$ holds it only remains to show that {$x_{j+1}<\infty$, that} $\overrightarrow{G}_\mu(F_\mu(x_{j+1})) \leq N(x_{j+1}) = \overrightarrow{n}_{x_{j},x_{j-1}}(x_{j+1}) \leq \overrightarrow{G}_\mu(F_\mu(x_{j+1})+)$ and that $M(x_{j+1})=\overrightarrow{m}_{x_{j},x_{j-1}}(x_{j+1}) < \overrightarrow{G}_\mu(F_\mu(x_{j}))$. Since $x_j > \alpha_\mu$ it follows directly from Lemma \ref{lem:strongermonotonicity} that $x_{j+1} < \infty$. The inequalities for $N$ follow immediately from Lemma~\ref{lem:w_vu>u}. Finally, for $M$, using the monotonicity of $\overrightarrow{m}_{x_j,x_{j-1}}$ (see Corollary \ref{cor:strictmon}), the fact that $F_\mu(x_{j+1}) > F_\mu(x_{j-1})$, and Lemma~\ref{lem:m=mn=n}, we have that 
\[ \overrightarrow{m}_{x_j,x_{j-1}}(x_{j+1}) <  \overrightarrow{m}_{x_j,x_{j-1}}(x_{j-1}) = \overleftarrow{m}_{x_{j-2},x_{j-1}}(x_{j}) \leq x_j \leq \overrightarrow{G}_\mu(F_\mu(x_{j})+).
\]
Then by Lemma~\ref{lem:strongermonotonicity} we have that $\overrightarrow{m}_{x_j,x_{j-1}}(x_{j+1}) <  \overrightarrow{G}_{\mu}(F_\mu(x_{j}))$, which finishes the proof.
\end{proof}

The following lemma, is the parallel result for odd $j\geq1$.

\begin{prop}\label{prop:POtoPE} Suppose $(\mu,\nu) \in \sK_{R}$.

{Suppose $j\geq1$ is odd and $(x_0, x_1, \ldots, x_{j}) \in \sY^\emptyset_{j}$. Suppose $\sP_{O}(j,(x_0, x_1, \ldots, x_{j}))$ holds, {and, if $j \geq 3$, also $\sP_E(j-1,(x_0, x_1, \ldots, x_{j-1})$ holds}. Then,
$(x_0, x_1, \ldots, x_{j},\underline{w}_{x_{j-1},x_{j}}) \in \sY^\emptyset_{j+1} \cup \sY^{ \{ j+1 \} }_{j+1} $ and
$\sP_{E}(j+1,(x_0, x_1, \ldots, x_{j},\underline{w}_{x_{j-1},x_{j}}))$ holds.

Suppose $j\geq1$ is odd and $(x_0, x_1, \ldots, x_{j}) \in \sY^{ \{ j \} }_{j}$. Suppose $\sP_{O}(j,(x_0, x_1, \ldots, x_{j}))$ holds, {and, if $j \geq 3$, also $\sP_E(j-1,(x_0, x_1, \ldots, x_{j-1})$ holds}. Then,
$(x_0, x_1, \ldots, x_{j},\underline{w}_{x_{j-1},x_{j}}) \in \sY^{ \{j,j+1 \}}_{j+1}$ and
$\sP_{E}(j+1,(x_0, x_1, \ldots, x_{j},\underline{w}_{x_{j-1},x_{j}}))$ holds.

}

\end{prop}

\begin{proof}
{\em Mutatis mutandis}, the majority of the proof is identical. The only exception is that it is now additionally necessary to show that {$\overleftarrow{m}_{x_{j-1}x_j}(x_{j+1})<\underline{x}_0$.}

Suppose $j \geq 3$. Then, by monotonicity of $\overleftarrow{m}_{x_{j-1},x_j}$ and $\overrightarrow{m}_{x_{j-1}x_{j-2}}$, Lemma~\ref{lem:m=mn=n} (twice) and the fact that we are assuming that $\sP_E(j-1,(x_0, x_1, \ldots, x_{j-1})$ holds,
$$\overleftarrow{m}_{x_{j-1},x_j}(x_{j+1}) \leq \overleftarrow{m}_{x_{j-1},x_j}(x_{ j-1}) = \overrightarrow{m}_{x_{j-1},x_{j-2}}(x_{j}) \leq
\overrightarrow{m}_{x_{j-1},x_{j-2}}(x_{j-2}) = \overrightarrow{m}_{x_{j-3},x_{j-2}}(x_{j-1}) < \tilde x_0.$$

The remaining case is when $j=1$. {Note that, since $(x_0,...,x_j)=(x_0,x_1)\in \sY^{\emptyset}_1 \cup \sY^{ \{ 1 \} }_1$, we must have $x_0>\alpha_\mu$.}
We are required to show that whenever $x_0>\alpha_\mu$ we have $\overleftarrow{m}_{x_{0},x_1}(x_{2})<\tilde x_0$.
Since $(\mu,\nu) \in \sK_R$,  $\sE_{x_0,x_1}$ lies below the tangent $L_{\tilde x_0}$ (to $D_{\mu,\nu}$ at $\tilde x_0$) to the right of $x_0$. Now suppose $x_1>x_0>\alpha_\mu$ is fixed and consider $(\sE_{u,x_1})_{u \leq x_0}$ and $D_{\mu,\nu}$. Note that $\sE_{x_2,x_1} = D_{\mu,\nu}$ on $[\overleftarrow{G}_\mu(x_2-)=\overrightarrow{G}_\mu(x_2) ,\overleftarrow{G}_\mu(x_2)]$. We have that $D_{\mu,\nu}'(\overleftarrow{G}_\mu(x_2-)) \leq  D'(\overleftarrow{m}_{x_0,x_1}(x_2))) \leq D'(\overleftarrow{G}_\mu(x_2))$. But then $L_{\overleftarrow{m}_{x_0,x_1}(x_2)}$ has slope less than $D_{\mu,\nu}'(x_0)=D_{\mu,\nu}'(\tilde x_0)$: if not then $L_{\overleftarrow{m}_{x_0,x_1}(x_2)}$ cannot touch $\sE_{x_2,x_1}$ to the right of $x_0$. In particular, $\overleftarrow{m}_{x_0,x_1}(x_2) < \tilde x_0$. 
\end{proof}

{It remains to show that the statement $\sP_O(1,(x_0,x_1))$ holds. 
However, there is no guarantee that $x_1 = \overline{w}_{x_0,x_0} > x_0$. Hence we introduce:
\begin{defn}\label{def:K*}
{$\sK_* = \{ (\mu,\nu) \in \sK_{R} : \overline{w}_{x_0,x_0}> x_0 \}$.}

\end{defn}}


\begin{prop}\label{prop:firststep}
Suppose $(\mu,\nu) \in \sK_*$. Then $(x_0,x_1 = \overline{w}_{x_0,x_0})\in\sY_1\setminus\sY_1^{\{ 0 \} }$ and $\sP_O(1,(x_0,x_1 = \overline{w}_{x_0,x_0}))$ holds.
\end{prop}

\begin{proof} The proof is a simplified version of the proof of Proposition~\ref{prop:PEtoPO}.

{It follows from Lemma~\ref{lem:w=muR} that, if $-\infty<x_0 = \alpha_\mu$ then $x_1 = \beta_\mu$, and then $(x_0,x_1)\in\sY^{ \{0,1 \}}_1$.} On the other hand, if $x_0>\alpha_\mu$, then,
since $\overline{w}_{x_0,x_0} > x_0$ by hypothesis, it follows that $(x_0,x_1) \in \sY^{\emptyset}_1 \cup \sY^{ \{ 1 \}}_1$. Note that we do not need to check the second item of $\sP_O(1,(x_0,x_1 = \overline{w}_{x_0,x_0}))$.

By Lemma~\ref{lem:properties1} part (ii) and Corollary \ref{cor:strictmon}, $M$ (resp. $N$) is decreasing (resp. increasing) on $(x_0,x_1]\cap\R$. On the other hand, by Lemma~\ref{lem:properties1} part (i), $S^\mu(\mu_{x_1} - \mu_{x_0}) = \nu|_{(M(x_1),N(x_1))}$.  

Now we verify the last statement. Suppose $(x_0,x_1)\in\sY^{\emptyset}_1 \cup \sY^{ \{ 1 \}}_1$, so that $x_0>\alpha_\mu$, and then Lemma \ref{lem:strongermonotonicity} ensures that $x_1<\infty$.
{By Lemma~\ref{lem:splitpart2}, $\overrightarrow{G}_\mu(F_\mu(\overline{w}_{x_0,x_0})) \leq \overline{w}_{x_0,x_0}\leq \overrightarrow{n}_{x_0,x_0}(\overline{w}_{x_0,x_0})\leq \overrightarrow{G}_\mu(F_\mu(\overline{w}_{x_0,x_0})+)$.
It only remains to show that 
if $x_0 > \alpha_\mu$ then 
$M(x_1) = \overrightarrow{m}_{x_0,x_0}(x_1) < \overrightarrow{G}_\mu(F_\mu(x_0))$.
Note that by the definition of $x_0$ as the maximal element such that $\mu \leq \nu$ on $(-\infty, \cdot)$ we must have that for all $x>x_0$, $F_\mu(x)>F_\mu(x_0)$. In particular, $F_\mu(x_1)>F_\mu(x_0)$. Then, by Corollary~\ref{cor:strictmon}, 
$\overrightarrow{G}_\mu(F_\mu(x_0)) = \overrightarrow{m}_{x_0,x_0}(x_0) > \overrightarrow{m}_{x_0,x_0}(x_1)$.
}

\end{proof}

\begin{thm}\label{thm:K*}
Suppose $(\mu,\nu) \in \sK_*$. Then the pair $(M,N)$ defines a martingale coupling of $\mu$ and $\nu$.
\end{thm}

\begin{proof}
By Proposition~\ref{prop:firststep}, $(x_0,x_1 = \overline{w}_{x_0,x_0})\in\sY_1\setminus\sY_1^{ \{ 0 \} }$ and $\sP_O(1,(x_0,x_1))$ holds. Then, either $(x_0,x_1)\in\sY^{ \{0,1 \}}_1$ (so that $x_0=\alpha_\mu$ and $x_1=\beta_\mu$), or $(x_0,x_1)\in \sY^{\emptyset}_1 \cup \sY^{ \{1 \} }_1$ and we can perform at least one more iteration. In particular, alternating between Propositions~\ref{prop:POtoPE} and \ref{prop:PEtoPO} it follows that, either there exists a finite $J\geq1$ such that $(x_0,...,x_j)\in \sY^{\emptyset}_j$ for each $j \leq J-2$, $(x_0,...,x_{J-1})\in \sY^{\{ J-1 \} }_{J-1}$ and $(x_0,...,x_J)\in\sY^{ \{J-1,J \} }_J$, or  $(x_0,...,x_j)\in \sY^\emptyset_j$ for all $j\geq1$ and then we set $J=\infty$. In both cases the statement $\sP_E(j,(x_0,x_1, \ldots x_j))$ holds if $j< J$ (and also $j=J$ if $J<\infty$) is even and $\sP_O(j,(x_0,x_1, \ldots x_j))$ holds if $j< J$ (and also $j=J$ if $J<\infty$) is odd.


If $J<\infty$ then the construction terminates. If $J$ is even then $x_J=\alpha_\mu$ and $\infty>x_{J-1} = \beta_\mu$, and if $J$ is odd then $x_J=\beta_\mu$, $-\infty<x_{J-1} = \alpha_\mu$. 

If $J=\infty$, define $x_\infty = \lim_{k\to\infty} x_{2k}$ and $x^\infty = \lim_{k\to\infty} x_{2k+1}$. It follows that $\alpha_\mu \leq x_\infty < x_0 < x^\infty \leq \beta_\mu$, so that $(x_0,x_1,...)\in \sY_\infty$. It remains to show that if $J=\infty$ then $x_\infty =\alpha_\mu$ and $x^\infty = \beta_\mu$, for if so then $(x_0, x_1, \ldots) \in \sY^*_\infty \subset \sY_\infty$.
The main idea is that if the assertion does not hold, then $b=\lim_{k\to\infty}M(x_{2k})$ and $c=\lim_{k\to\infty}N(x_{2k+1})$ violate the assumption that $(\mu,\nu) \in \sK_{R}$.

Define $m_\infty = \lim_{k \uparrow \infty}M(x_{2k})$, $n_\infty = \lim_{k \uparrow \infty}N(x_{2k})$, $m^\infty = \lim_{k \uparrow \infty}M(x_{2k+1})$, $n^\infty = \lim_{k \uparrow \infty}N(x_{2k+1})$, and note that $M(x_{2k})\geq M(x_{2k+1}) {\geq M(x_{2k+2})}$ and $N(x_{2k})\leq N(x_{2k+1}) {\leq N(x_{2k+2})} $ so that $m_\infty=m^\infty$ and $n_\infty=n^\infty$. Furthermore, 
$M(x_{2k+1})< x_{2k}  \leq \overrightarrow{G_\mu}(F_\mu(x_{2k})) = \overleftarrow{G_\mu}(F_\mu(x_{2k})-) \leq M(x_{2k}) \leq x_{2k} \leq \overleftarrow{G_\mu}(F_\mu(x_{2k}))$ so that $\overleftarrow{G_\mu}(F_\mu(x_{2k})-) \leq m^\infty \leq \overleftarrow{G_\mu}(F_\mu(x_{2k}))$. It follows that
$m_\infty=m^\infty= x_\infty$. Similarly we obtain that $n_\infty=n^\infty=x^\infty$.

Then, we have both $\lim_{k\to\infty} \sE_{x_{2k},x_{2k-1}}(M(x_{2k})) = \lim_{k\to\infty} D_{\mu,\nu}(M(x_{2k})) = D_{\mu,\nu}(m_\infty)$ and $\lim_{k\to\infty} \sE_{x_{2k},x_{2k-1}}(N(x_{2k})) = \sE_{x_\infty,x^\infty}(n_\infty)=D_{\mu,\nu}(n_\infty)$.
Moreover, $\lim_{k\to\infty} D_{\mu,\nu}'(M(x_{2k}))=\lim_{k\to\infty} \sE'_{x_{2k},x_{2k-1}}(M(x_{2k})) = \lim_{k\to\infty} \{ P_\nu'(M(x_{2k})) - F_\mu(M(x_{2k})) \} 
=D_{\mu,\nu}'(m_\infty)$. Similarly, $\lim_{k\to\infty}(\sE_{x_{2k},x_{2k-1}})'(N(x_{2k})) = \lim_{k\to\infty}\{P_\nu'(N(x_{2k}))- F_\mu(N(x_{2k}))\} 
= D'_{\mu,\nu}(n_\infty)$. On the other hand, by construction of $N,M$, and using Lemma \ref{lem:properties2}, Corollary \ref{cor:E_vu} and the properties of the convex hull, we have that
$$
\sE_{x_{2k},x_{2k-1}}'(M(x_{2k})) = \sE_{x_{2k},x_{2k-1}}'(N(x_{2k})) = \frac{\sE_{x_{2k},x_{2k-1}}(N(x_{2k})) - \sE_{x_{2k},x_{2k-1}}(M(x_{2k}))}{N(x_{2k})-M(x_{2k})}.
$$
Since $\sE_{x_{2k},x_{2k-1}}(k)=D_{\mu,\nu}(k)$ for $k\in\{M(x_{2k}),N(x_{2k})\}$, we find
\[ D_{\mu,\nu}'(m_\infty) = \frac{D_{\mu,\nu}(n_\infty) - D_{\mu,\nu}(m_\infty)}{n_\infty - m_\infty} = D_{\mu,\nu}'(n_\infty), \]
so that $L_{m_\infty}$ given by $L_{m_\infty}(x) = D_{\mu,\nu}(x) + (x- m_\infty) D'_{\mu,\nu}(m_\infty)$ {agrees with} $D_{\mu,\nu}$ at $n_\infty$.
Moreover, $\sE_{x_{2k},x_{2k-1}} > L_{\sE_{x_{2k},x_{2k-1}}}^{M(x_{2k}), D_{\mu,\nu}'(M(x_{2k}))}$ on $(M(x_{2k}+),N(x_{2k}+))$. Since, 
$\sE_{x_{2k},x_{2k-1}} = D_{\mu,\nu}$ on $[M(x_{2k}+),N(x_{2k}+)]$, it follows that
$D_{\mu,\nu} > L_{\sE_{x_{2k},x_{2k-1}}}^{M(x_{2k}), D_{\mu,\nu}'(M(x_{2k}))}$ on $(M(x_{2k}+), N(x_{2k}+))$.
Letting $k \uparrow \infty$ we conclude $D_{\mu,\nu} > L_{\sE_{x_{\infty},x^\infty}}^{m_\infty, D_{\mu,\nu}'(m_\infty)} \equiv L_{D_{\mu,\nu}}^{m_\infty, D_{\mu,\nu}'(m_\infty)}$ on $(m_\infty,n_\infty)$.

It follows that with $b=m_\infty$ (and supposing $m_\infty>\alpha_\nu$) the corresponding $c$ in Definition~\ref{def:KR} is $n_\infty$. But then $D'_{\mu,\nu}(m_\infty) = D'_{\mu,\nu}(n_\infty)$ contradicting the fact that $(\mu,\nu) \in \sK_R$. Thus it follows that $m_\infty = \alpha_\nu$. Then also $x_\infty = \alpha_\mu$, $n_\infty = \beta_\nu$ and $x^\infty= \beta_\mu$, as claimed.
Moreover, we must have $\alpha_\mu = \alpha_\nu$ and $\beta_\mu = \beta_\nu$.



Observe that if $J$ is (finite and) odd, then $(x_0, \ldots x_{J-2}, x_{J-1}=\alpha_\mu) \in \sY^{ \{ J-1 \} }_{J-1}$, $P_E(J-1,(x_0, \ldots x_{J-2}, x_{J-1}=\alpha_\mu))$ holds, and thus (by the fifth property) $x_{J-1}=\alpha_\mu>-\infty$. It follows that if $\alpha_\mu= -\infty$ then $J$ cannot be odd (similarly, if $\beta_\mu=\infty$ then $J$ cannot be (finite and) even). In particular, if $-\infty=\alpha_\mu < \beta_\mu=\infty$ then $J$ must be infinite. 

If $J<\infty$, then $(x_0,...,x_J)\in\sY^{\{ J-1,J \} }_J$, while if  $J=\infty$ we have that $(x_0, x_1, \ldots) \in \sY^*_\infty$. In either case, $\phi_J$ (see Definition \ref{def:psi}) is finite-valued on $\R\cap[\alpha_\mu,\beta_\mu]$ and maps $\R \cap [\alpha_\mu,\beta_\mu]$ to $[0,\beta_\mu-\alpha_\mu]\cap\R$. In particular, $\phi_J$ is a bijection and thus $\psi_J=\phi^{-1}_J$ is well-defined, finite-valued on $[0,\beta_\mu-\alpha_\mu]\cap\R$ and maps $[0,\beta_\mu-\alpha_\mu]\cap\R$ to $\R \cap [\alpha_\mu,\beta_\mu]$. 

For $J \leq \infty$, recall the definition of $\psi_J$ and $H_J$ (see Definition~\ref{def:psi}). Note that $H_J$ is continuous and non-decreasing on $[0,\beta_\mu-\alpha_\mu)$, but may fail to be strictly increasing. Define $H^{-1}_J(0)=0$ and for $u\in(0,\mu(\R))$ set $H_J^{-1}(u)=\sup\{k\in[0,\beta_\mu-\alpha_\mu):H_J(k)<u\}$, so that $H_J^{-1}$ is the (left-continuous version of) generalized inverse of $H_J$. $H^{-1}_J$ is strictly increasing (but may fail to be continuous) and $H_J\circ H^{-1}_J(u)=u$ for all $u\in[0,\mu(\R))$.

Set $\bU = \cup_{0 \leq j < J+1}\{u_j\}$, where $u_0=0$, $u_j=F_\mu(x_j)-F_\mu(x_{j-1})$ if $ j$ is odd, and $u_j=F_\mu(x_{j-1})-F_\mu(x_j)$ if $j$ is even. Note that by writing $j<J+1$ we include $j=J$ if $J$ is finite, but not if $J = \infty$.

Note that $\psi_J(z)=x_0$ if and only if $z=0$. 
On $\bigcup_{\textrm{even}~2\leq j < J+1}(x_{j-1}-x_{j-2},x_{j-1}-x_{j}]\cap\R$ 
$\psi_J<x_0$ we have $\psi_J < x_0$ and $\psi_J$ is strictly decreasing.  
On $\bigcup_{\textrm{odd}~1\leq j < J+1}(x_{j-2}-x_{j-1},x_{j}-x_{j-1}]\cap\R$ (with $x_{-1}=x_0$) we have $\psi_J > x_0$ and $\psi_J$ is strictly increasing.
On the other hand, for each even $2\leq j < J+1$, $H_J^{-1}(v)\in (x_{j-1}-x_{j-2},x_{j-1}-x_{j})$ if and only if $v\in (u_{j-1},u_{j})$. 
Similarly, for each odd $1\leq j <J+1$, $H_J^{-1}(v)\in (x_{j-2}-x_{j-1},x_{j}-x_{j-1})$ if and only if $v\in (u_{j-1},u_j)$.

Define $\tilde G:[0,\mu(\R))\to [\alpha_\mu,\beta_\mu] \cap \R$ by $\tilde{G} = \psi_J \circ H_{J}^{-1}$. From the above observations and the definition we have that $\tilde{G}(0) = x_0$, on $\bigcup_{2 \leq j < J+1, \mbox{\footnotesize{$j$ even}}} (u_{j-1},u_j)$ we have $\tilde{G} < x_0$ and $\tilde{G}$ is strictly decreasing, and on $\bigcup_{1 \leq j < J+1, \mbox{\footnotesize{$j$ odd}}} (u_{j-1},u_j)$ we have $\tilde{G} > x_0$ and $\tilde{G}$ is strictly increasing. Also, if $\{z \in \tilde{G}(u): u \in [0, \mu(\R)) \}$, then there exists a unique $u=u_z$ such that $\tilde{G}(u) = \psi(H_J^{-1}(u) = z$ and $u = H_J(\phi(z))$.

We are now in a position to show that $(M,N)$ define a martingale coupling of $\mu$ and $\nu$.

First we show that the law of $\tilde G(U)$ is $\mu$, where $U$ is a random variable with unit density on $[0,1]$. For each $x\in\R$ we need to show that $\int_0^{\mu(\R)} I_{ \{ \tilde{G}(v) \leq x \} } dv
   =F_\mu(x) $. We only treat the case when $x_0<x\leq \beta_\mu$, so that $x_{k-2} < x \leq x_{k}$ for some odd $1\leq k\leq J$; similar arguments can be used to treat the case $x \leq x_0$. Then
\begin{eqnarray*}
\int_0^{\mu(\R)} I_{ \{ \tilde{G}(v) \leq x \} } dv
   & = & \int_0^{\mu(\R)}  (I_{ \{ \tilde{G}(v) \leq x_0 \}  }  +I_{ \{ x_0 <\tilde{G}(v) \leq x_{k-2} \}}+ I_{ \{ x_{k-2} < \tilde{G}(v) \leq x \}})   dv  \\
  & = & \sum_{\textrm{even}~j\geq 2}^J (u_j-u_{j-1}) + \sum^{k-2}_{\textrm{odd}~j\geq1} (u_{j}-u_{j-1}) + \int_{u_{k-1}}^{u_k}I_{\{x_{k-2}<\tilde G(v)\leq x\}}dv \\
  & = & \sum_{\textrm{even}~j\geq 2}^J \mu( (x_{j}, x_{j-2})) + \sum_{\textrm{odd}~j\geq1}^{k-2} \mu((x_{j-2},x_{j})) + \int_{u_{k-1}}^{u_k}I_{\{x_{k-2}<\tilde G(v)\leq x\}}dv \\
  & = & \mu(( \alpha_\mu,x_{k-2}))+\int_{u_{k-1}}^{u_k}I_{\{x_{k-2}<\tilde G(v)\leq x\}}dv \\
  & = & F_\mu(x_{k-1})+\int_{u_{k-1}}^{u_k}I_{\{x_{k-2}<\tilde G(v)\leq x\}}dv ,
\end{eqnarray*}
and thus we are left to argue that $\int_{u_{k-1}}^{u_k}I_{\{x_{k-2}<\tilde G(v)\leq x\}}dv =F_\mu(x)-F_\mu(x_{k-2})$. For this, observe that
\begin{align*}
    \int_{u_{k-1}}^{u_k}I_{\{x_{k-2}<\tilde G(v)\leq x\}}dv&=\int_{u_{k-1}}^{u_k}I_{\{x_{k-2}-x_{k-1}<\psi_J(H^{-1}_J(v))-x_{k-1}\leq x-x_{k-1}\}}dv\\
    &=\int_{u_{k-1}}^{u_k}I_{\{x_{k-2}-x_{k-1}<H^{-1}_J(v)\leq x-x_{k-1}\}}dv\\
    &=\int_{u_{k-1}}^{u_k}I_{\{u_{k-1}<v\leq H_J(x-x_{k-1})\}}dv\\
    &=H_J(x-x_{k-1})-[F_\mu(x_{k-2})-F_\mu(x_{k-1})]\\
    &=[F_\mu(x)-F_\mu(x_{k-1})]-[F_\mu(x_{k-2})-F_\mu(x_{k-1})]\\
    &=F_\mu(x)-F_\mu(x_{k-2}),
\end{align*}
and we are done.

	Let $\tilde\mu$ be a measure on $[0,\mu(\R)]\times\R$ defined by $\tilde\mu(du,dx)=du\delta_{\tilde{G}(u)}(dx)$. From the above calculations it follows that the first and second marginals of $\tilde\mu$ are $\lambda$ and $\mu$, respectively.
	
	Now, by Beiglb\"{o}ck and Juillet \cite[Theorem 2.1]{BJ:21}, there exists the unique lifted martingale coupling $\tilde\pi$ (a measure on $[0,\mu(\R)]\times\R\times\R$) of $\mu$ and $\nu$, that, for each $u\in[0,\mu(\R)]$, $\tilde{\pi}$ embeds $\tilde\mu_{[0,u]}:=\int^u_0\delta_{\tilde G(v)}dv$ into $S^\nu(\tilde\mu_{[0,u]})$. More precisely,
	$$
	\tilde\pi(du,dx,dy)=du\delta_{\tilde{G}(u)}(dx)\chi_{\tilde{R}(u),\tilde{G}(u),\tilde{S}(u)},
	$$
	where
	$$
	\chi_{\tilde{R}(u),\tilde{G}(u),\tilde{S}(u)}=\begin{cases}
	\delta_{\tilde{G}(u)},&\textrm{if }\tilde{G}(u)\in\textrm{supp}(\nu-S^\nu(\tilde\mu_{[0,u]})),\\
	\frac{\tilde{S}(u)-\tilde{G}(u)}{\tilde{S}(u)-\tilde{R}(u)}
\delta_{\tilde{R}(u)}+\frac{\tilde{G}(u)-\tilde{R}(u)}{\tilde{S}(u)-\tilde{R}(u)}\delta_{\tilde{S}(u)},&\textrm{otherwise,}
	\end{cases}
	$$
	and $\tilde{R}(u)=\sup\{k\in\textrm{supp}(\nu-S^\nu(\tilde\mu_{[0,u]}))\cap(-\infty,\tilde{G}(u)]\}$, $\tilde{S}(u)=\inf\{k\in\textrm{supp}(\nu-S^\nu(\tilde\mu_{[0,u]}))\cap[\tilde{G}(u),\infty)\}$.	
	
	Fix $\tilde{u} \in[0,\mu(\R)] { \setminus \bU}$. 
 Then $\tilde{u} \in (u_{j-1}, u_{j})$ for some $j=j(\tilde{u}) \in \{1, \dots , J \}$.
Suppose that $j$ is odd (the case when $j$ is {even} follows by symmetry). 
	Then $\tilde{G}$ is (strictly) increasing at $\tilde{u}$
and $\tilde\mu_{[0,\tilde u]}=\mu_{\tilde{G}(\tilde{u})} -\mu_{\tilde{G}({{u_{j-1}}})}$. It follows from Lemma~\ref{lem:properties1}(i) that
	$$
	S^\nu(\tilde\mu_{[0,\tilde u]})=S^\nu(\mu_{\tilde{G}(\tilde u)} 
    -\mu_{\tilde{G}({{u_{j-1}}})})=\nu\lvert_{(M(\tilde{G}(\tilde u)),N(\tilde{G}(\tilde u)))},
	$$
	and therefore $\tilde{R}(\tilde u)=M(\tilde{G}(\tilde u))$ and $\tilde{S}(\tilde u)=N(\tilde{G}(\tilde u))$.
	
	We conclude that, for each {$j\geq0$}, and $u\in[0,\mu(\R)] \setminus \bU$, 
 we have that
 \begin{eqnarray*}
	\chi_{\tilde{R}(u),\tilde{G}(u),\tilde{S}(u)} & = &
	\frac{{N}(\tilde{G}(u))- \tilde{G}(u)}{N(\tilde{G}(u))-M(\tilde{G}(u))}\delta_{{M}(\tilde{G}(u))}
+\frac{\tilde{G}(u)-{M}(\tilde{G}(u))}{{N}(\tilde{G}(u))-{M}(\tilde{G}(u))}\delta_{{N}(\tilde{G}(u))} \\
& = & \pi^{M,N}_{\tilde{G}(u)},
\end{eqnarray*}
where $\pi^{M,N}_{\tilde{G}(u)}$ is defined as in \eqref{eq:pifh}.

Then, since $\bU$ is countable, and $\chi$ defines a lifted martingale coupling of $\mu$ and $\nu$,
\[ \nu(dy) = \int_x \int_u \tilde{\pi}(du,dx,dy) = \int_u \int_x du \delta_{\tilde{G}(u)}(dx)\pi^{M,N}_{\tilde{G}(u)}(dy) = \int_{u} du \pi^{M,N}_{\tilde{G}(u)}(dy). \]
Write
$$
\int_{u} du \pi^{M,N}_{\tilde{G}(u)}(dy)=\sum^J_{\textrm{odd}~ j=1}\int^{u_j}_{u_{j-1}}du \pi^{M,N}_{\tilde{G}(u)}(dy)+\sum^J_{\textrm{even}~ j=2}\int^{u_j}_{u_{j-1}}du \pi^{M,N}_{\tilde{G}(u)}(dy).
$$ 
We now argue that, for each odd  $j$ (with $j<J+1$), $\int^{u_j}_{u_{j-1}}du \pi^{M,N}_{\tilde{G}(u)}(dy)=\int^{x_j}_{x_{j-2}}\mu(dx) \pi^{M,N}_{x}(dy)$, while for each even $j$ (with $j<J+1$), $\int^{u_j}_{u_{j-1}}du \pi^{M,N}_{\tilde{G}(u)}(dy)=\int^{x_{j-2}}_{x_{j}}\mu(dx) \pi^{M,N}_{x}(dy)$. We will only consider the case when $j$ is odd---the case when $j$ is even can be treated using similar arguments.

Fix an odd $j$ with $j<J+1$. Note that $u\in(u_{j-1},u_{j})$ if and only if $\tilde G(u)\in (x_{j-2},x_j)$. Now use a change of variables $\tilde G(u)=\psi_J\circ H_J^{-1}(u)=x$, where $x\in\{\tilde G(v): v\in(u_{j-1},u_{j})\}$. (Observe that, since the law of $\tilde G(U)$ is $\mu$, $\mu(\{\tilde G(v): v\in(u_{j-1},u_{j})\})=\mu((x_j,x_{j-2}))$). Then, by recalling that $H_J\circ H^{-1}_J(v)=v$ for all $v\in[0,\mu(\R))$, we have that
$$
u=H_J(\phi_J(x))=H_J(x-x_{j-1})=F_\mu(x)-F_\mu(x_{j-1}),
$$
and therefore
$$
\int^{u_j}_{u_{j-1}}du \pi^{M,N}_{\tilde{G}(u)}(dy)=\int_{x\in\{\tilde G(v): v\in(u_j,u_{j-1})\}}d\mu(x)\pi_x^{M,N}(dy)=\int^{x_j}_{x_{j-2}}d\mu(x)\pi_x^{M,N}(dy).
$$

Combining both cases (when $j$ is even and when it is odd) we obtain that
\begin{align*}
\nu(dy)&=\sum^J_{\textrm{odd}~ j=1}\int^{u_j}_{u_{j-1}}du \pi^{M,N}_{\tilde{G}(u)}(dy)+\sum^J_{\textrm{even}~ j=2}\int^{u_j}_{u_{j-1}}du \pi^{M,N}_{\tilde{G}(u)}(dy)\\
&=\sum^J_{\textrm{odd}~ j=1}\int^{x_j}_{x_{j-2}}\mu(dx) \pi^{M,N}_{x}(dy)+\sum^J_{\textrm{even}~ j=2}\int^{x_{j-2}}_{x_{j}}\mu(dx) \pi^{M,N}_{x}(dy)\\
&=\int_\R\mu(dx)\pi^{M,N}_x(dy).
\end{align*}
It follows that $(M,N)$ defines a martingale coupling of $\mu$ and $\nu$.	
\end{proof}

\begin{thm}
\label{thm:K*exist}
Suppose $(\mu,\nu) \in \sK_*$. Then there exists a strongly injective martingale coupling of $\mu$ and $\nu$ on its irreducible component.
Moreover, in the definition of the strongly injective coupling we may assume that $\Gamma_\mu \subseteq \supp_I(\mu)$. 
\end{thm}

\begin{proof}
The idea is to show that although $\{ {\pi^0_x} \}_{x \in {\tiny{\supp}}(\mu)} $ given by $\pi^0_x = \pi^{M,N}_x$ is typically not injective, it can be modified to give a strongly injective martingale coupling, without invalidating the fact that it defines a martingale coupling of $\mu$ and $\nu$. 

We work inductively. The first step is to consider $x \in (x_0,x_1)$ and $M,N$ defined on this interval.

Let $\hat{\Gamma}_1 = \{ x \in (x_0,x_1) : F_\mu(x) > F_\mu(z) \quad \forall z \in (x_0,x) \}.$ (Equivalently, $\hat{\Gamma}_1$ is the interval $(x_0,x_1)$ with any intervals where $\mu$ places no mass removed, where we remove intervals in half-open form $(\underline{x},\overline{x}]$.) {Note that $\hat{\Gamma}_1 \subseteq \supp_I(\mu)$.} Since $(M,N)$ define a martingale coupling of $\mu^1:=\mu|_{(x_0,x_1)}$ and $\nu^1:=S^\nu(\mu|_{(x_0,x_1)})= \nu|_{(M(x_1-),N(x_1-))}$ we have that, for $y \in (M(x_1-),N(x_1-))$, 
$\nu(dy)  = \int_{x \in (x_0,x_1)} \mu(dx) \pi_x^{M,N}(dy) = \int_{x \in \hat{\Gamma}_1} \mu(dx) \pi_x^{M,N}(dy)$.

Next, for $x \in \hat{\Gamma}_1$ we replace $\pi^{M,N}_x$ with $\hat{\pi}_x$ where $\hat{\pi}_x = \frac{1}{4} \pi_x^{M(x-),N(x-)} 
+ \frac{1}{4} \pi_x^{M(x-),N(x+)} + \frac{1}{4} \pi_x^{M(x+),N(x-)} + \frac{1}{4} \pi_x^{M(x+),N(x+)}
$. Note that $\hat{\pi}_x$ has mean $x$ and support $\{ M(x-), M(x+), N(x-), N(x+) \}$. Of course, if $x \in \hat{\Gamma}_1$ is such that $M(x+)=M(x-)$ and $N(x+)=N(x-)$ then $\hat{\pi}_x = \pi^{M,N}_x$; since there are only countably many $x \in \hat{\Gamma}_1$ for which this is not the case (here we use the fact that $M$ and $N$ are monotonic on $(x_0,x_1)$), we have that $(\hat{\Gamma}_1,(\hat{\pi}_x)_{x \in \hat{\Gamma}_1})$ still defines a martingale coupling of $\mu^1$ and $\nu^1$. Moreover, $\hat{\pi}_x$ defines a strongly injective martingale coupling of $\mu^1=\mu|_{(x_0,x_1)}$ and $\nu|_{(M(x_1-),M(x_0+)) \cup (N(x_0+), N(x_1-))}$ in the sense that for each $x \in \hat{\Gamma}_1$ we have $\supp(\hat{\pi}_x) \subseteq (M(x_1-),M(x_0+)) \cup (N(x_0+), N(x_1-))$ and each $y \in (M(x_1-),M(x_0+)) \cup (N(x_0+), N(x_1-))\cap \supp(\nu)$, is in the support of exactly one $\hat{\pi}_x$. This last result follows from Lemma~\ref{lem:properties1} and the strict monotonicity of $M$ and $N$ on $\hat{\Gamma}_1$.

Now we deal with points in the set $\{ M(x_0+), N(x_0+) \}$. (Note it is possible that $N(x_0+)=x_0=M(x_0+)$, so this set may be a singleton.) Choose $\tilde{x}_0 \in \hat{\Gamma}_1$, let $\tilde{\Gamma}_0 = \{ \tilde{x}_0\}$ and let $\tilde{\pi}_{\tilde{x}_0}$ be any measure such that $\tilde{\pi}_{\tilde{x}_0}$ has mean ${\tilde{x}_0}$ and support $\supp(\hat{\pi}_{\tilde{x}_0}) \cup \{ M(x_0+), N(x_0+) \} $. Note that changing the disintegration $\hat{\pi}$ at a single point (which is not an atom of $\mu$) will not affect the fact that it defines a martingale coupling. Now we have that each $y \in (M(x_1-),M(x_0+)] \cup [N(x_0+), N(x_1-)) = \supp_I(\nu^1)$ is in the support of exactly one $x \in \hat{\Gamma}_1$.

If $x_0 = \alpha_\mu$ then by Lemma~\ref{lem:w=muR} we must have that $x_1= \beta_\mu$,  $N(x_1-) = \beta_\nu$ and $M(x_1-)= \alpha_\nu$ and that the inductive construction of Theorem~\ref{thm:K*} terminates. In that case we define $\tilde{\Gamma}= \tilde{\Gamma}_0$, $\Gamma_\mu = \hat{\Gamma}_1$, and for every $x$ in $\Gamma_\mu \setminus \tilde{\Gamma}$ we set $\pi_x = \hat{\pi}_x$ whereas for $x \in \tilde{\Gamma}$ we set $\pi_x = \tilde{\pi}_x$. Then $\supp({\pi}_x) \subseteq \supp_I(\nu)$ and for each $y \in  \supp_I(\nu)$ there exists a unique $x \in \Gamma_\mu$ such that $y \in \supp({\pi}_x)$. Thus, we have constructed a coupling of $\mu$ and $\nu$ which is strongly injective on its irreducible component. Moreover $\Gamma_\mu \subseteq \supp_I(\mu)$. 

If $x_0 > \alpha_\mu$ then we must have $N(x_1-) < \beta_\nu$ and $M(x_1-)>\alpha_\nu$. Choose $\tilde{x}_1 \in \hat{\Gamma}_1 \setminus \tilde{\Gamma}_0$, set $\tilde{\Gamma}_1 = \tilde{\Gamma}_0 \cup \{ \tilde{x}_1 \}$ and let 
$\tilde{\pi}_{\tilde{x}_1}$ be any measure such that 
$\tilde{\pi}_{\tilde{x}_1}$ has mean ${\tilde{x}_1}$ and support
$\supp(\hat{\pi}_{\tilde{x}_1}) \cup \{ M(x_1-), N(x_1-) \}$.
For $x \in \hat\Gamma_1 \setminus \tilde{\Gamma}_1$ set $\pi_x = \hat{\pi}_x$, and for $x \in \tilde{\Gamma}_1$ let $\pi_x = \tilde{\pi}_x$.
Then for every $x$ in $\hat\Gamma_1$, $\supp({\pi}_x) \subseteq \supp_I(\nu) \cap {[M(x_1-),N(x_1-)]}$ and for each $y \in (\supp(\nu) \cap {[M(x_1-),N(x_1-)]})$ there exists a unique $x \in \hat\Gamma_1$ such that $y \in \supp(\pi_x)$. Note that $\hat\Gamma_1 \subseteq \supp_I(\mu)$.

The next step (in the case $x_0>\alpha_\mu$) is to extend the construction to $(x_2,x_1)$. The idea is that we use $(\pi_x^{M,N})_{x \in (x_2,x_1)}$ to define a candidate coupling, but that we modify the  construction to deal with cases where $M$ or $N$ jumps and to cover any other points in the support of $\nu|_{(M(x_2+),N(x_2+))}$ which are otherwise missed.

Let $\hat{\Gamma}_2 = \{ x \in (x_2,x_0) : F_\mu(x) < F_\mu(z), \forall z \in (x,x_0) \} { \subseteq (x_2,x_0) \cap \supp_I(\mu).}$ 
Since $(M,N)$ define a martingale coupling of $\mu^{0,2}:=\mu|_{(x_2,x_0)}$ and $\nu^{0,2}:=S^\nu(\mu|_{(x_2,x_0)})= \nu|_{(M(x_2+),N(x_2+)) \setminus (M(x_1-),N(x_1-))}$ we have that for $y \in (M(x_2+),N(x_2+))\setminus (M(x_1-),N(x_1-))$, $\nu(dy) = \int_{x \in \hat{\Gamma}_2} \mu(dx) \pi_x^{M,N}(dy)$.

First, for $x \in \hat{\Gamma}_2$ we replace $\pi^{M,N}_x$ with $\hat{\pi}_x$ where $\hat{\pi}_x = \frac{1}{4} \pi_x^{M(x-),N(x-)} 
+ \frac{1}{4} \pi_x^{M(x-),N(x+)} + \frac{1}{4} \pi_x^{M(x+),N(x-)} + \frac{1}{4} \pi_x^{M(x+),N(x+)}
$. Again, if $x \in \hat{\Gamma}_2$ is such that $M(x+)=M(x-)$ and $N(x+)=N(x-)$ then $\hat{\pi}_x = \pi^{M,N}_x$; since there are only countably many $x \in \hat{\Gamma}_2$ for which this is not the case, we have that $(\hat{\pi}_x)_{x \in \hat{\Gamma}_2 \cup \hat{\Gamma}_1}$ still defines a martingale coupling of $\mu^2= \mu^1 + \mu^{0,2}$ and $\nu^2= \nu^1 + \nu^{0,2}$.

Second, let $\hat{\Sigma}_0 =  A_M^-(x_0) \cup A_N^-(x_0)$ where $A_M^-(x_0) = \{ M(x_0-) \}$ if $M(x_0-) < M(x_1-)$ and $A_M^-(x_0)$ is empty otherwise, and $A_N^-(x_0) = \{ N(x_0-) \}$ if $N(x_0-) > N(x_1-)$ and $A_N^-(x_0)$ is empty otherwise. 
If $\hat{\Sigma}_0$ is nonempty then choose $\check{x}_0 \in \hat{\Gamma}_2$ and set $\check{\Gamma}_0 = \{ \check{x}_0 \}$ (else $\check{\Gamma}_0$ is the emptyset). Let $\check{\pi}_{\check{x}_0}$ be any measure such that  $\check{\pi}_{\check{x}_0}$ has mean ${\check{x}_0}$ and support $\supp(\hat{\pi}_{\check{x}_0}) \cup \hat{\Sigma}_0$.

If $x_1 = \beta_\mu$ then by Lemma~\ref{lem:w=muR} we must have $x_2= \alpha_\mu$, $N(x_2+) = \beta_\nu$ and $M(x_2+)= \alpha_\nu$. In that case $\mu^2 = \mu$, $\nu^2 = \nu$ and we define $\Gamma_\mu = \hat{\Gamma}_2 \cup {\hat\Gamma_1} \subseteq \supp_I(\mu)$. For $x \in {\hat\Gamma_1}$ we let $\pi_x$ be defined as before. For $x \in \hat{\Gamma}_2 \setminus \check{\Gamma}_0$ we let $\pi_x = \hat{\pi}_x$. Finally, if $\check{\Gamma}_0$ is nonempty, we let $\pi_{\check{x}_0} = \check{\pi}_{\check{x}_0}$. 
Then for every $x$ in $\Gamma_\mu$, 
$\supp({\pi}_x) \subseteq \supp_I(\nu)$ and for each $y \in  \supp_I(\nu)$ there exists a unique $x \in \Gamma_\mu$ such that $y \in \supp({\pi}_x)$. Thus, we have constructed a coupling of $\mu$ and $\nu$ which is strongly injective on its irreducible component.  Moreover, $\Gamma_\mu \subseteq \supp_I(\mu)$. 

If $x_1 < \beta_\mu$ then we must have $N(x_2+) < \beta_\nu$ and $M(x_2+)>\alpha_\nu$. Choose $\tilde{x}_2 \in \hat{\Gamma}_2 \setminus \{ \check{x}_0 \}$ and let $\tilde{\pi}_{\tilde{x}_2}$ be any measure such that $\tilde{\pi}_{\tilde{x}_2}$ has mean ${\tilde{x}_2}$ and support $\supp(\hat{\pi}_{\tilde{x}_2}) \cup \{ M(x_2+), N(x_2+) \}$. For $x \in \hat{\Gamma}_2{\setminus\{\check{x}_0, \tilde{x}_2 \}} $ set $\pi_x = \hat{\pi}_x$, and for $x \in \{ \check{x}_0, \tilde{x}_2 \}$ let $\pi_x$ be given by $\pi_x = \check{\pi}_{\check{x}_0}$ or $\pi_x = \tilde{\pi}_{\tilde{x}_2}$ as appropriate. Then for every $x$ in $\Gamma_2=\hat\Gamma_1\cup\hat\Gamma_2$, $\supp({\pi}_x) \subseteq (\supp(\nu) \cap {[M(x_2+),N(x_2+)]})$ and for each $y \in (\supp(\nu) \cap {[M(x_2+),N(x_2+)]})$ there exists a unique $x \in \Gamma_2$ such that $y \in \supp(\pi_x)$. Note that $\Gamma_2 \subseteq \supp_I(\mu) \cap (x_2,x_1)$. 

Now we proceed by induction, working alternately left to right from $x_{2k-1}$ to $x_{2k+1}$, and then right to left from $x_{2k}$ to $x_{2k+2}$. We present the argument in the left-to-right direction, the reverse case being very similar.

Suppose we have $\Gamma_{2k} { \subseteq (x_{2k},x_{2k-1}) \cap \supp_I(\mu)}$ and $({\pi}_x)_{x \in \Gamma_{2k}}$ such that $(\Gamma_{2k},({\pi}_x)_{x \in \Gamma_{2k}})$ defines a martingale coupling of $\mu^{2k} := \mu|_{(x_{2k},x_{2k-1})}$ and $\nu^{2k} := S^\nu(\mu^{2k}) = \nu|_{[M(x_{2k}+),N(x_{2k}+)]}$. Suppose, moreover, that for each $x \in \Gamma_{2k}$, we have $\supp({\pi}_x) \subseteq (\supp_I(\nu) \cap [M(x_{2k}+),N(x_{2k}+)])$ and that for each $y \in (\supp(\nu) \cap [M(x_{2k}+),N(x_{2k}+)])$ we have $y \in \supp(\pi_x)$ for exactly one $x \in \Gamma_{2k}$.

Let $\hat{\Gamma}_{2k+1} = \{ x \in (x_{2k-1},x_{2k+1}) : F_\mu(x) > F_\mu(z), \; \forall z \in (x_0,x) \} \subseteq \supp_I(\mu) \cap (x_{2k-1},x_{2k+1}).$ 
Since $(M,N)$ define a martingale coupling of $\mu^{2k-1,2k+1}:=\mu|_{(x_{2k-1},x_{2k+1})}$ and $\nu^{2k-1,2k+1}:=S^\nu(\mu|_{(x_{2k-1},x_{2k+1})})= \nu|_{\Delta_{2k+1}}$ where $\Delta_{2k+1} =(M(x_{2k+1}-),M(x_{2k}+) ) \cup (N(x_{2k}+),N(x_{2k+1}-)$ we have that for $y \in \Delta_{2k+1}$, $\nu(dy)  = \int_{x \in (x_{2k-1},x_{2k+1})} \mu(dx) \pi_x^{M,N}(dy) = \int_{x \in \hat{\Gamma}_{2k+1}} \mu(dx) \pi_x^{M,N}(dy)$.

Next, for $x \in \hat{\Gamma}_{2k+1}$ we replace $\pi^{M,N}_x$ with $\hat{\pi}_x$ where $\hat{\pi}_x = \frac{1}{4} \pi_x^{M(x-),N(x-)} 
+ \frac{1}{4} \pi_x^{M(x-),N(x+)} + \frac{1}{4} \pi_x^{M(x+),N(x-)} + \frac{1}{4} \pi_x^{M(x+),N(x+)}$. 
For all but countably many $x$ we have that $\hat{\pi}_x = \pi^{M,N}_x$ and therefore 
$(\hat{\Gamma}_{2k+1},(\hat{\pi}_x)_{x \in \hat{\Gamma}_{2k+1}})$ defines a martingale embedding of $\mu^{2k-1,2k+1}$ and $\nu^{2k-1,2k+1}$. Combining this with $({\pi}_x)_{x \in \Gamma_{2k}}$ we have that $(\Gamma_{2k} \cup \hat{\Gamma}_{2k+1}, ((\pi_x)_{x \in \Gamma_{2k}},(\hat{\pi}_x)_{x \in \hat{\Gamma}_{2k+1}}))$ defines a martingale embedding of $\mu^{2k+1} = \mu^{2k} + \mu^{2k-1,2k+1}= \mu|_{(x_{2k},x_{2k+1})}$  and $\nu^{2k+1} = \nu^{2k} + \nu^{2k-1,2k+1} = \nu|_{(M(x_{2k+1}-),N(x_{2k+1}-))}$.

Note that for each $x \in \hat{\Gamma}_{2k+1}$ we have $\supp(\hat{\pi}_x) \subseteq (\supp_I(\nu) \cap \Delta_{2k+1})$. Define $\hat{\Sigma}_{2k+1} =  A^+_M(x_{2k-1}) \cup A^+_N(x_{2k-1})$ where, in turn, $A^+_M(x_{2k-1}) = \{ M(x_{2k-1}+) \}$ if $M(x_{2k-1}+)<M(x_{2k}+)$ and is empty otherwise, and $A^+_N(x_{2k-1}) = \{ N(x_{2k-1}+) \}$ if $N(x_{2k-1}+)>N(x_{2k}+)$ and is empty otherwise. If $A^+_M(x_{2k-1}) \cup A^+_N(x_{2k-1})$ is non-empty then we choose $\tilde{x}_{2k+1} \in \hat{\Gamma}_{2k+1}$ and let $\tilde{\Gamma}_{2k+1} = \{ \tilde{x}_{2k+1} \}$ (otherwise this set is empty) and let $\tilde{\pi}_{\tilde{x}_{2k+1}}$ be any measure with mean $\tilde{x}_{2k+1}$ and support $\supp(\hat{\pi}_{\tilde{x}_{2k+1}}) \cup \hat{\Sigma}_{2k+1}$.

If $x_{2k} = \alpha_\mu$ then we must have $x_{2k+1}= \beta_\mu$,  $N(x_{2k+1}-) = \beta_\nu$ and $M(x_{2k+1}-)= \alpha_\nu$. In that case we define $\Gamma_\mu = \Gamma_{2k} \cup \hat{\Gamma}_{2k+1}$. For $x \in \Gamma_{2k}$ let $\pi_x$ be defined as in the inductive hypothesis. Further, for $x \in \hat{\Gamma}_{2k+1} \setminus \tilde{\Gamma}_{2k+1}$ let $\pi_x = \hat{\pi}_x$ and for $x \in \tilde{\Gamma}_{2k+1}$ let $\pi_x = \tilde{\pi}_x$. It follows that $\pi_x$ is defined for all $x \in \Gamma_\mu$. Then, for every $x \in \Gamma_\mu$, $\supp({\pi}_x) \subseteq \supp_I(\nu)$ and for each $y \in  \supp_I(\nu)$ there exists a unique $x \in \Gamma_\mu$ such that $y \in \supp({\pi}_x)$. Thus, we have constructed a coupling of $\mu$ and $\nu$ which is strongly injective on its irreducible component. Moreover, $\Gamma_\mu \subseteq \supp_I(\mu)$. 

If $x_{2k} > \alpha_\mu$ then we must have $N(x_{2k+1}-) < \beta_\nu$ and $M(x_{2k+1}-)>\alpha_\nu$. Choose $\tilde{x}_{2k+1} \in \hat{\Gamma}_{2k+1}$ and let $\tilde{\pi}_{\tilde{x}_{2k+1}}$ be any measure such that $\tilde{\pi}_{\tilde{x}_{2k+1}}$ has mean ${\tilde{x}_{2k+1}}$ and support $\supp(\hat{\pi}_{\tilde{x}_{2k+1}}) \cup \{ M(x_{2k+1}-), N(x_{2k+1}-) \}$. For $x \in \hat{\Gamma}_{2k+1}{\setminus\{\tilde{x}_{2k+1}\}} $ set $\pi_x = \hat{\pi}_x$, and for $x = \tilde{x}_{2k+1} $ let $\pi_x$ be given by $\pi_x = \tilde{\pi}_{\tilde{x}_{2k+1}}$. Then for every $x$ in $\Gamma_{2k+1} = \Gamma_{2k} \cup \hat{\Gamma}_{2k+1} \subseteq \supp_I(\mu) \cap (x_{2k},x_{2k+1})$, we have $\supp({\pi}_x) \subseteq (\supp(\nu) \cap {[M(x_{2k+1}-),N(x_{2k+1}-)]})$ and for each $y \in (\supp(\nu) \cap {[M(x_{2k+1}-),N(x_{2k+1}-)]})$ there exists a unique $x \in \Gamma_{2k+1}$ such that $y \in \supp(\pi_x)$. Note that $\Gamma_{2k+1} \subseteq \supp_I(\mu) \cap (x_{2k},x_{2k+1})$. 

We repeat the construction, inductively, stopping if the construction terminates (and then we have a martingale coupling of $\mu$ and $\nu$ which is strongly injective on its irreducible component). Otherwise, if the construction never terminates then we set $\Gamma_\mu = \cup \Gamma_{k} = \lim_{k} \Gamma_k$.
In the non-terminating case we have that $\Gamma_\mu \subseteq (\alpha_\mu,\beta_\mu) \cap \supp_I(\mu)$. Moreover, $({\pi}_x)_{x \in \Gamma_\mu}$ defines a martingale coupling of $\mu$ and $\nu$. For each $x \in \Gamma_\mu$, $x \in \Gamma_{2k}$ for some $k$ and then $\supp(\hat{\pi}_x) \subseteq (\supp (\nu) \cap [M(x_{2k}+),N(x_{2k}+)]) \subseteq \supp_I(\nu)$.  
Conversely, if $y \in \supp_I(\nu)$ then  $y \in (\supp (\nu) \cap [M(x_{2k}+),N(x_{2k}+)])$ for sufficiently large $k$. Then $y \in \supp(\pi_x)$ for some (unique) $x \in \Gamma_{2k}$. Since $k$ is arbitrary, $y \in \supp(\pi_x)$ for a unique $x \in \Gamma_\mu$. In particular, $(\Gamma_\mu, (\pi_x)_{x \in \Gamma_\mu})$ defines a martingale coupling of $\mu$ and $\nu$ which is strongly injective on its irreducible component. Moreover $\Gamma_\mu \subseteq \supp_I(\mu)$. 

\end{proof}

Our next goal is to extend the result of Theorem~\ref{thm:K*exist} to $\sK_R$.

\begin{lem}
\label{lem:KRdecompose}
    Suppose $(\mu,\nu) \in \sK_R$. Let $x_0$ and $\underline{x}_0$ be as defined in Definition~\ref{def:KR}. Fix $x \in (\alpha_\nu,\underline{x}_0)$. 

    There exists $\olx,\oly$ with $x_0<\olx < \oly < \beta_\nu$ such that if $\hat{\mu} = \mu|_{(x, \olx)}$ and $\hat{\nu} = \nu|_{(x, \oly)}$ then $\hat{\mu} <_{cx} \hat{\nu}$, $\mu - \hat{\mu} <_{cx} \nu - \hat{\nu}$, $(\hat{\mu},\hat{\nu}) \in \sK_R$ and  $(\mu - \hat{\mu},\nu - 
    \hat{\nu}) \in \sK_*$. Moreover $\oly$ may be chosen so that $\nu(\oly, \oly + \epsilon)>0$ for every $\epsilon>0$.
\end{lem}

\begin{proof}

It follows from $x \in (\alpha_\nu,\underline{x}_0)$ that 
    $\mu(x,x_0)<\nu(x,x_0)$.
    
   Consider that tangent $L_x=\{L_x(z)\}_{z \in \R}$ to $D_{\mu,\nu}$ at $x$ and the family $(\sE_{\alpha_\mu,z})_{z \geq x_0}$. The fact that $x<\underline{x}_0$ ensures that $L_x(x_0) < D_{\mu,\nu}(x_0)$. Note also that the family $(\sE_{\alpha_\mu,z})_{z \geq x_0}$ is decreasing in $z$. Let $\olx$ be the largest value of $z$ such that $\sE_{\alpha_\mu,z} \geq L_x$ everywhere; let $\oly$ be the largest value of $w$ such that $\sE_{\alpha_\mu,\olx}(w) = L_x(w)$.

   Define $\hat{\mu}$ and $\hat{\nu}$ as in the statement of the lemma using these values of $\olx$ and $\oly$. Let $\tilde{\mu} = \mu - \hat{\mu}$ and $\tilde{\nu} = \nu - \hat{\nu}$. 

   Then writing $D$ (respectively $\hat{D}$, $\tilde{D}$) as shorthand for $D_{\mu,\nu}$ (respectively $D_{\hat{\mu},\hat{\nu}}$, ${D}_{\tilde{\mu},\tilde{\nu}}$), since $\mu$ and $\nu$ are atom free, we have that $\hat{D}$ and $\tilde{D}$ are continuously differentiable. Indeed,
\[ \hat{D} = \begin{cases} 0, & z \in (-\infty, x] \cup [\oly,\infty); \\
   D - L_x,  & z \in (x, \olx];  \\
   \sE_{\alpha_\mu,\olx} - L_x, & z \in (\olx, \oly) .\end{cases} \]
Further,
   \[ \tilde{D} = D - \hat{D}, \]
and it follows that $\tilde{D}$ is concave on $[\olx,\oly]$. From the $C^1$ property $\tilde{D}$ lies on or below $L_x$ on $[\olx,\oly]$.

It follows from the positivity of $\hat{D}$ and $\tilde{D}$ that $\hat{\mu} <_{cx} \hat{\nu}$ and $\tilde{\mu}  <_{cx} \tilde{\nu}$. It remains to show that these pairs lie in $\sK_R$ and $\sK_*$ respectively.

Letting a $\hat{\cdot}$ denote the relevant quantity and checking Definition~\ref{def:KR}, for $(\hat{\mu},\hat{\nu})$ we find $\{ \underline{\hat{x}}_0,\hat{x}_0)\} = \{ \underline{x}_0,x_0 \}$, and that the second property of Definition~\ref{def:KR} is inherited directly from $D$. The third property is inherited directly from $D$ if the crossing point $c$ is such that $c<\olx$; if $c \in (\olx,\oly)$ then it follows from the convexity of $\sE_{\alpha_\mu,\olx}$ on this region.

Similarly, letting a $\tilde{\cdot}$
denote the relevant quantity, for $\tilde{D}$ we can define $\{ \underline{\tilde{x}}_0,\tilde{x}_0 \}$.
We find $\underline{\tilde{x}}_0 \leq x$. 
Further, since $\tilde{x}_0$ is the largest value of $z$ for which $\tilde{D}(z)=L_x(z)$ we have $\tilde{x}_0 = \olx$. (Note that if $z > \olx$ then $\hat{D}(z) + L_x(z) > D(z)$ so that $L_x(z) > D(z)-\hat{D}(z) = \tilde{D}(z)$.) Then $\tilde{D}$ inherits from $D$ all the required properties in Definition~\ref{def:KR} to be in $\sK_R$: in particular, for any $b<x$ we have that the tangent to $\tilde{D}$ at $b$ crosses $\tilde{D}$ from below; if it crosses at some point $z>\oly$ then this is because the same is true for $D$; if it crosses at at some point $z \in (\olx,\oly)$ it is because of the concavity of $\tilde{D}$ on this region.

It only remains to show that $(\tilde{\mu},\tilde{\nu}) \in \sK_*$. Take $b<x$ and consider the tangent $L_b$ to $\tilde{D}$ at $b$. Note that $\tilde{D} = D$ to the left of $x$. Let $c(b)$ denote the point where this tangent crosses $\tilde{D}$. If $c(b) < \oly$ (which will be the case for $b$ sufficiently close to $x$), then we cannot have that $L_b$ is tangent to $\tilde{D}$ from below at $c(b)$ because $\tilde{D}$ is concave there. Therefore we must have   
\[ Z^+_{\tilde{\sE}_{\tilde{x}_0,l}}(l) > l,\quad l\in\R\cap(\tilde{x}_0,\tilde{x}_0 + \epsilon)  \]
for some positive $\epsilon$. Then, with $\overrightarrow{n}_{\cdot,\cdot}$ 
defined relative to $(\tilde{\mu},\tilde{\nu})$, $\overline{{w}}_{\tilde{x}_0,\tilde{x}_0}:= \inf \{ z \in (\tilde{x}_0,\beta_{\tilde{\mu}}), \overrightarrow{n}_{\tilde{x}_0,\tilde{x}_0}  \leq 
\overrightarrow{G}_{\tilde{\mu}} (F_{\tilde{\mu}}(z)+) \} > \tilde{x}_0+\epsilon$ and $(\tilde{\mu},\tilde{\nu}) \in \sK_*$.
\end{proof}

\begin{cor}
\label{cor:monol}
For $x \in (\alpha_\mu,\underline{x}_0)$ let $\olx=\olx(x)$ and $\oly=\oly(x)$ be as defined in Lemma~\ref{lem:KRdecompose}.

Then $\olx$ and $\oly$ are decreasing in $x$ and $\lim_{x \uparrow \underline{x}_0} \olx(x) = \lim_{x \uparrow \underline{x}_0} \oly(x) = x_0$.
\end{cor}

\begin{proof}
Since $\olx(x) \leq \oly(y)$ it is sufficient to show that $x_0 \leq \olx(x)$ and $\lim_{x \uparrow \underline{x}_0} \oly(x) = x_0$. The first fact can be taken directly from Lemma~\ref{lem:KRdecompose}. The second fact follows from the fact that $(\mu,\nu) \in \sK_R$. If $\overline{y}_\infty := \lim_{x \uparrow \underline{x}_0} \oly(x)  > x_0$ then we must have $D_{\mu,\nu} \geq L_{x_0}$ on $(x_0,\oly_\infty)$ but this is a contradiction to the second property of Definition~\ref{def:KR}. 
    
\end{proof}

\begin{thm}
\label{thm:KRexist}
Suppose $(\mu,\nu) \in \sK_R$. Then there exists a strongly injective martingale coupling of $\mu$ and $\nu$ on its irreducible component.   
Moreover, in the definition of the strongly injective coupling we may assume that $\Gamma \subseteq \supp_I(\mu)$. 
\end{thm}

\begin{proof}
If $(\mu,\nu) \in \sK_*$ then we are done by Theorem \ref{thm:K*exist}, so assume $(\mu,\nu) \in \sK_R \setminus \sK_*$. In this case we have that $\overline{w}_{x_0,x_0}=x_0$ and it is not possible to initialise the construction in Theorem~\ref{thm:K*exist}. 


Let $(x_n)_{n \geq 2}$ be a strictly increasing sequence such that $\lim_n x_n = \underline{x}_0$ and such that 
$\mu((x_{n-1},x_n)) > \nu((x_{n-1},x_n))$. This is possible from the properties of $\underline{x}_0$.


For $x_n$ define $\olx_n$ and $\oly_n$ as in Lemma~\ref{lem:KRdecompose}. From the construction in Lemma~\ref{lem:KRdecompose}, and from Corollary~\ref{cor:monol} we know that $(\olx_n)_{n \geq 2}$ and $(\oly_n)_{n \geq 2}$ are decreasining sequences with limit $x_0$. For $n \geq 2$ let $\hat{\mu}_n = \mu|_{(x_n, \olx_n)}$ and $\hat{\nu}_n = \nu|_{(x_n,\oly_n)}$. Let $\tilde{\mu}_n = \hat{\mu}_{n-1} - \hat{\mu}_{n}$ and 
$\tilde{\nu}_n = \hat{\nu}_{n-1} - \hat{\nu}_{n}$ (with $\hat{\mu}_1=\mu$ and $\hat{\nu}_1=\nu$). In the first step of the calculation, and in the notation of Lemma~\ref{lem:KRdecompose}, we find that $(\tilde{\mu}_2,\tilde{\nu}_2) \in \sK_*$ and $(\hat{\mu}_2,\hat{\nu}_2) \in \sK_R$. We can then decompose $(\hat{\mu}_2,\hat{\nu}_2)$ further and by repeated applications of Lemma~\ref{lem:KRdecompose} we find $(\tilde{\mu}_n,\tilde{\nu}_n) \in \sK_*$ for each $n \geq 2$.   

By assumption, as $n \uparrow \infty$, $x_n \uparrow \underline{x}_0$. Then, by Corollary~\ref{cor:monol}, since $(\mu,\nu) \in \sK_R$, $\overline{x}_n \downarrow x_0$ so that $\hat{\mu}_n(\R) \downarrow \mu((\underline{x}_0,x_0)) = \nu((\underline{x}_0,x_0)) $. Set $\mu_\infty = \mu|_{(\underline{x}_0,x_0)}$ and
$\nu_\infty = \mu|_{(\underline{x}_0,x_0)} = \nu|_{(\underline{x}_0,x_0)}$.
Then $\sum_{k \geq 2} \tilde{\mu}_n + \mu_\infty = \mu$ and $\sum_{k \geq 2} \tilde{\nu}_k + \nu_\infty = \nu$.
For each $n$, since $(\mu_n,\nu_n) \in \sK_*$, by Theorem~\ref{thm:K*exist}
there exists a strongly injective martingale coupling of $\tilde{\mu}_n$ and $\tilde{\nu}_n$ on its irreducible component. 
It remains to show that these couplings can be combined to give a strongly injective martingale coupling of $\mu$ and $\nu$.


Let $(\tilde{\Gamma}_n,(\tilde{\pi}^n_x)_{x \in \Gamma_n})$ denote a strongly injective martingale coupling of $\tilde{\mu}_n$ and $\tilde{\nu}_n$ on its irreducible component, as constructed in Theorem~\ref{thm:K*exist}. Let $\Gamma_0 = [\underline{x_0}, x_0] \cap \supp_I(\nu)$.
We want to show that $\Gamma = \Gamma_0 \cup \left( \cup_n \tilde{\Gamma}_n \right)_{n \geq 2}$ can be taken to be a disjoint union, and then if for $x \in \tilde{\Gamma}_n$ we set $\pi_x = \tilde{\pi}_x^n$ (and 
$\pi_{x}= \delta_{x}$ for $x \in \Gamma_0$) 
then we have $(\Gamma, (\pi_x)_{x \in \Gamma})$ is a strongly injective martingale coupling of $\mu$ and $\nu$ on its irreducible component. In some cases a few final adjustments are necessary. 

It is clear that for each $x \in \Gamma$, $\supp(\pi_x) \subseteq \supp_I(\nu)$.

Fix $y \in \supp_I(\nu)$. We want to show that there exists a unique $x \in \Gamma$ such that $y \in \supp(\pi_x)$. Either $y \in [\underline{x}_0,x_0]$ or $y < \underline{x}_0=\lim_n x_n$ or $y>x_0=\oly_\infty$. 

If $y \in [\underline{x}_0,x_0]$ (and $y \in \supp_I(\nu)$) then $y \in \Gamma_0$ and $y \in \supp(\pi_y)$.



If $y < \underline{x}_0$ then $x_{n-1} < y \leq x_n$ for some $n$, say $n=m$ so that 
$x_{m-1} < y \leq x_m$. Suppose $y < x_m$. Note that $\supp_I(\tilde{\nu}_n) = \supp(\nu)  \cap ((x_{n-1},x_n] \cup [\oly_{n},\oly_{n-1}))$. Since $y \in \supp_I (\nu)$ we must have $y \in \supp_I(\tilde{\nu}_m)$. Then there exists a unique $z \in \tilde{\Gamma}_m$ such that $y \in \supp(\pi_z)$. Moreover, for every $z \in \tilde{\Gamma}_n$ with $n \neq m$ we have $y \notin \supp(\pi_z)$ since $\nu_n((x_{m-1},x_m))=0$. Further, since $y<\underline{x}_0$, $y \notin \supp(\pi_z)$ for any $z \in \Gamma_0$. 

Now suppose $y = x_m$. Since $y \in \supp_I(\nu)$ then either $y \in \supp(\tilde{\nu}_m)$ or $y \in \supp(\tilde{\nu}_{m+1})$ or both. Note that $y \notin \supp_I(\tilde{\nu}_{m+1})$. 
If $y \in \supp(\tilde{\nu}_m)$ then there exists a unique $z \in \tilde{\Gamma}_m$ such that $y \in \supp(\pi_z)$. Moreover, since $y \notin \supp_I(\tilde{\nu}_{m+1})$, for every $z \in \tilde{\Gamma}_{m+1}$ we have $y \notin \supp (\pi_z)$.
The remaining case is $y \notin \supp(\tilde{\nu}_m)$. Fix $\hat{x}_m \in \tilde{\Gamma}_m$ such that $\pi_{\hat{x}_m} \neq \delta_{\hat{x}_m}$. Let $\pi_{\hat{x}_m}$ be such that $\pi_{\hat{x}_m}$ has mean $x$ and support $\tilde{\pi}_{\hat{x}_m} \cup \{ x_m \}$. Then $y = x_m \in \supp(\pi_z)$ for $z = \hat{x}_m$.


Now we consider $y>x_0 = \oly_\infty = \lim_{n \rightarrow \infty} \oly_n$. 
Then $\oly_{n} \leq y < \oly_{n-1}$ for some $n$, say $n=m$. By a parallel argument 
we find, possibly after modification of the support of $\pi_z$ for a single point $z \in \tilde{\Gamma}_m$, that $y \in \supp(\pi_x)$ for a unique $x \in \tilde{\Gamma}_m$ (and not in the support of $\pi_w$ for any $w \in \Gamma \setminus \tilde{\Gamma}_m$.

Putting this all together, we have constructed a strongly injective martingale coupling of $\mu$ and $\nu$ on its irreducible component.   
The final statement of the theorem follows directly from the construction.



\end{proof}

\section{Reducing the problem to countably many intervals}
\label{sec:lc}

The goal of this section is to explain how to divide the general problem with $(\mu,\nu) \in \sK$ into countably many intervals, in  such a way that if we can construct an injective mapping on each interval then we can construct an injective map overall.

To this end we use the left-curtain coupling introduced by Beiglb\"ock and Juillet~\cite{BeiglbockJuillet:16}, and studied further by Henry-Labord\`{e}re and Touzi~\cite{HenryLabordereTouzi:16} and Hobson and Norgilas~\cite{HobsonNorgilas:18,HobsonNorgilas:21}, although having defined the intervals using the left-curtain coupling we use a completely different construction to define the injective coupling, namely the construction of the previous section. Beiglb\"ock and Juillet~\cite{BeiglbockJuillet:16} studied existence and uniqueness of the left-curtain coupling and showed the construction was optimal for a class of Martingale Optimal Transport problems; Henry-Labord\`{e}re and Touzi~\cite{HenryLabordereTouzi:16} extended the optimality to a wider class of problems and gave a constructive proof under certain regularity conditions on the measures and Hobson and Norgilas~\cite{HobsonNorgilas:18} extended the construction to the case where $\mu$ and $\nu$ are general measures. Most relevantly for this work Hobson and Norgilas~\cite{HobsonNorgilas:21} give a graphical representation of the construction in the general case.

Define
$$
\hat\mu_u:=\mu\lvert_{(-\infty,\overrightarrow{G}_\mu(u))}+(u-F_\mu(\overrightarrow{G}_\mu(u)))\delta_{\overrightarrow{G}_\mu(u)},\quad u\in(0,\mu(\R)),
$$
and, for each $u\in(0,\mu(\R))$, let $\hat\sE_{u}:\R\to\R_+$ be given by $\hat\sE_{u}=P_\nu-P_{\hat\mu_u}$.
\begin{rem}\label{rem:EcontMU}
If $\mu$ is atom-less, then for each $x\in\R$ we have that
$\mu_x=\hat\mu_{F_\mu(x)}$ and $\sE_{\alpha_\mu,x}=\hat\sE_{F_\mu(x)}$.
\end{rem}

For $u\in(0,\mu(\R))$, set $S(u) = Z^-_{\hat\sE_{u}}(\overrightarrow{G}_{\mu}(u))$ and $R(u) = X^-_{\hat\sE_{u}}(\overrightarrow{G}_{\mu}(u))$.
Note that the definition of the lower function $R(u)$ used by Hobson and Norgilas \cite{HobsonNorgilas:21} is slightly different from $X^-_{\hat\sE_{u}}(\overrightarrow{G}_{\mu}(u))$, and is given by $\inf \{w : w \leq \overrightarrow{G}_{\mu}(u) , D(w) = L_{\hat\sE_{u}^{c}}^{z,(\hat\sE^c_{u})'_{ -}(z)}(w) \}$. However, given that $\hat\sE_{u}=D_{\mu,\nu}$ on $(-\infty,\overrightarrow{G}_{\mu}(u+)]$, it is easy to see that two definitions coincide.

\begin{thm}[Hobson and Norgilas~\protect{\cite[Theorem 3.8]{HobsonNorgilas:21}}]
\label{thm:disintegration}
Suppose $\mu \leq_{cx} \nu$. Define $\hat{\pi}^{LC}_{u,\overrightarrow{G}_{\mu}(u)}$ by
\begin{equation}
\label{eq:piLCdef}
\hat{\pi}^{LC}_{u,\overrightarrow{G}_{\mu}(u)}(dy) = \frac{S(u) - \overrightarrow{G}_{\mu}(u)}{S(u)-R(u)} \delta_{R(u)}(dy) + \frac{\overrightarrow{G}_\mu(u)-R(u)}{S(u)-R(u)} \delta_{{ S}(u)}(dy)
\end{equation}
on $R(u)<S(u)$ and $\hat\pi^{LC}_{u,\overrightarrow{G}_{\mu}(u)}(dy)=\delta_{\overrightarrow{G}_{\mu}(u)}(dy)$ otherwise.
Then $\hat{\pi}^{LC}$, defined by $\hat{\pi}^{LC}(du,dx,dy) = du \delta_{\overrightarrow{G}_{\mu}(u)}(dx) \hat{\pi}^{LC}_{u,\overrightarrow{G}_{\mu}(u)}(dy)$, is the lifted left-curtain martingale coupling of $\mu$ and $\nu$, i.e., the second and third marginals of $\hat{\pi}^{LC}$ are $\mu$ and $\nu$, respectively, and, for each $u\in(0,\mu(\R))$, the mean of  $\hat{\pi}^{LC}_{u,\overrightarrow{G}_{\mu}(u)}$ is $\overrightarrow{G}_{\mu}(u)$. 
\end{thm}

Define $T_d,T_u:(\alpha_\mu,\beta_\mu)\to\R$ by
\begin{equation}\label{eq:TdTu}
T_d(x)=R(F_\mu(x))\quad\textrm{and}\quad T_u(x)=S(F_\mu(x)),\quad x\in(\alpha_\mu,\beta_\mu).
\end{equation}
Note that, since $R\leq \overrightarrow{G}_{\mu}\leq S$ on $(0,\mu(\R))$, for each $x\in(\alpha_\mu,\beta_\mu)$ we have that $T_d(x)\leq \overrightarrow{G}_{\mu}(F_\mu(x))\leq T_u(x)$. {In particular, since $\overrightarrow{G}_{\mu}(F_\mu(x)) \leq x$, $T_d(x) \leq x$ on $(\alpha_\mu,\beta_\mu)$.} On the other hand, $\overrightarrow{G}_{\mu}(F_\mu(x))=x$ {$\mu$-a.e.}, and thus also $T_d(x)\leq x\leq T_u(x)$, for $\mu$-a.e. $x$. We now state some further important properties of $(T_d,T_u)$, most of which are directly inherited from $(R,S)$.

\begin{defn}[Hobson and Norgilas~\protect{\cite[Definition 3.5]{HobsonNorgilas:21}}]\label{def:lmonfns} Let $I\subseteq\R$ be an open interval.
Given a left-continuous, non-decreasing function $g:I \mapsto \R$, a pair of functions $r,s : I \mapsto \R$ is said to be left-monotone with respect to $g$ on $I$ if $r \leq g \leq s$ and $s$ is non-decreasing on $I$, and if for $u,u'\in I$ with $u < u'$ we have $r(u') \notin (r(u), s(u))$.
\end{defn}

\begin{lem}\label{lem:propertiesTdTu} Suppose Standing Assumption~\ref{sass:atomfree} holds (or, equivalently, $(\mu,\nu)\in\sK$). Let $T_d$ and $T_u$ be defined as in \eqref{eq:TdTu}. Then
\begin{enumerate}
\item $T_d$ and $T_u$ are left-monotone with respect to $x\mapsto \overrightarrow{G}_{\mu}(F_\mu(x))$ on $(\alpha_\mu,\beta_\mu)$
\item If $x\in(\alpha_\mu,\beta_\mu)$ is such that $T_u(x)>\overrightarrow{G}_{\mu}(F_\mu(x))$ then $T_d(x)<\overrightarrow{G}_{\mu}(F_\mu(x))$.
\item $T_u$ is left-continuous and $T_d$ satisfies $T_d(x) \leq \lim \inf_{u \uparrow x} T_d(v)$.
\item If $I\subseteq \{x:T_u(x)>\overrightarrow{G}_{\mu}(F_\mu(x)+)\}$ is an open interval, then $T_d$ is non-increasing on $I$.
\item $T_u$ is strictly increasing on $\{z:(\alpha_\mu,\beta_\mu):F_\mu(z)>F_\mu(y)~\textrm{for all }y<z\}$; $T_d$ is strictly decreasing on each open interval (provided it exists) $I\subseteq \{x:T_u(x)>\overrightarrow{G}_{\mu}(F_\mu(x)+)\}\cap\{z:(\alpha_\mu,\beta_\mu):F_\mu(z)>F_\mu(y)~\textrm{for all }y<z\}$.
\end{enumerate}
\end{lem}
\begin{proof}
Property 1. follows from the definitions of $T_d$ and $T_u$, and Hobson and Norgilas~{\cite[Theorem 4.9]{HobsonNorgilas:21}}, which states that $R,S:(0,\mu(\R))\to\R$, defined as above by $R(u) = X^-_{\hat\sE_{u}}(\overrightarrow{G}_{\mu}(u))$ and $S(u) = Z^-_{\hat\sE_{u}}(\overrightarrow{G}_{\mu}(u))$, are left-monotone with respect to $\overrightarrow{G}_{\mu}$ on $(0,\mu(\R))$.

    2. and 3. immediately follow from Hobson and Norgilas~{\cite[Lemma 4.1, Proposition 6.1]{HobsonNorgilas:21}}, where we use that $x\mapsto F_\mu(x)$ is continuous due to our Standing Assumption~\ref{sass:atomfree}.
    
    For 4., first note that, due to 2., for any $x\in I$ there exists $x'\in I$ such that
    \begin{align*}T_d(x)<\overrightarrow{G}_{\mu}(F_\mu(x))\leq x&\leq\overrightarrow{G}_{\mu}(F_\mu(x)+)\\&<\overrightarrow{G}_{\mu}(F_\mu(x'))\leq x'\leq \overrightarrow{G}_{\mu}(F_\mu(x')+)<T_u(x)\leq T_u(x').
    \end{align*}
    Then by the left-monotonicity (see Property 1.), {and the fact that $T_d(x') \leq x'$,} we must have that $T_d(x')\leq T_d(x)$.

    We are left to consider Property 5. By Hobson and Norgilas~{\cite[Theorem 4.9]{HobsonNorgilas:21}}, $S$ is non-decreasing. If $S$ takes the value $\{y\}$ on an interval $(\underline{u},\overline{u}] \subseteq (0,\mu(\R))$, where $\underline u<\overline u$, then
$\nu(\{y\}) \geq \int_{\underline{u}}^{\overline{u}} \frac{\overrightarrow{G}_{\mu}(u) - R(u)}{S(u)-R(u)} du > 0$, but this contradicts our standing assumption, and hence $S$ is strictly increasing on $(0,\mu(\R))$. Then, if $\underline x,\overline x\in(\alpha_\mu,\beta_\mu)$ are such that $\underline{u}:=F_\mu(\underline x)<F_\mu(\overline x)=:\overline{u}$, we immediately have that $T_u(\underline x)=S(\underline u)<S(\overline u)=T_u(\overline x)$. The proof of the strict monotonicity of $T_d$ uses similar arguments, together with part 4. of this lemma.
\end{proof}

Define $\pi^{LC}$ (a measure on $\R^2$) by $\pi^{LC}(dx,dy)=\mu(dx)\pi^{LC}_x(dy)$, where
\begin{equation}\label{eq:LCxy}
\pi^{LC}_x(dy)=\frac{T_u(x) - x}{T_u(x)-T_d(x)} \delta_{T_d(x)}(dy)+ \frac{x-T_d(x)}{T_u(x)-T_d(x)} \delta_{T_u(x)}(dy)\quad\textrm{if }x<T_u(x)
\end{equation}
and $\pi^{LC}_x(dy)=\delta_x(dy)$ otherwise.
\begin{lem}\label{lem:LCxy}
$\pi^{LC}$ {(or rather $(\Gamma_\mu, (\pi^{LC}_x)_{x \in \Gamma_\mu})$ for a support $\Gamma_\mu$ of $\mu$)} defined in \eqref{eq:LCxy} is a martingale coupling of $(\mu,\nu)\in\sK$.
\end{lem}
\begin{proof}
    That {$(\Gamma_\mu, (\pi^{LC}_x)_{x \in \Gamma_\mu})$ defines} a martingale coupling with first marginal $\mu$ is clear from the definition. We now verify that the second marginal is indeed $\nu$.

{Note that
\begin{align*}
I_{\{T_d(x)<T_u(x)\}}&=I_{\{x<T_u(x)\}}+I_{\{T_d(x)<T_u(x)\}}I_{\{\overrightarrow{G}_{\mu}(F_\mu(x))\leq T_u(x)\leq x\leq \overrightarrow{G}_{\mu}(F_\mu(x)+)\}}\\
&=I_{\{x<T_u(x)\}}+I_{\{T_d(x)<T_u(x)\}}I_{\{x=T_u(x)\}},\quad\textrm{for $\mu$-a.e. $x\in\R$}.
\end{align*}
Further, on $\{x:T_d(x)<x=T_u(x)\}$,
$$
 \frac{T_u(x) - x}{T_u(x)-T_d(x)} \delta_{T_d(x)}(dy)+ \frac{x-T_d(x)}{T_u(x)-T_d(x)} \delta_{T_u(x)}(dy)=\delta_x(dy).
$$
It follows that, with the first line an application of \eqref{eq:LCxy},
\begin{eqnarray*}
\lefteqn{\int_{x \in \R} \mu(dx)\pi^{LC}_x(dy)} \\
& = & \int_{x\in\R}\mu(dx) I_{\{x<T_u(x)\}}\left( \frac{T_u(x) - x}{T_u(x)-T_d(x)} \delta_{T_d(x)}(dy)+ \frac{x-T_d(x)}{T_u(x)-T_d(x)} \delta_{T_u(x)}(dy)\right)\\
   && \hspace{5mm} +\int_{x\in\R} {\mu(dx)} I_{\{x=T_u(x)\}}\delta_{x}(dy) \\
  &=&\int_{x\in\R}\mu(dx) I_{\{T_d(x)<T_u(x)\}}\left( \frac{T_u(x) - x}{T_u(x)-T_d(x)} \delta_{T_d(x)}(dy)+ \frac{x-T_d(x)}{T_u(x)-T_d(x)} \delta_{T_u(x)}(dy)\right)\\
   && \hspace{5mm} +\int_{x\in\R} {\mu(dx)}I_{\{T_d(x)=T_u(x)\}}\delta_{x}(dy)
\end{eqnarray*}

    Note that $\mu(\{z:z=\overrightarrow{G}_{\mu}(u),~u\in(0,\mu(\R))\})=\mu(\R)$. Then, using the change of variables $\overrightarrow{G}_{\mu}(u)=x$ (so that, by the continuity of $F_\mu$, $u=F_\mu(\overrightarrow{G}_{\mu}(u))=F_\mu(x)$ and $du=\mu(dx)$) and Theorem \ref{thm:disintegration}, we have that
\begin{eqnarray*}
\lefteqn{\int_{x \in \R} \mu(dx)\pi^{LC}_x(dy)} \\ 
   &=&\int_0^{\mu(\R)} duI_{\{R(u)<S(u)\}}\left( \frac{S(u) - \overrightarrow{G}_{\mu}(u)}{S(u)-R(u)} \delta_{R(u)}(dy)+ \frac{\overrightarrow{G}_\mu(u)-R(u)}{S(u)-R(u)} \delta_{S(u)}(dy)\right)\\
   && \hspace{5mm}+\int_0^{\mu(\R)}duI_{\{R(u)=S(u)\}}\delta_{\overrightarrow{G}_\mu(u)}(dy)\\   
& = & \int_0^{\mu(\R)}du\hat\pi^{LC}_{u,\overrightarrow{G}_{\mu}(u)}(dy) = \int_0^{\mu(\R)}\int_{x\in\R}\hat{\pi}^{LC}(du,dx,dy)= \nu(dy) .  
\end{eqnarray*}
}
\end{proof}
\begin{rem}\label{rem:piLC}
    The left-curtain martingale coupling of $\mu$ and $\nu$ (see Beiglb\"{o}ck and Juillet \cite{BeiglbockJuillet:16}) is uniquely identified by the monotonicity of its support. Since $\pi^{LC}$ is supported on the graphs of $T_d$ and $T_u$ (which are left-monotone) one could show that $\pi^{LC}$ is indeed the left-curtain coupling of Beiglb\"{o}ck and Juillet \cite{BeiglbockJuillet:16} (which explains our choice of notation). All we need in what follows, however, is that $\pi^{LC}$ is a martingale coupling of $(\mu,\nu)$.
\end{rem}

Now we want to partition $(\alpha_\mu,\beta_\mu)$ into disjoint intervals. Let
\begin{equation}\label{eq:A<}
A_<=\{z\in(\alpha_\mu,\beta_\mu):\overrightarrow{G}_\mu(F_\mu(z)+)< { T_u} (z)\}.
\end{equation}
Note that, since we assume that $\mu\neq\nu$, we must have that $\mu(A_<)>0$. Indeed, if $\mu(A_<)=0$, then $\mu(\{z\in(\alpha_\mu,\beta_\mu):\overrightarrow{G}_\mu(F_\mu(z))\leq T_u(z)\leq\overrightarrow{G}_\mu(F_\mu(z)+))\})=1$, and therefore $T_u(z)=z$ for $\mu$-a.e. $z\in\R$, from which we conclude that $\pi^{LC}(dx,dy)=\mu(dx)\delta_x(dy)$. But then $\mu=\nu$, which contradicts our assumption.
\begin{lem}\label{lem:A<}
$A_<$ defined in \eqref{eq:A<} is a countable union of disjoint open intervals $A_<=\bigcup_{k\geq1}\tilde A^k_<$, where $\mu(\tilde A_<^k)>0$ for all $k\geq 1$. 

Furthermore, let $\tilde A_<^k:=(d_k,u_k)$ for each $k\geq1$. Then, for $k\geq1$, $d_k=\overrightarrow{G}_\mu(F_\mu(d_k)+)$ and $u_k=\overrightarrow{G}_\mu(F_\mu(u_k))$.
\end{lem}
\begin{rem}\label{rem:A<}
Since, for each $k\geq 1$, $d_k=\overrightarrow{G}_\mu(F_\mu(d_k)+)$ and $u_k=\overrightarrow{G}_\mu(F_\mu(u_k))$, it follows that $\mu((d_k,d_k+\epsilon))\wedge\mu((u_k-\epsilon,u_k))>0$ for all sufficiently small $\epsilon>0$.

\end{rem}
\begin{proof}[Proof of Lemma \ref{lem:A<}]
We show that $A_<$ is open. Let $x\in A_<$, so that $\overrightarrow{G}_\mu(F_\mu(x)+)<T_u(x)$.
    
    First, by the right continuity of $z\mapsto\overrightarrow{G}_\mu(F_\mu(z)+)$ and the monotonicity of $T_u$, we have that there exists $x_+\in(x,\beta_\mu)$ with
    $$
    \overrightarrow{G}_\mu(F_\mu(x)+)<\overrightarrow{G}_\mu(F_\mu(x_+)+)<T_u(x)\leq T_u(x_+),
    $$
    so that $x_+\in A_<$ and $F_\mu(x)<F_\mu(x_+)$. But then for all $z\in(x,x_+)$, $\overrightarrow{G}_\mu(F_\mu(z)+)\leq\overrightarrow{G}_\mu(F_\mu(x_+)+)<T_u(x)\leq T_u(z)\leq T_u(x_+)$, and thus $[x,x_+]\subset A_<$ and $\mu([x,x_+])>0$.

    Similarly, for all $z\in[\overrightarrow{G}_\mu(F_\mu(x)),x]$ we have that $F_\mu(z)=F_\mu(x)$, and therefore $$\overrightarrow{G}_\mu(F_\mu(z)+)=\overrightarrow{G}_\mu(F_\mu(x)+)<T_u(z)=S(F_\mu(z))=S(F_\mu(x))=T_u(x),$$ so that $[\overrightarrow{G}_\mu(F_\mu(x)),x]\subset A_<$ (however, $\mu([\overrightarrow{G}_\mu(F_\mu(x)),x])=0$). Then by the left-continuity of $T_u$, there exists $x_-<\overrightarrow{G}_\mu(F_\mu(x))\leq x$ with
    $$
    \overrightarrow{G}_\mu(F_\mu(x_-)+) {<}\overrightarrow{G}_\mu(F_\mu(x)+)<T_u(x_-)\leq T_u(x),
    $$
    so that $x_-\in A_<$, and necessarily $\mu([x_-,x])>0$. But then, for all $z\in(x_-,\overrightarrow{G}_\mu(F_\mu(x)))$,
    $$
    \overrightarrow{G}_\mu(F_\mu(z)+)\leq\overrightarrow{G}_\mu(F_\mu(x)+)<T_u(x_-)\leq T_u(z),
    $$
    and thus $[x_-,x]\subset A_<$ and $\mu([x_-,x])>0$.

    We conclude that $A_<$ is open, and thus a union of (at most) countably many disjoint open intervals $\tilde A_<^k$, $k\geq 1$.

    Finally, fix $k\geq1$ and consider $\tilde A^k_<=(d_k,u_k)$. Since $u_k\notin A_<$, we have that $T_u(u_k)\leq \overrightarrow{G}_\mu(F_\mu(u_k)+)$. Now suppose that $u_k\neq \overrightarrow{G}_\mu(F_\mu(u_k))$, so that $\overrightarrow{G}_\mu(F_\mu(u_k))<u_k\leq \overrightarrow{G}_\mu(F_\mu(u_k)+)$. 
    {Then $z=\overrightarrow{G}_\mu(F_\mu(u_k))$ is such that $F_\mu(z)=F_\mu(u_k)$. But then, for all $\tilde{z} \in (d_k \vee z,u_k)$, $\tilde{z} \in \tilde A^k_<$ and
    $$
    \overrightarrow{G}_\mu(F_\mu(u_k)+)=\overrightarrow{G}_\mu(F_\mu(\tilde{z})+)<T_u(\tilde{z})=S(F_\mu(\tilde{z}))=S(F_\mu(u_k))=T_u(u_k),
    $$
    a contradiction. Hence, $u_k = \overrightarrow{G}_\mu(F_\mu(u_k))$.}
    
    Symmetric arguments show that $d_k=\overrightarrow{G}_\mu(F_\mu(d_k)+)$, which concludes the proof.
\end{proof}

Fix $k\geq 1$. Then both $T_d$ and $T_u$ are monotonic on $\tilde A^k_<$ (recall Lemma \ref{lem:propertiesTdTu}). Define $\beta_k =\lim_{x \uparrow u_k} T_u(x) =T_u(u_k)$ and $\alpha_k = \lim_{x \uparrow u_k} T_d(x)$. Let $A_k = (t_k:=\overrightarrow{G}_\mu(F_\mu(\alpha_k)+), u_k )$.

\begin{lem}\label{lem:alphaK}
Fix $k\geq1$. Then, for all $z\in[\overrightarrow{G}_\mu(F_\mu(\alpha_k)),\overrightarrow{G}_\mu(F_\mu(\alpha_k)+)]$, we have that
$$
\overrightarrow{G}_\mu(F_\mu(\alpha_k))\leq T_u(z)\leq\alpha_k\leq \overrightarrow{G}_\mu(F_\mu(\alpha_k)+)=t_k.
$$
\end{lem}
\begin{proof}
    Since, for all $z,z'\in[\overrightarrow{G}_\mu(F_\mu(\alpha_k)),\overrightarrow{G}_\mu(F_\mu(\alpha_k)+)]$, $F_\mu(z)=F_\mu(z')$, it is enough to show that 
\begin{equation}
\label{eq:lem88}
\overrightarrow{G}_\mu(F_\mu(\alpha_k))\leq T_u(\overrightarrow{G}_\mu(F_\mu(\alpha_k)))\leq\alpha_k\leq \overrightarrow{G}_\mu(F_\mu(\alpha_k)+).
\end{equation}
Since the first and the last inequalities hold due to the definitions of ${ T_u}$ and $\overrightarrow{G}_\mu\circ F_\mu$, the  
{case we wish to rule out is 
$
\alpha_k<T_u(\overrightarrow{G}_\mu(F_\mu(\alpha_k)))
$.

Suppose $T_u(\overrightarrow{G}_\mu(F_\mu(\alpha_k)))> \alpha_k$.
By setting $x=\overrightarrow{G}_\mu(F_\mu(\alpha_k))$ (and noting that, for $v\in(0,\mu(\R))$, $F_\mu(\overrightarrow{G}_\mu(v))=v$ due to the continuity of $F_\mu$), we have that $\overrightarrow{G}_\mu(F_\mu(x))=x\leq\alpha_k<T_u(x)$. Then, by Lemma \ref{lem:propertiesTdTu} (see property 2.) we further have that $T_d(x)<x\leq\alpha_k<T_u(x)$. Then, since $\alpha_k=\lim_{z\uparrow u_k}T_d(z)$, by taking a large enough $x'\in(x=\overrightarrow{G}_\mu(F_\mu(\alpha_k)),u_k)$ and using the monotonicity of $T_d$ on $A^k_<$, we obtain $\alpha_k\leq T_d(x')<T_u(x)$ and therefore $T_d(x')\in(T_d(x),T_u(x))$, contradicting the left-monotonicity of $(T_d,T_u)$. Hence $T_u(\overrightarrow{G}_\mu(F_\mu(\alpha_k))) \leq \alpha_k$. Hence \eqref{eq:lem88} holds.}
\end{proof}

Note that $A_k=(t_k,d_k]\cup\tilde{A}^k_<$ and (by the left monotonicity) $\inf_{u\in\tilde{A}_k}R(u)=\inf_{u\in A_k}R(u)=\alpha_k$.


\begin{lem}\label{lem:empty}
Given $k \neq k'$ either $A_k \subsetneq (t_{k'},d_{k'}]\subsetneq A_{k'}$ or $A_{k'} \subsetneq (t_k,d_k]\subsetneq A_{k}$ or $A_{k} \cap A_{k'} = \emptyset$.
\end{lem}
\begin{proof}
Consider $A_k,A_{k'}$ for some $k \neq k'$. Note that we cannot have $A_k = A_{k'}$.  Without loss of generality we may assume $u_{k'} < u_{k}$. Then, since $\tilde A^k_<\cap \tilde A^{k'}_<=\emptyset$, we further have that $u_{k'}\leq d_k$. If $u_{k'}\leq t_k$ then $A_{k} \cap A_{k'} = \emptyset$. If $t_k\leq t_{k'}$ then $A_{k'}\subsetneq(t_k,d_k] \subsetneq A_{k}$. 

Finally, we show that the case $t_{k'}<t_k<u_{k'}$ 
cannot happen. {Suppose $t_{k'}<t_k<u_{k'}$ and note $u_{k'}<d_k$.} Since $\overrightarrow{G}_{\mu}(\cdot+)$ is strictly increasing, while $F_\mu$ is non-decreasing, $t_{k'}<t_k$ implies that $\alpha_{k'}<\alpha_k$. Then, since $\alpha_k\leq t_k$ (recall Lemma \ref{lem:alphaK}), we further have that
$$
\alpha_{k'}<\alpha_k\leq t_k\leq \overrightarrow{G}_{\mu}(F_\mu(t_k)+)<\overrightarrow{G}_{\mu}(F_\mu(z)+)<T_u(z),\quad\textrm{for all }z\in(t_k\vee d_{k'},u_{k'}).
$$
It follows that $\alpha_{k'}<\alpha_k<\lim_{z\uparrow u_{k'}}T_u(z)=T_u(u_{k'})$. On the other hand, by Lemma \ref{lem:propertiesTdTu}, $T_d(u_{k'})\leq \alpha_{k'}$, and therefore $T_d(u_{k'})<\alpha_k<T_u(u_{k'})$. But then we can pick a large enough $x\in(d_k,u_k)=\tilde A^k_<$ satisfying $\alpha_k\leq T_d(x)<T_u(u_{k'})$. It follows that $u_{k'}<x$ and $T_d(x)\in(T_d(u_{k'}),T_u(u_{k'}))$, contradicting the left-monotonicity of $(T_d,T_u)$. {Hence
$t_{k'}<t_k<u_{k'}$ cannot occur. }
\end{proof}

In the following definition the $LC$ in the subscript refers to `left-curtain coupling' and the $S$ to `simple' in the sense that the set $\{x\in(\alpha_\mu,\beta_\mu): \overrightarrow{G}_\mu(F_\mu(x)+)<T_u(x) \}$ takes the form of a single interval.

\begin{defn} \label{def:KSLC}
$(\mu,\nu) \in \sK_{SLC}$ if $(\mu,\nu) \in \sK$ and, in the construction of the left-curtain martingale coupling {$\pi^{LC}$} of $\mu$ and $\nu$, 
$\exists x_0 \in [\alpha_\mu,\beta_\mu)$ such that $A_<=(x_0,\beta_\mu)$, or equivalently, $\overrightarrow{G}_\mu(F_\mu(x))\leq T_u(x)\leq \overrightarrow{G}_\mu(F_\mu(x)+)$  on $(\alpha_\mu,x_0]$ (which is void if $x_0=\alpha_\mu$) and $T_u(x) > \overrightarrow{G}_\mu(F_\mu(x)+)$ on $(x_0,\beta_\mu)$, whence also $T_d$ is decreasing on $(x_0,\beta_\mu)$.
\end{defn}
\begin{rem}\label{rem:x0}
If $(\mu,\nu)\in\sK_{SLC}$ and $x_0$ as in Definition \ref{def:KSLC}, then by Remark \ref{rem:A<} we have that $x_0=\overrightarrow{G}_\mu(F_\mu(x_0)+)$.
\end{rem}

Define $\tilde{J}_k = A_k \setminus \left( \cup_{k' \neq k : A_{k'} \subsetneq A_{k}} A_{k'} \right)$ 
and set
\begin{equation}\label{eq:pi^k}
    \pi_k:=\pi^{LC}\lvert_{\tilde J_k\times \R},\quad k\geq 1.
\end{equation}
Let $\mu_k$ and $\nu_k$ be the first and second marginals of $\pi^k$, respectively. Then $\mu_k=\mu\lvert_{\tilde J_k}$, $\nu_k=\int_{x\in\tilde J_k}\pi^{LC}(dx,dy)$ and $(\alpha_{\nu_k},\beta_{\nu_k})=(\alpha_k,\beta_k)$. 
{{Then $\supp_I(\mu_k) \subseteq \tilde{J}_k$.} }Furthermore, since $(d_k,u_k)=\tilde A^k_<\subseteq \tilde J_k$ and $\bigcup_{k\geq1}\tilde A^k_<=A_<$, we have that $(\alpha_\mu,\beta_\mu)\setminus\bigcup_{k\geq 1}\tilde J_k\subseteq (\alpha_\mu,\beta_\mu)\setminus A_<$, and therefore
{
$$
\pi_0(dx,dy):=\pi^{LC}(dx,dy)-\sum_{k\geq 1}\pi_k(dx,dy)
$$
is such that $\pi_0(dx,dy) =I_{\{x\notin \bigcup_{k\geq 1}\tilde J_k\}}\mu(dx)\delta_x(dy).
$}
Then $\mu_0:=\mu\lvert_{\R\setminus\bigcup_{k\geq 1}\tilde J_k}$ is the first and also the second marginal of $\pi_0$, and we set $\nu_0:=\mu_0$.

Since, for each $k\geq 1$, $\tilde J_k\subseteq A_k$, by Lemma~\ref{lem:empty} we have that $\tilde J_k\cap\tilde J_{k'}=\emptyset$ and {{since $\supp_I(\mu_k) \subseteq \tilde{J}_k,$} }
$\supp_I(\mu_k)\cap\supp_I(\mu_{k'})=\emptyset$ for $k\neq k'$.

Similarly, by Lemma~\ref{lem:empty}, and using that $\pi^{LC}$ is supported by left-monotone maps $(T_d,T_u)$, we have that $\supp_I(\nu_k)\cap\supp_I(\nu_{k'})=\emptyset$ for all $k\neq k'$.

\begin{prop}
\label{prop:intervalsSLC}
Suppose $\mu \leq_{cx} \nu$ and both measures are atom-less. Let $\pi^{LC}$ be 
as in \eqref{eq:LCxy}. Then there exists a partition of $(\alpha_\mu,\beta_\mu)$ into countably many disjoint sets $(\tilde J_k)_{k \geq 1}$ such that if $\mu_k = \mu^{\tilde J_k}$ and $\nu_k = \int_{x \in \tilde J_k} \pi^{LC}(dx,dy)$, then $(\mu-\sum_{k\geq1}\mu_k)= (\nu-\sum_{k\geq1}\nu_k)$ and $\mu_k \leq_{cx} \nu_k$. Moreover, this partition can be chosen such that
for each $k \geq 1$ we have 
$(\mu_k, \nu_k) \in \sK_{SLC}$ and such that $\nu_k \wedge \nu_{k'} = 0$ for $k \neq k'$.
\end{prop}

\begin{proof}
We are left to argue that, for each $k\geq1$, $(\mu_k, \nu_k) \in \sK_{SLC}$. But this is immediate, since $$\nu_k=\int_{\tilde J_k}\pi^{LC}(dx,dy)=\int_{(t_k,d_k]\cap\tilde J_k}\mu(dx)\delta_x(dy)+\int_{\tilde A^k_<}\mu(dx)\pi^{LC}_x(dy)
$$
and, on $\tilde A^k_<$, $\pi^{LC}_x$ is supported on $\{T_d(x),T_u(x)\}$. In particular, if $x^k_0$ is the bifurcation point for $(\mu_k,\nu_k)$ as in the Definition \ref{def:KSLC}, then $x^k_0=d_k$.


\end{proof}

{{It follows from Proposition \ref{prop:intervalsSLC} that $\pi^{LC}$ can be written as
\[ \pi^{LC}(dx,dy)=\sum_{k\geq 1}I_{\{x\in\tilde J_k\}}\mu(dx)\pi^{LC}_x(dy)+I_{\{x\notin\bigcup_{k\geq1}\tilde J_k\}}\mu(dx)\delta_{x}(dy).
\] 
The main idea we use to construct a strongly injective martingale coupling of $\mu$ and $\nu$ is to replace the coupling of $\mu_k$ and $\nu_k$ for each $k \geq 1$ with a strongly injective martingale coupling of the same pair of measures. However, in order to preserve the injectivity property for the global construction we need to be careful over the support of $\mu_k$ we use for the coupling of $\mu_k$ and $\nu_k$ and also the support of $\mu_0$.
We choose the support $\tilde{O}_k$ of $\mu_k$ later (to coincide we the supports we constructed in previous sections). Next we focus on carefully choosing the support of the measure on the diagonal where $\mu_0 \equiv \nu_0$.}}

For each $k\geq 1$, let $\tilde{O}_k$ be a support of $\mu_k$ with $\tilde O_k\subseteq\supp_I(\mu_k)\subseteq \tilde J_k$ (each $\tilde O_k$ is assumed to be Borel). Then $\mu_k(\tilde O_k)=\mu_k(\R)$. Also $\mu_0(\cup_{k\geq 1}\tilde O_k)=0$ and $\mu(\cup_{k\geq 1} \tilde J_k\setminus \tilde O_k)=0$.  It follows that $\pi_k=\pi^{LC}\lvert_{\tilde O_k\times(\alpha_k,\beta_k)}$, and we can write
$$
\pi^{LC,2}(dx,dy)=\sum_{k\geq 1}I_{\{x\in\tilde O_k\}}\mu(dx)\pi^{LC}_x(dy)+I_{\{x\notin\bigcup_{k\geq1}\tilde J_k\}}\mu(dx)\delta_{x}(dy)$$
is a martingale coupling of $\mu$ and $\nu$.

Since $\nu(\{x:x\notin\supp_I(\nu)\cup(\cup_{k\geq 1}\supp_I(\nu_k))\})=0$ and $\pi^{LC}_x=\delta_x$ for all $x\notin\bigcup_{k\geq1}\tilde J_k$, we have that
$$
I_{\{x\notin\bigcup_{k\geq1}\tilde J_k\}} I_{\{x\notin\supp_I(\nu)\cup(\cup_{k\geq 1}\supp_I(\nu_k))\}}\mu(dx)\delta_{x}(dy)
$$
is a zero measure, and therefore
$$
I_{\{x\notin\bigcup_{k\geq1}\tilde J_k\}}\mu(dx)\delta_{x}(dy)=I_{\{x\notin\bigcup_{k\geq1}\tilde J_k\}}\left(I_{\{x\in\supp_I(\nu)\setminus\cup_{k\geq1}\supp_I(\nu_k)\}}+I_{\{x\in\cup_{k\geq1}\supp_I(\nu_k)\}}\right)\mu(dx)\delta_{x}(dy).
$$
\begin{lem}\label{lem:aux1}
{We have $\{x\notin\bigcup_{k\geq1}\tilde J_k\} \cap \{x\in\cup_{k\geq1}\supp_I(\nu_k)\}=\emptyset$,} so that
\begin{align*}
I_{\{x\notin\bigcup_{k\geq1}\tilde J_k\}}\mu(dx)\delta_{x}(dy)&=I_{\{x\notin\bigcup_{k\geq1}\tilde J_k\}}I_{\{x\in\supp_I(\nu)\setminus\cup_{k\geq1}\supp_I(\nu_k)\}}\mu(dx)\delta_{x}(dy)\\
&=I_{\{x\notin\bigcup_{k\geq1}\tilde O_k\}}I_{\{x\in\supp_I(\nu)\setminus\cup_{k\geq1}\supp_I(\nu_k)\}}\mu(dx)\delta_{x}(dy).
\end{align*}
\end{lem}
\begin{proof}
Recall that $\mu(\cup_{k\geq 1}\tilde J_k\setminus\tilde O_k)=0$, and thus we immediately have that
\begin{align*}
I_{\{x\notin\bigcup_{k\geq1}\tilde J_k\}}I_{\{x\in\supp_I(\nu)\setminus\cup_{k\geq1}\supp_I(\nu_k)\}}\mu(dx)\delta_{x}(dy)
=I_{\{x\notin\bigcup_{k\geq1}\tilde O_k\}}I_{\{x\in\supp_I(\nu)\setminus\cup_{k\geq1}\supp_I(\nu_k)\}}\mu(dx)\delta_{x}(dy).
\end{align*}
We are left to prove the first assertion.

Suppose $x\in\bigcup_{k\geq 1}\supp_I(\nu_k)$, so that (since the supports $\supp_I(\nu_k)$, $k\geq1$, are disjoint) $x\in\supp_I(\nu_k)\subseteq(\alpha_k,\beta_k)=A_k$ for some unique $k\geq 1$. Now, in addition, suppose that $x\notin\bigcup_{j\geq 1}\tilde J_j$. Then $x\notin\tilde J_k\subseteq A_k$ and we must have that $x\in(t_k,d_k]$. Since $x\in\supp_I(\nu_k)$, we have that either
$$
x\in\supp(\mu_k)\cap (t_k,d_k]\subseteq \tilde J_k\cap(t_k,d_k],
$$
or
$$
x\in\{T_d(x'\pm),T_u(x'\pm)\}\quad\textrm{for some }x'\in\tilde A^k_<=(d_k,u_k).
$$
Since {{$x\notin \tilde J_k$}}, only the second case is feasible. But $x\in(t_k,d_k]$, and therefore
$$
x\in\{T_d(x'-),T_d(x'+)\}\quad\textrm{for some }x'\in\tilde A^k_<=(d_k,u_k).
$$

Since $x\notin\tilde {{J_k}}$ but $x\in(t_k,d_k]$, we must have that $x\in\bigcup_{1\leq l\neq k,~A_l\subseteq(t_k,d_k]} A_l$. Let $1\leq l\neq k$ be such that $x\in A_l=(t_l,u_l) \subseteq(t_k,d_k]$. Then $x\in(\alpha_l,\beta_l)$, and by taking large enough $\tilde x\in(d_l,u_l)$ we have that $x\in(T_d(\tilde x),T_u(\tilde x))\subset(\alpha_l,\beta_l)$. Fix $\epsilon>0$ with
$$
(x-\epsilon,x+\epsilon)\subset(T_d(\tilde x),T_u(\tilde x))\subset(\alpha_l,\beta_l).
$$

Since $x\in\{T_d(x'-),T_d(x'+)\}$ for some $x'\in\tilde A^k_<=(d_k,u_k)$, in the case $x=T_d(x'-)$ (resp. $x=T_d(x'+)$) we can find a large enough $\bar x<x'$ (resp. small enough $\bar x>x'$) such that $T_d(\bar x)\in(x,x+\epsilon)$ (resp. $T_d(\bar x)\in(x-\epsilon),x)$). In either case we have that $T_d(\bar x)\in(T_d(\tilde x),T_u(\tilde x))$, which contradicts the left-monotonicity since $\bar x>\tilde x$.
\end{proof}

Clearly,
{{
\begin{eqnarray*}
\{x\in\supp_I(\nu)\setminus\cup_{k\geq1}\supp_I(\nu_k)\}&=&\left( \{x\notin\bigcup_{k\geq1}\tilde O_k\} \cap \{x\in\supp_I(\nu)\setminus\cup_{k\geq1}\supp_I(\nu_k)\} \right)\\
& & \cup \left( \{x\in\bigcup_{k\geq1}\tilde O_k\} \cap \{x\in\supp_I(\nu)\setminus\cup_{k\geq1}\supp_I(\nu_k)\} \right).
\end{eqnarray*}
}}

\begin{lem}\label{lem:aux2}
We have that 
$\left(\supp_I(\nu)\setminus\cup_{k\geq1}\supp_I(\nu_k)\right)\subseteq\left(\R\setminus\bigcup_{k\geq1}\tilde O_k\right)$.
\end{lem}
\begin{proof}
{{  We show that if $x\in\bigcup_{k\geq1}\tilde O_k$ 
then either $x \in \cup_{k\geq1}\supp_I(\nu_k)$ or $x \notin \supp_I(\nu)$.
}}

Suppose that $x\in\bigcup_{k\geq1}\tilde O_k$. Then $x\in\tilde O_k$ for some unique $k\geq1$. If $x\in(t_k,d_k]$, then (since $\pi^{LC}_x=\delta_x$ on $\tilde O_k\subseteq\supp_I(\mu_k)$) we have that $x\in\supp_I(\nu_k)$. 
{{Hence it is sufficient to show that if $x\in\bigcup_{k\geq1}\tilde{A}^k_<$ 
then either $x \in \cup_{k\geq1}\supp_I(\nu_k)$ or $x \notin \supp_I(\nu)$. }}

{{Suppose $x\in\bigcup_{k\geq1}\tilde A^k_<$ so that $x\in\tilde A^k_<=(d_k,u_k)$ for some unique $k\geq 1$.
We show that either $x \in \supp(\nu_k)$ or $x \notin \supp_I(\nu)$. In particular, we show that if $x \notin \supp(\nu_k)$ then there is an interval $(x-\epsilon,x+\epsilon)$ such that $\nu_j((x-\epsilon,x+\epsilon))=0$ for every $j$, including $j=k$, $j=0$ and $j \notin \{0,k\}$. Then $\nu((x-\epsilon,x+\epsilon))=0$ and $x \notin \supp(\nu)$.}}

{{Suppose that $x\in (d_k,u_k) \setminus \supp_I(\nu_k)$. Then there exists $\epsilon>0$ such that $H_\epsilon :=(x-\epsilon,x+\epsilon)\subseteq (d_k,u_k)$ and $\nu_k(H_\epsilon)=0$. Since $H_\epsilon \subseteq \tilde{A}^k_<$ we must have that $\nu_0(H_\epsilon)=0$. It only remains to show that $\nu_j(H_\epsilon)=0$ for all $j \geq 1, j \neq k$, but first we derive a further property of $T_d$ and $T_u$ on $\tilde{A}^k_<$. In particular, we claim that there exists $z\in(x,u_k)$ with 
\begin{equation}
    \label{eq:existsz}
T_d(z) < x-\epsilon<x<x+\epsilon  < T_u(z).
\end{equation}
}}

{{Since $x\notin\supp_I(\nu_k)$ and $d_k<x$ we cannot have $x= T_u(d_k +)$. Suppose $x \in ( T_u(d_k+),u_k)$. Then 
there exists $x'\in(d_k,u_k)$ with
$$T_d(x')<T_u(x')=T_u(x'-)\leq x-\epsilon<x<x+\epsilon\leq T_u(x'+).$$ By the monotonicity of $T_d,T_u$ on $(d_k,u_k)$ we then have that for all $z\in(x',u_k)$ we have
$$
T_d(z) < T_d(x')<T_u(x'+)  < T_u(z)
$$ 
and the claim follows.
Conversely, suppose $d_k < x - \epsilon < x < T_u(d_k+)$. Then, adjusting $\epsilon$ downwards as necessary to ensure that $\epsilon < T_u(d_k+) - x$, we have $T_d(d_k+)<x-\epsilon < x + \epsilon < T_u(d_k+)$. Then, for all $z \in (d_k,u_k)$, $T_d(z)<T_d(d_k+) < x-\epsilon<x+\epsilon < T_u(z)$
and again the claim follows.}}

Finally, we show that for each $1 \leq j \neq k$ we have $\nu_j(H_\epsilon)=0$. Suppose not and choose $j$ such that $\nu_j(H_\epsilon)>0$. 
Then $x\in\supp_I(\nu_{j})$. It follows from Lemma \ref{lem:empty}, that we must have $A_k\subsetneq (t_{j},d_{j}]\subsetneq A_{j}$. But then $(x-\epsilon,x+\epsilon)$ is not a subset of $\tilde J_{j}$, and hence we must have that $x$ is `reached' by $T_d$ on $(d_{j},u_{j})$, i.e., there exists $\tilde x\in(d_{j},u_{j})$ such that $x\in\{T_d(\tilde x-),T_d(\tilde x+)\}$. Then, if $x=T_d(\tilde x-)$ (resp. $x=T_d(\tilde x+)$) we can find a large enough $\hat x<\tilde x$ (resp. small enough $\hat x>\tilde x$) such  that $T_d(\hat x)\in {{(x=T_d(\tilde x-),x+\epsilon)}}$ (resp. $T_d(\hat x)\in{{(x-\epsilon,T_d(\tilde x+)=x)}}$). In either case
$$
T_d(z)<x-\epsilon<T_d(\hat x)<x+\epsilon\leq T_u(z),\quad z\in(x',u_k),
$$
which contradicts the left-monotonicity since $z<u_k<\hat x$. Hence $\nu_j(H_\epsilon)=0$ as required.
\end{proof}

Before proving our final result, we need a result to say that for each $k \geq 1$, if $(\mu_k,\nu_k) \in \sK_{SLC}$ then $(\mu_k,\nu_k) \in \sK_R$ and hence there exists a strongly injective martingale coupling of $\mu_k$ and $\nu_k$. 

\begin{prop}
\label{prop:backforward}
$ \sK_{SLC} \subseteq \sK_R$.
\end{prop}

\begin{proof}
Suppose $(\mu,\nu) \in \sK_{SLC}$. Then $\overrightarrow{G}_{\mu}(F_\mu(x))\leq T_u(x)\leq \overrightarrow{G}_{\mu}(F_\mu(x+))$ for $x\in(\alpha_\mu,x_0]$ and $\sA_< =(x_0,\beta_\mu)$ for some unique $x_0\in(\alpha_\mu,\beta_\mu)$.

Now we argue that $(\mu,\nu)$ satisfy the first numbered property of Definition~\ref{def:KR}. To see this, recall that $\pi^{LC}_x=\delta_x$ for all ($\mu$-a.e.) $x\in(\alpha_\mu,x_0]$, so that $\mu\lvert_{(\alpha_\mu,x_0]}$ is the first and also the second marginal of $\pi^{LC}\lvert_{(\alpha_\mu,x_0]\times \R}$. Hence, for all Borel $A\subseteq\R$,
$$
\mu\lvert_{(\alpha_\mu,x_0]}(A)=\pi^{LC}\lvert_{(\alpha_\mu,x_0]\times \R}(\R\times A)\leq \pi^{LC}\lvert_{\R\times \R}(\R\times A)=\pi^{LC}(\R\times A)=\nu(A)
$$
and therefore $\mu=\mu\lvert_{(\alpha_\mu,x_0]}\leq\nu$ on $(\alpha_\mu,x_0]$.

If $\nu\geq\mu$ on $(-\infty,a)$ for some $a>x_0$ then, since $\overrightarrow{G}_{\mu}(F_\mu(x_0)+)=x_0$, we have that $\nu([{{x_0,a}}))\geq\mu([{{x_0,a}}))>0$ and $D_{\mu,\nu}$ is convex on $(-\infty,a)$. Then we can pick $x'>x_0$ with $\overrightarrow{G}_{\mu}(F_\mu(x_0)+)=x_0< \overrightarrow{G}_{\mu}(F_\mu(x'))\leq x'\leq \overrightarrow{G}_{\mu}(F_\mu(x')+)<a$ and then convexity of $D_{\mu,\nu}$ on $(-\infty,a)$ ensures that $\hat\sE_{F_\mu(x')}$ is convex everywhere, since $\hat\sE_{F_\mu(x')}\geq D_{\mu,\nu}$ and $\overrightarrow{G}_{\mu}(F_\mu(x'))<a$. It follows that $T_u(x')=\overrightarrow{G}_{\mu}(F_\mu(x'))$, contradicting the fact that $T_u>\overrightarrow{G}_{\mu}\circ F_\mu$ on $(x_0,\beta_\mu)$. This proves that our candidate quantity $x_0$ satisfies the first (i.e., maximality) property of Definition \ref{def:KR}.

Now we argue that the middle listed property of Definition~\ref{def:KR} holds, namely that the tangent $L_{x_0}$ to $D_{\mu,\nu}$ at $x_0$ lies above $D_{\mu,\nu}$ on $(x_0,\infty)$.

There are four possible cases: either $L_{x_0} > D_{\mu,\nu}$ on $(x_0,\infty)$; or there exists $c_2>c_1>x_0$ with $D_{\mu,\nu}(c_1)<L_{x_0}(c_1)$ and $D_{\mu,\nu}(c_2) \geq L_{x_0}(c_2)$; or there exists $d_1>x_0$ such that $L_{x_0} \leq D_{\mu,\nu}$ on $(x_0,d_1)$ and $L_{x_0}(d_1) < D_{\mu,\nu}(d_1)$; or there exists $e>x_0$ such that $L_{x_0} = D_{\mu,\nu}$ on $[x_0,e]$. The last case cannot happen else $x_0$ is not maximal (see the first property of Definition~\ref{def:KR}). Therefore it is sufficient to show that the second and third cases also lead to a contradiction.

For the second case, suppose there exists $c_2>c_1>x_0$ with $D_{\mu,\nu}(c_1)<L_{x_0}(c_1)$ and $D_{\mu,\nu}(c_2) \geq L_{x_0}(c_2)$. {Taking $c_2$ smaller if necessary, but still with $c_2>c_1$ and $D_{\mu,\nu}(c_2) \geq L_{x_0}(c_2)$, we may assume that $\hat{\sE}_{F_\mu(c_2)} \geq L_{x_0}$ on $(c_2,\infty)$.} Consider $\hat\sE_{F_\mu(c_2)}$; note that $c_2 < \beta_\mu$ and $\hat\sE_{F_\mu(c_2)}=D_{\mu,\nu}$ on ${(-\infty,\overrightarrow{G}_\mu(F_\mu(c_2)+))}$. Given the existence of $c_1$, there exists $c_3 \in (x_0, c_2)$ such that $\hat\sE^c_{F_\mu(c_2)}(c_3) =  \hat\sE_{F_\mu(c_2)}(c_3)=D_{\mu,\nu}(c_3) < L_{x_0}(c_3)$. Then, using that $\hat \sE_{{F_\mu(c_3)}}=D_{\mu,\nu}$ on $(\alpha_\mu,\overrightarrow{G}_\mu(F_\mu(c_3)+)]$, $\hat \sE_{F_\mu(c_3)}\geq \hat\sE_{F_\mu(c_2)}$ (and thus also ${\hat \sE}_{ F_\mu(c_3)}\geq\hat \sE^c_{ F_\mu(c_3)}\geq \hat\sE^c_{F_\mu(c_2)}$) everywhere, we have that $\hat\sE^c_{F_\mu(c_3)}(c_3)={ \hat\sE}_{F_\mu(c_3)}(c_3)$.  It follows that $T_u(c_3) \leq c_3\leq\overrightarrow{G}_\mu(F_\mu(c_3)+)$. But this contradicts the fact that $A_< = (x_0, \beta_\mu)$.

For the third case, suppose there exists $d_1>x_0$ such that $L_{x_0} \leq D_{\mu,\nu}$ on $(x_0,d_1)$ and $L_{x_0}(d_1) < D_{\mu,\nu}(d_1)$. There exists $d_2 \in (x_0,d_1)$ such that $L_{x_0} \leq D_{\mu,\nu}$ on $(x_0,d_2)$, $L_{x_0}(d_2) < D_{\mu,\nu}(d_2)$ and $D'_{\mu,\nu}(d_2) > L'_{x_0}(d_2)$ (recall that, since $\mu$ and $\nu$ are atom-less, $D_{\mu,\nu}$ is differentiable). Note that $D'_{\mu,\nu}(\overrightarrow{G}_\mu(F_\mu(d_2)+))\geq D'_{\mu,\nu}(d_2)$.  Then $\hat\sE^c_{F_\mu(d_2)} > L_{x_0}$ on $[d_2,\infty)$. It follows that there exists $d_3 \in (x_0,d_2 {]}$ such that $\hat\sE^c_{F_\mu(d_2)}(d_3) = \hat\sE_{F_\mu(d_2)}(d_3)$ {and then $\hat\sE^c_{F_\mu(d_3)}(d_3) = \hat\sE_{F_\mu(d_3)}(d_3)$}. Then using similar arguments as in the second case, we have that $T_u(d_3) \leq \overrightarrow{G}_\mu(F_\mu(d_3)+)$, which again gives us a contradiction.

Now consider the final part of Definition \ref{def:KR}. Note that for $b \in (\inf \{ k : D_{\mu,\nu}(k) > 0 \} ,  \underline{x}_0 {)}=(\alpha_\nu,\underline{x}_0)$ we have that $D_{\mu,\nu}'(b)$ exists and is positive. Let $L_b:=L^{b,D_{\mu,\nu}'(b)}_{D_{\mu,\nu}}$. We must have that $\{k>b:D_{\mu,\nu}(k)>L_b(k)\}\neq\emptyset$.  Define $\overline{c}:=\sup\{k>b:D_{\mu,\nu}(k)>L_b(k)\}$. Continuity of $D_{\mu,\nu}$ (and $L_b$) implies that $D_{\mu,\nu}(\overline c)=L_b(\overline c)$. Moreover, since $D_{\mu,\nu}$ is convex and non-decreasing on $(-\infty,x_0)$, and $\lim_{k\to\infty}D_{\mu,\nu}(k)=0$, we have that $x_0<\overline c<\infty$. There are {three} cases: {either $L_b < D_{\mu,\nu}$ on $[x_0,\overline c)$;
or $L_b \leq D_{\mu,\nu}$ on $[x_0,\overline c)$ and there exists $\tilde{c} \in (x_0,\overline c)$ for which $D_{\mu,\nu}(\tilde c)=L_b(\tilde c)$;}
or $\{k\in(x_0,\overline c):D_{\mu,\nu}(k)<L_b(k)\}\neq\emptyset$.
We show the first case leads to $D_{\mu,\nu}'(b) = L_b' > D_{\mu,\nu}'(\overline c)$, and that the second and third cases cannot happen.

\textit{Case 1: $L_b < D_{\mu,\nu}$ on $[x_0,\overline c)$.} Suppose $D_{\mu,\nu}<L_b$ on $(\overline c,\infty)$ and $D_{\mu,\nu}'(b)=D_{\mu,\nu}'(\overline c)$ (note that $D_{\mu,\nu}'(b)<D_{\mu,\nu}'(\overline c)$ cannot happen). {Since $D_{\mu,\nu}'(b)>0$ we must have $F_\mu(\overline c)< \mu(\R)$ and hence  $\overline c < \beta_\mu$.} Then $D_{\mu,\nu}=\hat\sE_{F_\mu(\overline c)}\geq L_b$ on $(-\infty,\overline{c}]$, with equalities throughout at $\overline{c}$. On the other hand, since $\mu_{\overline c}$ does not charge $[\overline c,\infty)$, $D_{\mu,\nu}'(b)=D_{\mu,\nu}'(\overline c)=\hat\sE'_{F_\mu(\overline c)}(\overline c)\leq \hat\sE'_{F_\mu(\overline c)}(z)$ for all $z>\overline c$, and therefore $\hat\sE_{F_\mu(\overline c)}\geq L_b$ everywhere. It follows that $\hat\sE^c_{F_\mu(\overline c)}(\overline c)=\hat\sE_{F_\mu(\overline c)}(\overline c)$. Then $T_u(\overline c)\leq \overrightarrow{G}_{\mu}(F_\mu(\overline{c})+)$,  a contradiction (since $\overline c\in (x_0,\beta_\mu)=A_<$). Hence $D_{\mu,\nu}'(b) = L_b' > D_{\mu,\nu}'(\overline c)$.

\textit{Case 2: $L_b \leq D_{\mu,\nu}$ on $[x_0,\overline c)$ and there exists $\tilde c\in(x_0,\overline c)$ with $\hat\sE_{F_\mu(\tilde c)}=D_{\mu,\nu}(\tilde c)=L_b(\tilde c)$.} In this case we have that $D_{\mu,\nu}\leq\hat\sE_{F_\mu(\tilde c)}$ on $[\tilde c,\overline c]$, and since $\hat\sE_{F_\mu(\tilde c)}$ is convex to the right of $\tilde c$, we obtain that $\hat\sE_{F_\mu(\tilde c)} {\geq} L_b$ everywhere. Then $\hat\sE'_{F_\mu(\tilde c)}(\tilde c)=\hat\sE_{F_\mu(\tilde c)}(\tilde c)$ and $T_u(\tilde c)\leq \overrightarrow{G}_{\mu}(F_\mu(\tilde{c})+)$, a contradiction.

\textit{Case 3: $\{k\in(x_0,\overline c):D_{\mu,\nu}(k)<L_b(k)\}\neq\emptyset$}. In this case, since $D_{\mu,\nu}>0$ on $(\alpha_\nu,\beta_\mu)$, we can find $b'<b$ (with $D'_{\mu,\nu}(b')>0$) for which there exist $c_1,c_2\in\R$ with $x_0<c_1<\overline c<c_2$ and such that $L_{b'}\leq D_{\mu,\nu}$ on $(-\infty,c_2]$ and  $L_{b'}(k)=D_{\mu,\nu}(k)$ for $k\in\{c_1,c_2\}$. Then, similarly as in $\textit{Case 2}$, $T_u(c_1)\leq \overrightarrow{G}_{\mu}(F_\mu(c_1)+)$, a contradiction.
\end{proof}

\begin{thm}
\label{thm:maincts}
Suppose $(\mu,\nu) \in \sK$. Then there exists a strongly injective martingale coupling of $\mu$ and $\nu$ on its irreducible component.

\end{thm}

\begin{proof}
Combining Lemmas \ref{lem:aux1} and \ref{lem:aux2}, we obtain that, for $\tilde{O}_k \subseteq \supp_I(\mu_k)$ a support of $\mu_k$,
$$
\pi^{LC}(dx,dy)=\sum_{k\geq 1}I_{\{x\in\tilde O_k\}}\mu(dx)\pi^{LC}_x(dy)+I_{\{x\in\supp_I(\nu)\setminus\cup_{k\geq1}\supp_I(\nu_k)\}}\mu(dx)\delta_{x}(dy)$$
where for each $k \geq 1$, $(\mu_k,\nu_k) \in \sK_{SLC}$.

For each $k\geq1$, $\pi^{k,LC}(dx,dy)=I_{\{x\in\tilde O_k\}}\mu(dx)\pi^{LC}_x(dy)$ is a martingale coupling of $(\mu_k,\nu_k)$. By Proposition~\ref{prop:backforward}, $(\mu_k,\nu_k) \in \sK_R$, and then by Theorem~\ref{thm:KRexist} we can choose $\tilde O_k = \Gamma_k\subseteq \supp_I(\mu_k)$ and a family of probability measures $\{ \pi^k_x:\int_\R\pi^k _x(dy)=x,~x\in \Gamma_k\}$ such that $\pi^k(dx,dy):=I_{\{x\in\Gamma_k\}}\mu(dx){\pi^{k}_x}(dy)$ is a strongly injective martingale coupling of $(\mu_k,\nu_k)$. 

Let $\Gamma_0= \supp_I(\nu)\setminus\cup_{k\geq1}\supp_I(\nu_k)$ and $\Gamma = \cup_{k \geq 0} \Gamma_k$. For $x \in \Gamma$ define $\pi_x$ by $\pi_x = {\pi^k_x}$ if $x \in \Gamma_k$ {for some $k\geq1$,}  and $\pi_x = \delta_x$ otherwise. 
It then follows that $(\Gamma,(\pi_x)_{x \in \Gamma})$
defines a strongly injective martingale coupling of $\mu$ and $\nu$, which is given by
$$
\pi(dx,dy):=\sum_{k\geq1}\pi^k(dx,dy)+I_{\{x\in\supp_I(\nu)\setminus\cup_{k\geq1}\supp_I(\nu_k)\}}\mu(dx)\delta_{x}(dy).
$$
Moreover, for each $x \in \Gamma$, $\pi_x$ has finite support.

To confirm this claim it is sufficient to show that 
the sets $(\Gamma_j)_{j \geq 0}$ are disjoint, and then 
that we have the strong injectivity property, namely
for $x \in \Gamma$, $\supp(\pi_x) \subseteq \supp(\nu)$; for $x ,x' \in \Gamma$ with $x\neq x'$, $\supp(\pi_x) \cap \supp(\pi_{x'}) = \emptyset$ and, for all $y \in \supp(\nu)$, $y \in \supp(\pi_x)$ for some $x\in\Gamma$. Given Lemma~\ref{lem:aux2}, each of these facts is straightforward to show.
\end{proof}
Finally we can prove the main result of the paper.
\begin{proof}[Proof of Theorem \ref{thm:mainintro}]
By Propositions \ref{prop:reduction} and \ref{prop:irreducible} we can assume that $\mu$ and $\nu$ are both continuous, and that $(\mu,\nu)$ has a single irreducible component. In particular, $(\mu,\nu)\in\sK$; recall Definition \ref{def:K}. Then the result immediately follows from Theorem~\ref{thm:maincts}.
\end{proof}

\bibliographystyle{plainnat}

\end{document}